\documentclass[11pt]{amsart}

\usepackage[margin=1.25in]{geometry}
\usepackage[T1]{fontenc}
\usepackage{lmodern} 
\usepackage{microtype}

\usepackage{amsmath,amssymb,amsthm,mathtools}
\numberwithin{equation}{section}

\usepackage{graphicx}
\usepackage{tikz}
\usetikzlibrary{arrows.meta,positioning,calc}

\usepackage{array,longtable}
\usepackage{enumitem}

\usepackage[hidelinks]{hyperref}
\hypersetup{
pdfauthor={Yin Cai, Xiang Fang, Hongdou Qu},
pdftitle={Exact Fourier dimensions of dyadic Mandelbrot cascades on curves of
nonvanishing curvature under minimal integrability},
pdfkeywords={dyadic Mandelbrot cascades, Fourier dimension, nondegenerate \(C^2\) arcs, \(C^2\) Jordan curves, minimum lower local dimension}
}

\newtheorem{theorem}{Theorem}[section]
\newtheorem{proposition}[theorem]{Proposition}
\newtheorem{lemma}[theorem]{Lemma}
\newtheorem{corollary}[theorem]{Corollary}

\theoremstyle{definition}
\newtheorem{definition}[theorem]{Definition}

\theoremstyle{remark}
\newtheorem{remark}[theorem]{Remark}

\newcommand{\weak}{\xrightarrow{\mathrm{w}}}

\title[Exact Fourier Dimensions on curves]
{Exact Fourier dimensions of dyadic Mandelbrot cascades on curves of
nonvanishing curvature under minimal integrability}

\author[Y. Cai]{Yin Cai}
\address{School of Mathematics\\
Hangzhou Normal University\\
Hangzhou 311121\\
P. R. China}
\email{cymath@hznu.edu.cn}

\author[X. Fang]{Xiang Fang}
\address{Department of Applied Mathematics\\
National Yang Ming Chiao Tung University\\
Hsinchu 30010\\
Taiwan}
\email{xfang@nycu.edu.tw}

\author[H. Qu]{Hongdou Qu}
\address{School of Mathematical Sciences\\
Dalian University of Technology\\
Dalian 116024\\
P. R. China}
\email{quhongdou0926@mail.dlut.edu.cn}

\date{\today}



\begin{document}

\begin{abstract}
We prove exact Fourier-dimension formulas for scalar dyadic Mandelbrot cascades
pushed forward by fixed nondegenerate \(C^2\) parametrizations of embedded
arcs and Jordan curves in \(\mathbb R^2\). Let \(W\) be in the minimal Kahane--Peyri\`ere regime. For each fixed nondegenerate \(C^2\) embedded arc \(\gamma:[0,1]\to\mathbb R^2\), the pushforward \(\mu_\gamma\) of the interval cascade satisfies, almost surely on non-extinction,
\[
    \dim_{\mathrm F}(\mu_\gamma)=A_{\mathrm{loc}}(W),
\]
where
\[
    A_{\mathrm{loc}}(W)
    =
    \sup_{q>1}
    \max\left\{
        0,\,
        \frac{q-1-\log_2\mathbb E[W^q]}{q}
    \right\},
\]
with the \(q\)-term interpreted as \(0\) when \(\mathbb E[W^q]=\infty\). The analogous formula holds for scalar circle cascades pushed forward by
fixed nondegenerate \(C^2\) Jordan curves
\(\gamma:\mathbb T\to\mathbb R^2\), with the pushforward denoted by
\(\mu_\gamma^{\mathbb T}\).

This extends the scalar circle endpoint formula of \cite{2026-M}
from the canonical circle to arbitrary fixed nondegenerate \(C^2\)
embedded arcs and Jordan curves. The main new
issue beyond the canonical circle is the loss of the explicit trigonometric
phase and, for arcs, the presence of endpoint stationary regimes. We prove the
arc lower bound by a finite-\(r\) annular Fourier theorem based on an
endpoint-safe phase decomposition, phase-bin coefficient estimates, predictable
capping, complex Freedman concentration, and an \(r\)-tail compensator. The
Jordan lower bound follows by first-generation dyadic cutting into two fixed
arcs. The matching upper bounds use deterministic curved-support obstructions
together with the scalar-circle minimum lower local-dimension theorem.
\end{abstract}

\maketitle

\tableofcontents

\section{Introduction and main results}
\label{S:introduction}

\subsection{Motivation and main contributions}
\label{SS:introduction-background-main-contribution}

Mandelbrot cascades are fundamental examples of random multifractal measures. In the dyadic model, mass is redistributed along the binary tree by independent copies of a nonnegative weight. The limiting martingale measures arise naturally in multiplicative chaos, branching random walks, and multifractal analysis \cite{Kahane1985,Kahane1985Chaos,KahanePeyriere1976, Mandelbrot1974a,Mandelbrot1976,RhodesVargas2014}. A classical problem, going back to Mandelbrot and Kahane, asks when such random measures have polynomial Fourier decay, and how the almost sure Fourier dimension is determined by the distribution of the cascade weight. Recent progress on this Mandelbrot--Kahane Fourier-dimension problem and related cascade models can be found in \cite{ChenHanQiuWang2025Harmonic, ChenHanQiuWang2025MandelbrotKahane, ChenLiSuomala2025}.

Ryou and Suomala \cite{RyouSuomala2026} recently studied Mandelbrot cascades
on planar \(C^2\) curves with nonvanishing curvature.  In their \(b\)-adic
curvilinear setting, under superpolynomial tail assumptions, they proved that
the Fourier dimension is almost surely equal to the minimum lower pointwise
dimension \(\alpha_{\min}\) on non-extinction, and also obtained spherical
\(L^p\)-average decay estimates.

The present paper is complementary in three respects.  First, in the scalar
dyadic setting we work under the minimal Kahane--Peyri\`ere regime and obtain
an exact endpoint formula in terms of the cascade law.  More precisely, the
endpoint is identified explicitly as \(A_{\mathrm{loc}}(W)\), and each strict
subendpoint exponent requires only one finite-moment witness.  Second, the
fixed-arc lower bound is proved by a finite-\(r\) annular theorem with an
endpoint-safe phase decomposition, designed to handle one-sided endpoint
stationary regimes.  Third, the matching upper bounds are obtained by combining
deterministic curved-support obstructions with deterministic transfers of the
minimum local dimension.  The Jordan-curve formula is then obtained from the
fixed-arc theorem by first-generation dyadic cutting.

The scalar circle theorem of \cite{2026-M} is the benchmark and one of the
inputs for the upper bounds.  In the notation of the present paper, it assumes
the minimal Kahane--Peyri\`ere regime
\[
    W\ge0,\qquad
    \mathbb E W=1,\qquad
    \mathbb E[W\log_2^+W]<\infty,\qquad
    \mathbb E[W\log_2 W]<1,
\]
and identifies
\[
    A_{\mathrm{loc}}(W)
    =
    \sup_{q>1}
    \max\left\{
        0,\,
        \frac{q-1-\log_2\mathbb E[W^q]}{q}
    \right\},
\]
where the \(q\)-term is interpreted as \(0\) whenever
\(\mathbb E[W^q]=\infty\).  It proves
\[
    \dim_{\mathrm F}(\mu_\circ)=A_{\mathrm{loc}}(W)
\]
for the canonical scalar circle cascade on non-extinction.  The only external cascade-theoretic input used for the upper bounds is the
scalar-circle minimum lower local dimension identity
\[
    \alpha_{\min}(\mu_\circ)=A_{\mathrm{loc}}(W)
\]
from \cite{2026-M}.  The fixed-arc lower bound is proved independently of the
scalar-circle Fourier lower bound; it uses the annular martingale argument
developed in the present paper, together with standard martingale facts and the
complex Freedman inequality in the form recorded in \cite{2026-M}.

This result provides a natural benchmark, but it also raises a geometric
robustness question.  Is the endpoint Fourier-dimension formula a consequence
of the special circular parametrization, or does the same formula persist for
scalar cascades pushed forward to a fixed nondegenerate \(C^2\) curve?  This is
not merely a question of changing variables.  The unit circle has an explicit
trigonometric phase, constant curvature, global periodicity, and no endpoints.
A fixed \(C^2\) arc has no corresponding global structure.  For the phase
\[
    t\mapsto -2\pi \xi\cdot\gamma(t),
    \qquad 0\le t\le1,
\]
near-stationarity may occur at \(t=0\) or \(t=1\).  These endpoint regimes are
absent from the closed-circle argument and must be treated as part of the
Fourier analysis.

The purpose of the present paper is to prove that the same endpoint value
governs scalar cascade pushforwards to fixed nondegenerate \(C^2\) curves in
\(\mathbb R^2\), in two settings: fixed nondegenerate \(C^2\) embedded arcs
\(\gamma:[0,1]\to\mathbb R^2\), and fixed nondegenerate \(C^2\) Jordan curves
\(\gamma:\mathbb T\to\mathbb R^2\), where
\(\mathbb T=\mathbb R/\mathbb Z\).  For each fixed curve in these two classes,
the corresponding interval or circle cascade pushforward, respectively, has
Fourier dimension \(A_{\mathrm{loc}}(W)\) almost surely on the appropriate
non-extinction event.  Thus the endpoint value is determined by the cascade law
rather than by the particular fixed nondegenerate \(C^2\) curve, although the
constants in the Fourier estimates may depend on the chosen parametrized curve. In this sense, the endpoint Fourier-dimension formula is stable under replacing
the canonical circle by an arbitrary fixed nondegenerate \(C^2\) arc or Jordan
curve.

The fixed-arc lower bound is the main Fourier-analytic component of the paper.
It is proved within the present paper by a finite-\(r\) annular martingale
argument combined with an endpoint-safe phase decomposition.  The nonvanishing
curvature assumption supplies the quantitative phase geometry needed for the
derivative-band analysis, while the endpoint-safe pieces are controlled by
local-mass estimates.  This gives almost sure Fourier decay for every strict
exponent below \(A_{\mathrm{loc}}(W)\).  Apart from standard background facts
and the complex-valued Freedman inequality recorded in \cite{2026-M}, this
lower-bound argument does not use the scalar-circle Fourier lower bound.

The Jordan lower bound is reduced to the fixed-arc theorem by cutting the
circle cascade at the first dyadic generation.  This produces two interval
cascades on fixed dyadic subarcs, and the decay estimate is inherited by their
finite sum.

The upper bounds are obtained from local dimension rather than from annular
Fourier estimates.  The scalar-circle identity
\[
    \alpha_{\min}(\mu_\circ)=A_{\mathrm{loc}}(W)
\]
from \cite{2026-M} is transferred deterministically to the interval,
fixed-arc, and fixed-Jordan settings, and is then combined with deterministic
curved-support obstructions of the form
\[
    \dim_{\mathrm F}(\eta)\le \alpha_{\min}(\eta).
\]
The arc obstruction explicitly includes support points at the two endpoints,
where the local geometry is one-sided.

The results are fixed-curve statements.  The constants may depend on the chosen
parametrized arc or Jordan curve, including its \(C^2\)-norm, the lower bound
for \(|\gamma'|\), the lower curvature determinant bound, the diameter of the
image, and the relevant global embedding constants.  No uniformity over
families of curves is asserted.  We do not address curves with flat points,
random curves, curve-uniform families, or vector-valued cascades.

\subsection{Endpoint formulas for fixed arcs and Jordan curves}
\label{SS:introduction-models-main-results}

We first state the fixed-arc formula.  The curve class is the following.

\begin{definition}[Fixed nondegenerate \(C^2\) embedded arc]
\label{D:fixed-nondegenerate-arc}
A map \(\gamma:[0,1]\to\mathbb R^2\) is called a fixed nondegenerate
\(C^2\) embedded arc if \(\gamma\) is a \(C^2\) embedding and
\[
    \inf_{0\le t\le1}|\gamma'(t)|>0,
    \qquad
    \inf_{0\le t\le1}
    |\det(\gamma'(t),\gamma''(t))|>0 .
\]
\end{definition}

Let \(\widetilde\mu\) be the dyadic scalar Mandelbrot cascade on \([0,1]\)
generated by \(W\).  For a fixed nondegenerate \(C^2\) embedded arc
\(\gamma:[0,1]\to\mathbb R^2\), set
\(
\mu_\gamma=\gamma_\#\widetilde\mu,
\)
and let
\[
    S_\gamma=\{\widetilde\mu([0,1])>0\}
\]
be the non-extinction event.

\begin{theorem}[Endpoint formula on a fixed nondegenerate arc]
\label{T:arc-main-endpoint}
Let \(W\) be in the minimal Kahane--Peyri\`ere regime, and let
\(\widetilde\mu\) be the dyadic scalar Mandelbrot cascade on \([0,1]\)
generated by \(W\).  Let \(\gamma:[0,1]\to\mathbb R^2\) be a fixed
nondegenerate \(C^2\) embedded arc, and set
\(
\mu_\gamma=\gamma_\#\widetilde\mu.
\)
Then, almost surely on \(S_\gamma\),
\[
    \dim_{\mathrm F}(\mu_\gamma)=A_{\mathrm{loc}}(W).
\]
Moreover, if \(A_{\mathrm{loc}}(W)>0\), then, almost surely on \(S_\gamma\),
for every \(0<\sigma<A_{\mathrm{loc}}(W)\), there are finite random constants
\(C_\sigma,R_\sigma<\infty\) such that
\[
    |\widehat{\mu_\gamma}(\xi)|
    \le
    C_\sigma |\xi|^{-\sigma/2}
    \qquad
    (|\xi|\ge R_\sigma,\ \xi\in\mathbb R^2).
\]
\end{theorem}

We next state the closed-curve analogue.  Here
\(\mathbb T=\mathbb R/\mathbb Z\), and the curve class is defined as follows. 

Since the cascade lives on the parameter space, the parametrization is part of
the data.  For both arcs and Jordan curves, we allow any fixed \(C^2\)
parametrization satisfying the corresponding nondegeneracy assumptions.
No invariance of the resulting cascade pushforward under nonlinear
reparametrizations is asserted.

\begin{definition}[Fixed nondegenerate \(C^2\) Jordan curve]
\label{definition:fixed-curve}
A map \(\gamma:\mathbb T\to\mathbb R^2\) is called a fixed nondegenerate
\(C^2\) Jordan curve if \(\gamma\) is a \(C^2\) embedding and satisfies
 \[
     \inf_{t\in\mathbb T}|\gamma'(t)|>0,
     \qquad
    \inf_{t\in\mathbb T}
     |\det(\gamma'(t),\gamma''(t))|>0.
 \]
When derivatives are written, \(\gamma\) is represented by the associated
\(1\)-periodic \(C^2\) map
\(\bar\gamma=\gamma\circ\pi:\mathbb R\to\mathbb R^2\), where
\(\pi:\mathbb R\to\mathbb T\) is the quotient map; the notation
\(\gamma'(t)\), \(\gamma''(t)\) refers to the derivatives of
\(\bar\gamma\) at any representative of \(t\in\mathbb T\).
 \end{definition}

Let \(\widetilde\mu^{\mathbb T}\) be the scalar dyadic cascade on
\(\mathbb T\) generated by \(W\).  For a fixed nondegenerate \(C^2\) Jordan
curve \(\gamma:\mathbb T\to\mathbb R^2\), set
\(
    \mu_\gamma^{\mathbb T}
    =
    \gamma_\#\widetilde\mu^{\mathbb T},
\)
and let
\[
    S_\gamma^{\mathbb T}
    =
    \{\widetilde\mu^{\mathbb T}(\mathbb T)>0\}
\]
be the non-extinction event.

\begin{theorem}[Endpoint formula on a fixed nondegenerate Jordan curve]
\label{T:jordan-main-endpoint}
Let \(W\) be in the minimal Kahane--Peyri\`ere regime, and let
\(\widetilde\mu^{\mathbb T}\) be the dyadic scalar Mandelbrot cascade on
\(\mathbb T\) generated by \(W\).  Let
\(\gamma:\mathbb T\to\mathbb R^2\) be a fixed nondegenerate \(C^2\) Jordan
curve, and set
\(
    \mu_\gamma^{\mathbb T}
    =
    \gamma_\#\widetilde\mu^{\mathbb T}.
\)
Then, almost surely on \(S_\gamma^{\mathbb T}\),
\[
    \dim_{\mathrm F}(\mu_\gamma^{\mathbb T})
    =
    A_{\mathrm{loc}}(W).
\]
Moreover, if \(A_{\mathrm{loc}}(W)>0\), then, almost surely on
\(S_\gamma^{\mathbb T}\), for every
\(0<\sigma<A_{\mathrm{loc}}(W)\), there are finite random constants
\(C_\sigma,R_\sigma<\infty\) such that
\[
    |\widehat{\mu_\gamma^{\mathbb T}}(\xi)|
    \le
    C_\sigma |\xi|^{-\sigma/2}
    \qquad
    (|\xi|\ge R_\sigma,\ \xi\in\mathbb R^2).
\]
\end{theorem}

In both theorems, the Fourier decay estimate is stated only for strict
subendpoint exponents.  No uniform estimate is asserted at the endpoint
\(\sigma=A_{\mathrm{loc}}(W)\).  The endpoint equality is obtained at the
level of Fourier dimension by taking a countable family of strict exponents.

The two formulas have the same endpoint value, but their proofs use different
reductions.  The fixed-arc formula is proved through an endpoint-safe annular
argument.  The Jordan-curve lower bound is obtained by first-generation dyadic
cutting of the circle cascade into two interval cascades and applying the
fixed-arc theorem to the resulting subarcs.  The matching upper bound uses the scalar circle minimum lower local dimension
input together with the deterministic upper obstruction for measures supported
on a fixed closed curve.

\subsection{The finite-\texorpdfstring{\(r\)}{r} annular theorem}
\label{SS:introduction-finite-r-theorem}

The lower bound in Theorem~\ref{T:arc-main-endpoint} is obtained from a
finite-\(r\) annular Fourier theorem.  We state this theorem with an auxiliary
cascade weight \(U\).  This formulation is useful in the endpoint argument,
where \(U\) will be specialized to \(W\) only after a finite-\(r\) witness has
been chosen for a strict exponent below \(A_{\mathrm{loc}}(W)\).

\begin{theorem}[Finite-\(r\) annular Fourier decay on fixed arcs]
\label{T:arc-finite-r-annular}
Let \(0<s<1\).  Let \(U\ge0\) satisfy \(\mathbb E U=1\), and assume that there
exist \(r>1\) and \(\delta>0\) such that
\[
    \mathbb E[U^r]<\infty
    \qquad\text{and}\qquad
    2^{1-r}\mathbb E[U^r]\le 2^{-r(s+\delta)}.
\]
Let \(\widetilde\nu\) denote the almost sure weak limit of the dyadic scalar
Mandelbrot cascade on \([0,1]\) generated by \(U\).  Let
\(\gamma:[0,1]\to\mathbb R^2\) be a fixed nondegenerate \(C^2\) embedded arc,
and set \(\nu_\gamma=\gamma_\#\widetilde\nu\).  Then there exist constants
\(C_\gamma<\infty\), \(c_\gamma>0\), and \(\eta>0\), depending only on
\(s,r,\delta\), the law of \(U\), and the fixed arc \(\gamma\), such that, for
every integer \(n\ge1\),
\[
\mathbb P
\left(
    \sup_{2^n\le|\xi|\le2^{n+1}}
    |\widehat{\nu_\gamma}(\xi)|
    >
    C_\gamma 2^{-sn/2}
\right)
\le
C_\gamma\exp(-c_\gamma2^{\eta n})
+
C_\gamma2^{-c_\gamma n}.
\]
Consequently,
\[
    |\widehat{\nu_\gamma}(\xi)|
    =
    O(|\xi|^{-s/2})
    \qquad
    (|\xi|\to\infty)
\]
almost surely.
\end{theorem}

The estimate is annular: the supremum is taken uniformly over the whole
frequency annulus \(2^n\le|\xi|\le2^{n+1}\).  This annular uniformity allows
the summable probability estimate to be converted, by Borel--Cantelli, into
almost sure Fourier decay.  The stretched-exponential term comes from
martingale concentration after predictable capping, while the summable
\(2^{-c n}\) term comes from the \(r\)-tail compensator.

We now indicate how Theorem~\ref{T:arc-finite-r-annular} gives the endpoint
lower bound.  Let \(0<\sigma<A_{\mathrm{loc}}(W)\).  By the definition of
\(A_{\mathrm{loc}}(W)\), there exists \(r>1\) such that
\[
    \mathbb E[W^r]<\infty
    \qquad\text{and}\qquad
    \frac{r-1-\log_2\mathbb E[W^r]}{r}>\sigma .
\]
Choose \(\delta>0\) such that
\[
    \sigma+\delta
    <
    \frac{r-1-\log_2\mathbb E[W^r]}{r}.
\]
Equivalently,
\(
2^{1-r}\mathbb E[W^r]\le 2^{-r(\sigma+\delta)}.
\)
Since \(A_{\mathrm{loc}}(W)\le1\) by Jensen's inequality, we have
\(0<\sigma<1\).
Therefore Theorem~\ref{T:arc-finite-r-annular} applies with \(U=W\) and
\(s=\sigma\), and gives
\[
    |\widehat{\mu_\gamma}(\xi)|
    =
    O(|\xi|^{-\sigma/2})
    \qquad
    (|\xi|\to\infty)
\]
almost surely.

Taking a countable intersection over rational exponents
\(\sigma\in\mathbb Q\cap(0,A_{\mathrm{loc}}(W))\) shows that, almost surely,
every strict subendpoint exponent is an admissible Fourier decay exponent for
\(\mu_\gamma\).  On the non-extinction event \(S_\gamma\), the measure
\(\mu_\gamma\) is nonzero.  Hence
\[
    \dim_{\mathrm F}(\mu_\gamma)\ge A_{\mathrm{loc}}(W)
\]
almost surely on \(S_\gamma\).

The closed-curve lower bound requires no separate annular theorem.  The same
strict subendpoint decay is transferred to the Jordan-curve setting by the
first-generation dyadic cutting argument in Section~\ref{S:endpoint-assembly},
which reduces the circle cascade to two interval cascades on fixed
nondegenerate subarcs.

\subsection{Main ideas of the proof}
\label{SS:introduction-main-ideas}

We outline the proof at the level of its main mechanisms.  The fixed-arc lower
bound is the part that requires the annular Fourier estimate.  The upper bound
is local-dimensional, and the Jordan lower bound is obtained by a finite dyadic
cutting reduction to fixed arcs.

For the fixed-arc lower bound, the key deterministic input is an
endpoint-safe phase decomposition.  For \(\xi\ne0\), the phase is
\[
    \varphi_\xi(t)=-2\pi\xi\cdot\gamma(t),
    \qquad 0\le t\le1.
\]
For closed curves there are no boundary endpoint regimes.  On an interval,
however, the phase may be nearly stationary at \(t=0\) or \(t=1\).  The
decomposition therefore separates the endpoint-safe pieces, the
small-derivative piece, and the dyadic derivative bands:
\[
    \chi_{\xi,0}
    +
    \chi_{\xi,1}
    +
    \chi_{\xi,\mathrm{sd}}
    +
    \sum_{d\in\mathfrak D_\xi}\chi_{\xi,d}
    =
    1
    \qquad\text{on }[0,1].
\]
The endpoint-safe and small-derivative pieces are controlled by local-mass
estimates.  On each derivative band, the nonvanishing curvature gives the
phase geometry needed for uniform phase-bin coefficient estimates.

The probabilistic part of the fixed-arc theorem turns this deterministic
decomposition into annular martingale estimates.  After reduction to phase
bins, the increments are split into a predictably capped centered martingale
part and an \(r\)-tail compensator.  The centered part is controlled by a
complex Freedman inequality, while the compensator is bounded using the finite
\(r\)-moment hypothesis together with local-mass budgets.  A finite annular
grid then converts the phase-bin estimates into a uniform bound over each
frequency annulus.

The upper bound has a different nature.  It does not come from annular Fourier
estimates.  Instead, it uses a deterministic obstruction for measures supported
on curved sets: if a nonzero finite measure \(\eta\) is supported on a fixed
\(C^2\) embedded arc, then
\(
\dim_{\mathrm F}(\eta)\le \alpha_{\min}(\eta).
\)
The proof uses a Gaussian average along a normal frequency line.  Near every
point of a \(C^2\) arc, including endpoints, the normal projection of the curve
is quadratic in the arclength distance.  Fourier decay along the normal line
therefore forces an upper bound on local mass.  Applying this obstruction
to the cascade pushforward and combining it with the local-dimension identity
gives
\[
    \dim_{\mathrm F}(\mu_\gamma)
    \le
    \alpha_{\min}(\mu_\gamma)
    =
    A_{\mathrm{loc}}(W).
\]

For Jordan curves, no separate oscillatory analysis is needed.  After cutting
the circle cascade at the first dyadic generation, the pushforward by a fixed
Jordan curve decomposes as
\[
    \mu_\gamma^{\mathbb T}
    =
    \frac{W_0}{2}(\gamma\circ\rho_0)_\#\widetilde\mu^{(0)}
    +
    \frac{W_1}{2}(\gamma\circ\rho_1)_\#\widetilde\mu^{(1)},
\]
on the probability-one event on which dyadic endpoints carry no mass.  Since
\(\gamma\circ\rho_0\) and \(\gamma\circ\rho_1\) are fixed nondegenerate
\(C^2\) embedded arcs, the strict subendpoint Fourier decay obtained for fixed
arcs applies to both summands and therefore to their finite sum.  The Jordan
upper bound is obtained in the same local-dimensional way, using the
closed-curve version of the deterministic upper obstruction and the scalar
circle local-dimension input.

\subsection{Organization of the paper}
\label{SS:proof-strategy-organization}

Section~\ref{S:preliminaries} sets up the dyadic notation on \([0,1]\), the
scalar cascade notation, the moment assumptions, Fourier dimension, minimum
lower local dimension, endpoint gluing, the scalar circle local-dimension
input, and the dyadic cutting relation for circle cascades.  It also records
the fixed-arc and fixed-Jordan-curve conventions used throughout the paper.

Section~\ref{S:phase-geometry} proves the endpoint-safe phase geometry for
fixed arcs.  This includes the tangent-angle estimates, the small-derivative
region, the endpoint-safe pieces, the dyadic derivative bands, and the
phase-bin coefficient estimates.

Section~\ref{S:finite-r-annular} proves the finite-\(r\) annular theorem,
Theorem~\ref{T:arc-finite-r-annular}.  The section carries out the local-mass
estimates, the deterministic annular reduction, the capped martingale
concentration argument, the \(r\)-tail compensator estimate, and the final
annular grid assembly.

Section~\ref{S:local-dimension-upper-obstruction} proves the interval and
fixed-arc local-dimension identities and the deterministic upper obstruction
for measures supported on fixed \(C^2\) embedded arcs.  The local-dimension
part transfers the scalar circle local-dimension theorem first to the interval
cascade and then to fixed-arc pushforwards.  The obstruction part shows that a
measure supported on a fixed \(C^2\) embedded arc cannot have Fourier dimension
larger than its minimum lower local dimension.

Section~\ref{S:endpoint-assembly} assembles the endpoint results. It first
proves the fixed-arc endpoint formula by combining the annular lower bound with
the local-dimension upper bound.  It then treats fixed Jordan curves: the lower
bound is obtained by dyadic cutting into two fixed arcs, while the upper bound
is obtained from the closed-curve local-dimension identity and the
closed-curve deterministic upper obstruction.

\section{Preliminaries}
\label{S:preliminaries}

This section fixes the notation and basic inputs used throughout the proof.
We recall dyadic intervals, scalar dyadic cascades, moment assumptions,
Fourier dimension, minimum lower local dimension, endpoint gluing, the imported
scalar circle local-dimension theorem, and the dyadic cutting relation for
circle cascades.  The main parameter space is the unit interval \([0,1]\).
Closed curves parametrized by \(\mathbb T=\mathbb R/\mathbb Z\) enter only
through a dyadic cutting reduction to interval cascades.

We use standard facts on nonnegative cascade martingales and
Kahane--Peyri\`ere nondegeneracy; see, for example,
\cite{Kahane1985,KahanePeyriere1976}.  We also use standard facts on Fourier
dimension and Fourier-energy estimates; see \cite{Mattila1995}.

\subsection{Dyadic intervals and the binary tree}
\label{SS:dyadic-intervals}

For \(n\ge0\), let
\(\Sigma_n=\{0,1\}^n\),
\(\Sigma_*=\bigcup_{n=0}^{\infty}\Sigma_n\).
The unique word of length \(0\) is denoted by \(\varnothing\). 
If \(v=(v_1,\ldots,v_n)\in\Sigma_n\), then \(|v|=n\).  
For \(0\le k\le n\), set
\(
    v|k=(v_1,\ldots,v_k),
\)
with \(v|0=\varnothing\).
Concatenation of words is denoted by juxtaposition:
if \(v,w\in\Sigma_*\), then \(vw\) is the word obtained by appending \(w\) to \(v\). 
We write \(v\preceq w\) if \(v\) is a prefix of \(w\), 
and \(v\prec w\) if \(v\preceq w\) and \(v\ne w\).

For \(v=(v_1,\ldots,v_n)\in\Sigma_n\), set
\[
    a_v=\sum_{j=1}^{n} v_j2^{-j},
    \qquad
    a_\varnothing=0 .
\]
The dyadic interval associated to \(v\) is \(I_\varnothing=[0,1]\), and, for
\(n\ge1\),
\[
    I_v
    =
    \begin{cases}
    [a_v,a_v+2^{-n}), & v\ne (1,\ldots,1),\\
    [a_v,1], & v=(1,\ldots,1).
    \end{cases}
\]
Let
\[
    \mathcal D_n([0,1])=\{I_v:v\in\Sigma_n\}.
\]
Then \(\mathcal D_n([0,1])\) is a partition of \([0,1]\) into \(2^n\) intervals of length \(2^{-n}\). 
If \(I=I_v\in\mathcal D_n([0,1])\), its dyadic children are \(I_{v0}\) and \(I_{v1}\).

The half-open convention is used only to make the dyadic partition
single-valued at endpoints; all local and Fourier estimates are insensitive to
this convention.

For \(v\in\Sigma_n\), we sometimes use the affine map
\[
    \rho_v:[0,1]\to I_v,
    \qquad
    \rho_v(t)=a_v+2^{-n}t .
\]

\subsection{Scalar dyadic cascades on the interval}
\label{SS:scalar-cascades}

Let \(V\ge0\) satisfy \(\mathbb E V=1\). 
Attach independent copies
\(
    \{V_v:v\in\Sigma_*,\ v\ne\varnothing\}
\)
of \(V\) to the nonempty vertices of the binary tree.  
For
\(v\in\Sigma_n\), define
\[
    Q_v=\prod_{j=1}^{n} V_{v|j},
    \qquad
    Q_\varnothing=1 .
\]
The level-\(n\) scalar dyadic cascade measure on \([0,1]\) is
\[
    d\widetilde\lambda_n(t)
    =
    \sum_{|v|=n} Q_v\mathbf 1_{I_v}(t)\,dt .
\]
Equivalently,
\(\widetilde\lambda_n(I_v)=2^{-n}Q_v\)
for 
\(|v|=n\).

The total mass
\[
    Y_n=\widetilde\lambda_n([0,1])
    =
    2^{-n}\sum_{|v|=n}Q_v
\]
is a nonnegative martingale with respect to the natural filtration
\[
    \mathcal F_n^V
    =
    \sigma\{V_v:1\le |v|\le n\},
    \qquad
    \mathcal F_0^V \text{ trivial}.
\]
Hence \(Y_n\) converges almost surely to a finite limit \(Y\).  Whenever the
level measures \(\widetilde\lambda_n\) are known to converge weakly, their weak
limit is denoted by \(\widetilde\lambda\).

If \(v\in\Sigma_*\), the descendant environment rooted at \(v\) is
\[
    \{V_{vw}:w\in\Sigma_*,\ w\ne\varnothing\}.
\]
The level-\(m\) descendant total mass rooted at \(v\) is
\[
    Y_m^{(v)}
    =
    2^{-m}
    \sum_{|w|=m}
    \prod_{j=1}^{m} V_{v(w|j)},
\]
with \(Y_0^{(v)}=1\).  
Its limit is denoted by \(Y^{(v)}\).  
Descendant cascades rooted at distinct vertices of the same generation are independent, have the same law as the original cascade, 
and are independent of the environment above that generation.

In the finite-\(r\) annular theorem, the scalar weight is denoted by \(U\), the level-\(\ell\) parameter cascade is denoted by \(\widetilde\nu_\ell\), 
its weak limit by \(\widetilde\nu\), and the arc pushforward by
\(
 \nu_\gamma=\gamma_\#\widetilde\nu .
\)
The corresponding filtration is denoted by
\[
    \mathcal F_\ell
    =
    \sigma\{U_v:1\le |v|\le \ell\},
    \qquad
    \mathcal F_0 \text{ trivial}.
\]

In the endpoint theorem, the scalar weight is denoted by \(W\), the level-\(n\)
parameter cascade is denoted by \(\widetilde\mu_n\), its weak limit by
\(\widetilde\mu\), and the arc pushforward by
\(
    \mu_\gamma=\gamma_\#\widetilde\mu .
\)
The non-extinction event for the interval cascade is
\(
S_\gamma=\{\widetilde\mu([0,1])>0\}.
\)

Under the moment hypotheses used below, the standard nonnegative martingale
and Kahane--Peyri\`ere theory gives a probability-one event on which the
cascade measures converge weakly and all descendant terminal masses exist
simultaneously.  On this event, dyadic endpoints carry no limiting mass, and
\[
    \widetilde\lambda(I_v)=2^{-|v|}Q_vY^{(v)}
    \qquad (v\in\Sigma_*).
\]
Indeed, for each fixed binary coding \(\omega\),
\[
    \mathbb E\left[
        2^{-m}Q_{\omega|m}Y^{(\omega|m)}
    \right]\le 2^{-m},
\]
and dyadic endpoints are countable and have at most two binary codings.

\subsection{Moment assumptions and the endpoint exponent}
\label{SS:moment-assumptions}

\begin{definition}[Finite-\(r\) witnessed hypothesis]
\label{D:finite-r-witnessed}
Let \(0<s<1\).  A nonnegative random variable \(U\) with \(\mathbb E U=1\)
satisfies the finite-\(r\) witnessed hypothesis at exponent \(s\) if there exist \(r>1\) and \(\delta>0\) such that
\[
    \mathbb E[U^r]<\infty
    \qquad\text{and}\qquad
    2^{1-r}\mathbb E[U^r]\le 2^{-r(s+\delta)} .
\]
\end{definition}

\begin{definition}[Minimal Kahane--Peyri\`ere regime]
\label{D:minimal-kahane-peyriere}
A nonnegative random variable \(W\) is in the minimal Kahane--Peyri\`ere regime
if
\[
    W\ge0,
    \qquad
    \mathbb E W=1,
    \qquad
    \mathbb E[W\log_2^+W]<\infty,
    \qquad
    \mathbb E[W\log_2 W]<1,
\]
where
\(
    \log_2^+x=\max\{\log_2x,0\}    
\)
and
\(
0\log_2 0=0.
\)
\end{definition}

\begin{definition}[Endpoint local exponent]
\label{D:endpoint-local-exponent}
For a nonnegative random variable \(W\) with \(\mathbb E W=1\), define
\[
    A_{\mathrm{loc}}(W)
    =
    \sup_{q>1}
    \max\left\{
        0,\,
        \frac{q-1-\log_2\mathbb E[W^q]}{q}
    \right\},
\]
where the \(q\)-term is interpreted as \(0\) whenever
\(\mathbb E[W^q]=\infty\).
\end{definition}

By Jensen's inequality,
\(
   0\le A_{\mathrm{loc}}(W)\le1 .
\)
Indeed, \(\mathbb E[W^q]\ge1\) for \(q>1\), so each finite \(q\)-term is at most
\((q-1)/q<1\).

\subsection{Fourier dimension and local dimension}
\label{SS:fourier-local-dimension}

For a finite Borel measure \(\eta\) on \(\mathbb R^2\), define
\[
    \widehat\eta(\xi)
    =
    \int_{\mathbb R^2} e^{-2\pi i x\cdot\xi}\,d\eta(x),
    \qquad
    \xi\in\mathbb R^2 .
\]
If \(\eta=\gamma_\#\widetilde\eta\) for a finite Borel measure
\(\widetilde\eta\) on \([0,1]\), then
\[
    \widehat\eta(\xi)
    =
    \int_0^1 e^{-2\pi i \xi\cdot\gamma(t)}\,d\widetilde\eta(t).
\]

For a nonzero finite Borel measure \(\eta\) on \(\mathbb R^2\), its Fourier
dimension is
\[
    \dim_{\mathrm F}(\eta)
    =
    \sup\left\{
        0\le\sigma\le2:
        |\widehat\eta(\xi)|=O(|\xi|^{-\sigma/2})
        \text{ as }|\xi|\to\infty
    \right\}.
\]
We shall say that \(\sigma>0\) is an admissible Fourier decay exponent for
\(\eta\) if
\[
    |\widehat\eta(\xi)|=O(|\xi|^{-\sigma/2})
    \qquad (|\xi|\to\infty).
\]

The minimum lower local dimension of a nonzero finite Borel measure \(\eta\) on
\(\mathbb R^2\) is
\[
    \alpha_{\min}(\eta)
    =
    \inf_{x\in\operatorname{spt}\eta}
    \liminf_{\rho\downarrow0}
    \frac{\log_2\eta(B(x,\rho))}{\log_2\rho}.
\]
For a nonzero finite Borel measure \(\widetilde\eta\) on \([0,1]\), define
\[
    \alpha_{\min}^{[0,1]}(\widetilde\eta)
    =
    \inf_{t\in\operatorname{spt}\widetilde\eta}
    \liminf_{\rho\downarrow0}
    \frac{
        \log_2\widetilde\eta([0,1]\cap(t-\rho,t+\rho))
    }{
        \log_2\rho
    }.
\]
For a nonzero finite Borel measure \(\zeta\) on \(\mathbb T\), define
\[
    \alpha_{\min}^{\mathbb T}(\zeta)
    =
    \inf_{\theta\in\operatorname{spt}\zeta}
    \liminf_{\rho\downarrow0}
    \frac{
        \log_2\zeta(B_{\mathbb T}(\theta,\rho))
    }{
        \log_2\rho
    },
\]
where \(B_{\mathbb T}(\theta,\rho)\) is the ball in the quotient metric on
\(\mathbb T\).

We shall also use the standard Fourier-energy comparison
\(
    \dim_{\mathrm F}(\eta)\le \dim_{\mathrm H}(\operatorname{spt}\eta)
\)
for nonzero finite Borel measures \(\eta\) on \(\mathbb R^2\); see \cite{Mattila1995}.  
In particular,
every measure supported on a Lipschitz curve has Fourier dimension at most \(1\).

\subsection{Fixed-curve consequences}
\label{SS:fixed-curve-consequences}

The fixed curve classes are defined in
Definitions~\ref{D:fixed-nondegenerate-arc} and
\ref{definition:fixed-curve}.  We record here the elementary
consequences used later.

Constants denoted by \(C_\gamma,c_\gamma,\eta_\gamma\), and similar symbols may
depend on the fixed arc or Jordan curve \(\gamma\), including its \(C^2\)-norm,
the lower bound for \(|\gamma'|\), the lower curvature determinant bound, the
diameter of the image, and the relevant global embedding constants.  All
curve-dependent estimates are fixed-curve estimates; no uniformity over a class
of curves is asserted.

We shall use without further comment that a \(C^1\) embedding
\(\gamma:[0,1]\to\mathbb R^2\) with \(\inf|\gamma'|>0\) is bi-Lipschitz from
\([0,1]\) onto its image.  The analogous statement holds for a \(C^1\)
embedding \(\gamma:\mathbb T\to\mathbb R^2\) with \(\inf_{\mathbb T}|\gamma'|>0\),
viewed as a map from \((\mathbb T,d_{\mathbb T})\) onto its image.

\subsection{Endpoint gluing and the imported circle theorem}
\label{SS:endpoint-gluing-preliminaries}

Let \(\mathbb T=\mathbb R/\mathbb Z\), and let
\(q:[0,1]\to\mathbb T\) be the quotient map identifying \(0\) and \(1\).
Let
\[
    \mathbb S^1=\{x\in\mathbb R^2:|x|=1\},
\]
and define \(f_{\mathbb T}:\mathbb T\to\mathbb S^1\) by
\[
    f_{\mathbb T}(\theta)=(\cos 2\pi\theta,\sin 2\pi\theta).
\]
The map \(f_{\mathbb T}\) is bi-Lipschitz.  Indeed,
\[
    |f_{\mathbb T}(\theta)-f_{\mathbb T}(\theta')|
    =
    2|\sin(\pi d_{\mathbb T}(\theta,\theta'))|,
\]
where
\(
    d_{\mathbb T}(\theta,\theta')
    =
    \min_{k\in\mathbb Z}|\theta-\theta'-k|.
\)
The elementary inequalities for \(\sin\) on \([0,\pi/2]\) give the required
two-sided comparison.

If \(\widetilde\mu\) is the endpoint scalar cascade on \([0,1]\), 
define the associated circle pushforward by
\(
    \mu_\circ=(f_{\mathbb T})_\# q_\#\widetilde\mu .
\)
Then
\(
\mu_\circ(\mathbb S^1)=\widetilde\mu([0,1]),
\)
so the circle and interval non-extinction events agree.

The following quoted result is \cite[Theorem~7.12]{2026-M}, rewritten in the
notation of the present paper and in the precise form used below.  It is the
only external probabilistic input concerning minimum lower local dimensions.

\begin{theorem}[Scalar circle minimum lower local dimension theorem {\cite[Theorem~7.12]{2026-M}}]
\label{T:imported-circle-local-dimension}
Let \(W\) be in the minimal Kahane--Peyri\`ere regime.  Let \(\mu_\circ\) be
the scalar dyadic Mandelbrot cascade on \(\mathbb S^1\) generated by \(W\).
Then, almost surely on
\(
    \{\mu_\circ(\mathbb S^1)>0\},
\)
one has
\[
    \alpha_{\min}(\mu_\circ)=A_{\mathrm{loc}}(W).
\]
\end{theorem}

Theorem~\ref{T:imported-circle-local-dimension} is used below only through
deterministic transfers of minimum lower local dimension, namely through the
quotient map \(q:[0,1]\to\mathbb T\), the standard bi-Lipschitz
parametrization \(f_{\mathbb T}:\mathbb T\to\mathbb S^1\), and fixed
bi-Lipschitz curve parametrizations.  No scalar Fourier-decay theorem from
\cite{2026-M} is used to prove the fixed-arc or fixed-Jordan-curve lower
bounds.

\subsection{Circle cascades and dyadic cutting}
\label{SS:circle-cascades-cutting}

We now record the elementary dyadic cutting relation that will be used to deduce the closed-curve result from the arc result.

Let \(W\) be in the minimal Kahane--Peyri\`ere regime, and let
\(\widetilde\mu^{\mathbb T}\) be the scalar dyadic cascade on \(\mathbb T\)
generated by \(W\).  Let \(J_0\) and \(J_1\) be the two first-generation dyadic
arcs of \(\mathbb T\), corresponding respectively to \([0,1/2]\) and
\([1/2,1]\).  Define
\[
    \rho_0:[0,1]\to\mathbb T,
    \qquad
    \rho_0(t)=\frac t2,
\]
and
\[
    \rho_1:[0,1]\to\mathbb T,
    \qquad
    \rho_1(t)=\frac{1+t}{2}\pmod 1.
\]

Let \(W_0,W_1\) be the two first-generation weights of
\(\widetilde\mu^{\mathbb T}\).  Let \(\widetilde\mu^{(0)}\) and
\(\widetilde\mu^{(1)}\) be the descendant interval cascades rooted at the two
first-generation vertices.  Thus \(\widetilde\mu^{(0)}\) and
\(\widetilde\mu^{(1)}\) are independent scalar dyadic cascades on \([0,1]\),
generated by the same weight \(W\).

\begin{lemma}[Dyadic cutting of the circle cascade]
\label{L:circle-cascade-dyadic-cutting}
There exists a probability-one event \(E_{\mathrm{cut}}\), depending only on
the circle cascade and the two descendant interval cascades, such that, on
\(E_{\mathrm{cut}}\),
\[
    \widetilde\mu^{\mathbb T}|_{J_i}
    =
    \frac{W_i}{2}(\rho_i)_\#\widetilde\mu^{(i)},
    \qquad i=0,1.
\]
Consequently, for any continuous map \(\gamma:\mathbb T\to\mathbb R^2\),
\[
    \gamma_\#\widetilde\mu^{\mathbb T}
    =
    \frac{W_0}{2}(\gamma\circ\rho_0)_\#\widetilde\mu^{(0)}
    +
    \frac{W_1}{2}(\gamma\circ\rho_1)_\#\widetilde\mu^{(1)}.
\]
\end{lemma}

\begin{proof}
Let \(E_{\mathrm{cut}}\) be the probability-one event on which the circle
cascade and the two descendant interval cascades converge weakly, and on which
all dyadic endpoints have zero limiting mass.  This is the same endpoint-zero
event as in the scalar cascade construction above, applied to the circle
cascade and to the two descendant interval cascades.

We prove the identity for \(i=0\); 
the case \(i=1\) is identical.  
At level \(n+1\), the restriction of the circle cascade to \(J_0\) is obtained by fixing the first digit equal to \(0\).
Thus, for every dyadic interval
\(I_v\subset[0,1]\) with \(|v|=n\),
\[
    \widetilde\mu^{\mathbb T}_{n+1}(\rho_0(I_v))
    =
    2^{-(n+1)}W_0
    \prod_{j=1}^{n}W_{0(v|j)}
    =
    \frac{W_0}{2}\,\widetilde\mu^{(0)}_n(I_v).
\]
Equivalently,
\[
    \widetilde\mu^{\mathbb T}_{n+1}|_{J_0}
    =
    \frac{W_0}{2}(\rho_0)_\#\widetilde\mu^{(0)}_n .
\]
On \(E_{\mathrm{cut}}\), the boundary of \(J_0\) has zero
\(\widetilde\mu^{\mathbb T}\)-mass.  Hence restriction to \(J_0\) is continuous
along the weakly convergent sequence under consideration.  Pushforward by the
continuous map \(\rho_0\) is also continuous under weak convergence.  Passing to the limit
gives
\[
    \widetilde\mu^{\mathbb T}|_{J_0}
    =
    \frac{W_0}{2}(\rho_0)_\#\widetilde\mu^{(0)}.
\]
The same argument gives the formula on \(J_1\).

Finally, \(J_0\cup J_1=\mathbb T\), and their overlap is contained in the
dyadic endpoints, which have zero \(\widetilde\mu^{\mathbb T}\)-mass on
\(E_{\mathrm{cut}}\).  Therefore
\[
    \widetilde\mu^{\mathbb T}
    =
    \widetilde\mu^{\mathbb T}|_{J_0}
    +
    \widetilde\mu^{\mathbb T}|_{J_1}.
\]
Pushing this identity forward by \(\gamma\) gives the asserted formula.
\end{proof}

\subsection{Consequences for fixed nondegenerate Jordan curves}
\label{SS:fixed-jordan-curves}

The fixed nondegenerate \(C^2\) Jordan curves used in this paper are defined in
Definition~\ref{definition:fixed-curve}. The determinant condition is invariant under orientation reversal and rules out
flat points.

We shall also use that such a curve is bi-Lipschitz from
\((\mathbb T,d_{\mathbb T})\) onto its image.  This follows from compactness
and from the fact that \(\gamma\) is a \(C^1\) embedding with
\(\inf_{\mathbb T}|\gamma'|>0\).  Consequently, pushforward by \(\gamma\) preserves minimum lower local
dimension by the standard bi-Lipschitz ball-comparison argument.

\begin{lemma}
\label{L:jordan-dyadic-subarcs}
Let \(\gamma:\mathbb T\to\mathbb R^2\) be a fixed nondegenerate \(C^2\) Jordan
curve.  For \(i=0,1\), set
\[
    \gamma_i=\gamma\circ\rho_i:[0,1]\to\mathbb R^2,
\]
where \(\rho_0,\rho_1\) are the first-generation dyadic maps defined in
Subsection~\ref{SS:circle-cascades-cutting}.  Then each \(\gamma_i\) is a fixed
nondegenerate \(C^2\) embedded arc.
\end{lemma}

\begin{proof}
Since \(\gamma\) is a \(C^2\) embedding of \(\mathbb T\), its restriction to
each first-generation dyadic half-arc of \(\mathbb T\) is a \(C^2\) embedded
arc after parametrization by \(\rho_i\).  Moreover,
\[
    \gamma_i'(t)=\frac12\gamma'(\rho_i(t)),
    \qquad
    \gamma_i''(t)=\frac14\gamma''(\rho_i(t)).
\]
Therefore
\[
    \det(\gamma_i'(t),\gamma_i''(t))
    =
    \frac18
    \det(\gamma'(\rho_i(t)),\gamma''(\rho_i(t))).
\]
The lower bounds for \(|\gamma_i'|\) and
\(|\det(\gamma_i',\gamma_i'')|\) follow from the corresponding lower bounds for
\(\gamma\).  Hence each \(\gamma_i\) is a fixed nondegenerate \(C^2\) embedded
arc.
\end{proof}

\subsection{Phase notation and annular conventions}
\label{SS:phase-notation-conventions}

Throughout the phase-geometry and annular sections,
\(
    \gamma:[0,1]\to\mathbb R^2
\)
denotes a fixed nondegenerate \(C^2\) embedded arc.  For
\(\xi\in\mathbb R^2\setminus\{0\}\), the arc phase is
\[
    \varphi_\xi(t)
    =
    -2\pi\xi\cdot\gamma(t),
    \qquad
    0\le t\le1.
\]
The derivative-band variable is
\[
    h_\xi(t)^{1/2}
    =
    \frac{|\xi\cdot\gamma'(t)|}{|\xi|\,|\gamma'(t)|}.
\]

For \(n\ge1\), write
\[
    \mathcal A_n
    =
    \{\xi\in\mathbb R^2:2^n\le |\xi|\le 2^{n+1}\}.
\]
Finitely many low-frequency annuli are always absorbed into constants.

Once a derivative-band scale \(d\) has been introduced, the associated
phase-bin generation \(m_{\xi,d}\) is the unique integer satisfying
\[
    2^{m_{\xi,d}-1}<|\xi|d\le 2^{m_{\xi,d}}.
\]
We also set
\[
    m_{\xi,d,+}=\max\{m_{\xi,d},0\}.
\]

After the cutoffs \(\chi_{\xi,d}\) are constructed in the phase decomposition,
the coefficient convention used in the finite-\(r\) annular proof is
\[
    c_J(\xi,d,\ell)
    =
    2^\ell
    \int_J e^{i\varphi_\xi(t)}
    \chi_{\xi,d}(t)\,dt,
    \qquad
    J\in\mathcal D_{\ell+1}([0,1]).
\]

If \(I\in\mathcal D_\ell([0,1])\), \(J\in\mathcal D_{\ell+1}([0,1])\),
\(J\subset I\), and
\(
    M_I=\widetilde\nu_\ell(I),
\)
then
\[
    \widetilde\nu_{\ell+1}(J)=\frac12 M_IU_J,
\]
where \(U_J\) denotes the cascade weight attached to the child interval \(J\).
With the coefficient convention above, the raw martingale increment associated
with \(J\) is
\(
c_J(\xi,d,\ell)M_I(U_J-1).
\)

Constants denoted by \(C_\gamma,c_\gamma,\eta_\gamma\), and similar symbols may
change from line to line.  They may depend on the fixed arc \(\gamma\), on the
parameters fixed in the relevant theorem, and on the law of the cascade weight,
but not on \(\xi\), \(n\), \(\ell\), or on the dyadic intervals under
consideration, unless explicitly stated otherwise.

\section{Endpoint-safe phase geometry for arcs}
\label{S:phase-geometry}

This section proves the deterministic phase package used in the finite-\(r\)
annular theorem for fixed nondegenerate arcs.  Throughout the section,
\(
    \gamma:[0,1]\to\mathbb R^2
\)
is a fixed nondegenerate \(C^2\) embedded arc.  We use the phase
\(\varphi_\xi\) and the derivative-band variable \(h_\xi\) introduced in
Subsection~\ref{SS:phase-notation-conventions}.

The closed-curve case will not require a separate phase decomposition.  By the
dyadic cutting reduction in Lemma~\ref{L:circle-cascade-dyadic-cutting} and
Lemma~\ref{L:jordan-dyadic-subarcs}, a fixed nondegenerate Jordan curve is
reduced to two fixed nondegenerate arcs.  Therefore the endpoint-safe arc
package proved in this section is the only deterministic phase geometry needed
for the Fourier-decay part of the paper.

All constants \(C_\gamma,c_\gamma\), and similar constants in this section may
depend on the fixed arc \(\gamma\), but not on the frequency \(\xi\).  The main
point is that an arc has endpoints.  Thus the phase decomposition must include
unconditional endpoint-safe pieces, in addition to the usual small-derivative
region and dyadic derivative bands.

\subsection{Tangent angle and derivative level sets}
\label{SS:tangent-angle-level-sets}

\begin{lemma}
\label{L:arc-tangent-angle}
There exists a \(C^1\) function
\(
    \Theta_\gamma:[0,1]\to\mathbb R
\)
such that
\[
    \gamma'(t)
    =
    |\gamma'(t)|
    (\cos\Theta_\gamma(t),\sin\Theta_\gamma(t)).
\]
Moreover,
\[
    \Theta_\gamma'(t)
    =
    \frac{\det(\gamma'(t),\gamma''(t))}{|\gamma'(t)|^2},
\]
and there are constants \(0<c_\gamma\le C_\gamma<\infty\) such that
\[
    c_\gamma\le |\Theta_\gamma'(t)|\le C_\gamma
    \qquad (0\le t\le1).
\]
Consequently, \(\Theta_\gamma\) is monotone and bi-Lipschitz from \([0,1]\)
onto its image.
\end{lemma}

\begin{proof}
Since \(\gamma\in C^2\) and \(\inf_{[0,1]}|\gamma'|>0\), the unit tangent field
\(
    \frac{\gamma'(t)}{|\gamma'(t)|}
\)
is a well-defined \(C^1\) map from \([0,1]\) to \(\mathbb S^1\).  Since the
interval is simply connected, this map admits a \(C^1\) angle lift
\(\Theta_\gamma\).  Differentiating
\[
    \gamma'(t)
    =
    |\gamma'(t)|
    (\cos\Theta_\gamma(t),\sin\Theta_\gamma(t))
\]
and taking the determinant with \(\gamma'(t)\) gives
\[
    \det(\gamma'(t),\gamma''(t))
    =
    |\gamma'(t)|^2\Theta_\gamma'(t).
\]
The curvature determinant is continuous and nowhere zero on \([0,1]\), hence
has a fixed sign and a positive absolute minimum.  Since \(|\gamma'|\) is
continuous and bounded above and below away from zero, the displayed bounds for
\(\Theta_\gamma'\) follow.  The monotonicity and bi-Lipschitz property are
immediate.
\end{proof}

If
\(
    \xi=|\xi|(\cos\alpha,\sin\alpha),
\)
then, by the definition of \(h_\xi\),
\[
    h_\xi(t)=\cos^2(\Theta_\gamma(t)-\alpha).
\]

\begin{lemma}
\label{L:arc-g-level-geometry}
For every \(\xi\ne0\),
\[
    c_\gamma|\xi|h_\xi(t)^{1/2}
    \le
    |\varphi_\xi'(t)|
    \le
    C_\gamma|\xi|h_\xi(t)^{1/2}
    \qquad (0\le t\le1),
\]
and
\[
    |h_\xi'(t)|\le C_\gamma h_\xi(t)^{1/2}
    \qquad (0\le t\le1).
\]
Moreover, for every \(0<d\le1\), the sublevel set
\[
    \left\{
        t\in[0,1]:
        h_\xi(t)^{1/2}\le C_\gamma d
    \right\}
\]
is contained in the union of at most \(C_\gamma\) intervals of total length at
most \(C_\gamma d\).  Consequently, for every dyadic \(d>0\), the band
\[
    \left\{
        t\in[0,1]:
        C_\gamma^{-1}d
        \le
        h_\xi(t)^{1/2}
        \le
        C_\gamma d
    \right\}
\]
is contained in the union of at most \(C_\gamma\) intervals of total length at
most \(C_\gamma d\).
\end{lemma}

\begin{proof}
The derivative of the phase is
\(
    \varphi_\xi'(t)=-2\pi\xi\cdot\gamma'(t),
\)
so
\[
    |\varphi_\xi'(t)|
    =
    2\pi|\xi|\,|\gamma'(t)|\,h_\xi(t)^{1/2}.
\]
Since \(|\gamma'|\) is bounded above and below away from zero, the first
estimate follows.

Using
\(
    h_\xi(t)=\cos^2(\Theta_\gamma(t)-\alpha),
\)
we get
\[
    h_\xi'(t)
    =
    -2
    \cos(\Theta_\gamma(t)-\alpha)
    \sin(\Theta_\gamma(t)-\alpha)
    \Theta_\gamma'(t),
\]
and hence
\(
    |h_\xi'(t)|\le C_\gamma h_\xi(t)^{1/2}.
\)

It remains to prove the level-set estimates.  By
Lemma~\ref{L:arc-tangent-angle}, the map \(t\mapsto\Theta_\gamma(t)\) is
bi-Lipschitz onto an interval of bounded length.  In the angle variable
\(u=\Theta_\gamma(t)-\alpha\), the zeros of \(\cos u\) in this interval are
simple and have cardinality at most \(C_\gamma\).  The set where
\(|\cos u|\le C_\gamma d\) is therefore contained in at most \(C_\gamma\)
intervals of total length at most \(C_\gamma d\).  Pulling back by the
bi-Lipschitz map \(\Theta_\gamma\) proves the sublevel estimate.  The band
estimate follows from the sublevel estimate.  If \(d>1\), the assertion is
trivial after increasing \(C_\gamma\).
\end{proof}

\subsection{Endpoint tubes and dyadic derivative bands}
\label{SS:endpoint-tubes-derivative-bands}

We now construct the endpoint-safe partition.  The low-frequency range is
harmless and will be absorbed into constants.

For the high-frequency construction, assume
\(
    |\xi|>16,
\)
and set
\(
    \rho_\xi=|\xi|^{-1/2}.
\)
Choose \(C^1\) functions
\(
    \eta_{\xi,0},\ \eta_{\xi,1}:[0,1]\to[0,1]
\)
such that
\[
    \eta_{\xi,0}(t)=0\quad (0\le t\le \rho_\xi),
    \qquad
    \eta_{\xi,0}(t)=1\quad (2\rho_\xi\le t\le1),
\]
and
\[
    \eta_{\xi,1}(t)=1\quad (0\le t\le1-2\rho_\xi),
    \qquad
    \eta_{\xi,1}(t)=0\quad (1-\rho_\xi\le t\le1).
\]
They are chosen with
\[
    |\eta_{\xi,0}'(t)|+|\eta_{\xi,1}'(t)|
    \le C|\xi|^{1/2},
\qquad
    \int_0^1|\eta_{\xi,0}'(t)|\,dt
    +
    \int_0^1|\eta_{\xi,1}'(t)|\,dt
    \le C.
\]
Define
\[
    \eta_\xi(t)=\eta_{\xi,0}(t)\eta_{\xi,1}(t),
    \qquad
    \chi_{\xi,0}(t)=1-\eta_{\xi,0}(t),
    \qquad
    \chi_{\xi,1}(t)=\eta_{\xi,0}(t)(1-\eta_{\xi,1}(t)).
\]
Thus \(\chi_{\xi,0}\) and \(\chi_{\xi,1}\) are endpoint-safe pieces.

Choose a \(C^\infty\) cutoff
\(
    \psi:[0,\infty)\to[0,1]
\)
such that
\[
    \psi(u)=1\quad (0\le u\le1),
    \qquad
    \psi(u)=0\quad (u\ge4).
\]
Define the small-derivative piece
\[
    \chi_{\xi,\mathrm{sd}}(t)
    =
    \eta_\xi(t)\psi(|\xi|h_\xi(t)).
\]

Next choose a nonnegative \(C^\infty\) function \(\zeta\), supported in
\([1/4,4]\), such that
\[
    \sum_{d\in2^{\mathbb Z}}
    \zeta\left(\frac{u}{d^2}\right)=1
    \qquad (u>0).
\]
For a dyadic \(d\), define
\[
    \chi_{\xi,d}(t)
    =
    \eta_\xi(t)
    \left(1-\psi(|\xi|h_\xi(t))\right)
    \zeta\left(\frac{h_\xi(t)}{d^2}\right).
\]
Let \(\mathfrak D_\xi\) be the finite set of dyadic \(d\)'s for which
\(\chi_{\xi,d}\) is not identically zero.

For \(|\xi|\le16\), we use the convention
\[
    \chi_{\xi,\mathrm{sd}}\equiv1,
    \qquad
    \chi_{\xi,0}\equiv0,
    \qquad
    \chi_{\xi,1}\equiv0,
    \qquad
    \mathfrak D_\xi=\varnothing.
\]

\begin{proposition}[Endpoint-safe phase-bin package]
\label{P:arc-endpoint-safe-phase-bins}
For every \(\xi\in\mathbb R^2\setminus\{0\}\), the following conclusions hold.

\begin{enumerate}[label=\textup{(\roman*)}]
\item Partition of unity.  The functions above satisfy
\[
    \chi_{\xi,0}(t)
    +
    \chi_{\xi,1}(t)
    +
    \chi_{\xi,\mathrm{sd}}(t)
    +
    \sum_{d\in\mathfrak D_\xi}\chi_{\xi,d}(t)
    =
    1
    \qquad (0\le t\le1).
\]

\item Size of the safe region.  Let
\[
    E_\xi
    =
    \operatorname{spt}\chi_{\xi,0}
    \cup
    \operatorname{spt}\chi_{\xi,1}
    \cup
    \operatorname{spt}\chi_{\xi,\mathrm{sd}}.
\]
Then
\[
    |E_\xi|\le C_\gamma|\xi|^{-1/2}.
\]

\item Localization of each derivative band.  For each
\(d\in\mathfrak D_\xi\), set
\(
    S_{\xi,d}=\operatorname{spt}\chi_{\xi,d}.
\)
Then
\[
    S_{\xi,d}
    \subset
    \left\{
        t\in[0,1]:
        C_\gamma^{-1}d
        \le
        h_\xi(t)^{1/2}
        \le
        C_\gamma d
    \right\}.
\]

\item Phase-derivative bounds on each band.  For each
\(d\in\mathfrak D_\xi\), one has
\[
    c_\gamma|\xi|d
    \le
    |\varphi_\xi'(t)|
    \le
    C_\gamma|\xi|d
    \qquad (t\in S_{\xi,d}).
\]

\item Range and number of active dyadic scales.  The active dyadic scales
satisfy
\[
    C_\gamma^{-1}|\xi|^{-1/2}\le d\le C_\gamma
    \qquad (d\in\mathfrak D_\xi),
\]
and
\[
    \#\mathfrak D_\xi
    \le
    C_\gamma(1+\log(2+|\xi|)).
\]

\item Thickening estimate for each band.  For every \(d\in\mathfrak D_\xi\)
and every integer \(\ell\ge0\),
\[
    \left|
    \left\{
        t\in[0,1]:
        \operatorname{dist}(t,S_{\xi,d})\le2^{-\ell}
    \right\}
    \right|
    \le
    C_\gamma(d+2^{-\ell}).
\]

\item Bounded-variation integration-by-parts estimate.  For every
\(d\in\mathfrak D_\xi\),
\[
    \int_0^1
    \frac{|\chi_{\xi,d}'(t)|}{|\varphi_\xi'(t)|}\,dt
    \le
    C_\gamma(|\xi|d)^{-1}.
\]
\end{enumerate}
\end{proposition}

\begin{proof}
For \(|\xi|\le16\), the assertions are immediate after increasing
\(C_\gamma\).  Assume \(|\xi|>16\).

The partition identity follows from
\(
    \chi_{\xi,0}(t)+\chi_{\xi,1}(t)=1-\eta_\xi(t).
\)
If \(h_\xi(t)=0\), then \(\psi(|\xi|h_\xi(t))=1\), so
\[
    \chi_{\xi,\mathrm{sd}}(t)=\eta_\xi(t),
    \qquad
    \chi_{\xi,d}(t)=0.
\]
If \(h_\xi(t)>0\), then the dyadic partition of unity gives
\[
    \chi_{\xi,\mathrm{sd}}(t)
    +
    \sum_{d\in\mathfrak D_\xi}\chi_{\xi,d}(t)
    =
    \eta_\xi(t).
\]
This proves the partition of unity.

The two endpoint pieces have supports of total length at most
\(C|\xi|^{-1/2}\).  Also,
\[
    \operatorname{spt}\chi_{\xi,\mathrm{sd}}
    \subset
    \left\{
        t\in[0,1]:
        h_\xi(t)^{1/2}\le2|\xi|^{-1/2}
    \right\}.
\]
By Lemma~\ref{L:arc-g-level-geometry}, this set is contained in at most
\(C_\gamma\) intervals of total length at most \(C_\gamma|\xi|^{-1/2}\).
Hence
\(
    |E_\xi|\le C_\gamma|\xi|^{-1/2}.
\)

If \(t\in S_{\xi,d}\), then the support of \(\zeta\) gives
\(
\frac12 d\le h_\xi(t)^{1/2}\le2d,
\)
which proves the localization.  Combining this with
Lemma~\ref{L:arc-g-level-geometry} gives
\[
    c_\gamma|\xi|d
    \le
    |\varphi_\xi'(t)|
    \le
    C_\gamma|\xi|d
    \qquad (t\in S_{\xi,d}).
\]
Since \(h_\xi\le1\), the same support condition gives \(d\le C_\gamma\).
Since \(1-\psi(|\xi|h_\xi)\ne0\) on \(S_{\xi,d}\), we have
\(|\xi|h_\xi>1\) somewhere on the band, hence
\(d\ge C_\gamma^{-1}|\xi|^{-1/2}\).  The bound on
\(\#\mathfrak D_\xi\) follows from the number of dyadic scales in this range.

The support-neighborhood estimate follows from the localization and
Lemma~\ref{L:arc-g-level-geometry}: \(S_{\xi,d}\) is contained in a union of at
most \(C_\gamma\) intervals of total length at most \(C_\gamma d\).  Enlarging
a union of \(C_\gamma\) intervals by distance \(2^{-\ell}\) increases its
length by at most \(C_\gamma2^{-\ell}\).

It remains to prove the bounded-variation estimate.  Differentiate
\[
    \chi_{\xi,d}
    =
    \eta_\xi
    (1-\psi(|\xi|h_\xi))
    \zeta(h_\xi/d^2).
\]
There are three contributions.

First, for the endpoint-transition term, we use
\(
    \int_0^1|\eta_\xi'(t)|\,dt\le C
\)
and the lower bound \(|\varphi_\xi'(t)|\ge c_\gamma|\xi|d\) on the
\(d\)-band.  Hence
\[
    \int_0^1
    \frac{
        |\eta_\xi'(t)|
        (1-\psi(|\xi|h_\xi(t)))
        \zeta(h_\xi(t)/d^2)
    }{
        |\varphi_\xi'(t)|
    }\,dt
    \le
    \frac{1}{c_\gamma|\xi|d}
    \int_0^1|\eta_\xi'(t)|\,dt
    \le
    C_\gamma(|\xi|d)^{-1}.
\]

Second, the derivative of \(1-\psi(|\xi|h_\xi)\) is supported where
\(
    1\le |\xi|h_\xi(t)\le4.
\)
If this intersects the \(d\)-band, then \(d\asymp_\gamma|\xi|^{-1/2}\).  On
this set,
\[
    \left|
    \frac{d}{dt}\psi(|\xi|h_\xi(t))
    \right|
    \le
    C|\xi|\,|h_\xi'(t)|
    \le
    C_\gamma|\xi|h_\xi(t)^{1/2}
    \le
    C_\gamma|\xi|^{1/2}
    \le
    C_\gamma d^{-1}.
\]
The transition set has length at most \(C_\gamma d\), and
\(|\varphi_\xi'|\ge c_\gamma|\xi|d\) on the band.  Hence this contribution is
also at most \(C_\gamma(|\xi|d)^{-1}\).

Third, for the dyadic band cutoff,
\[
    \left|
    \frac{d}{dt}
    \zeta\left(\frac{h_\xi(t)}{d^2}\right)
    \right|
    \le
    C\frac{|h_\xi'(t)|}{d^2}
    \le
    C_\gamma d^{-1}
\]
on the \(d\)-band.  Since the support has length at most \(C_\gamma d\) and
\(|\varphi_\xi'|\ge c_\gamma|\xi|d\), the contribution is at most
\(C_\gamma(|\xi|d)^{-1}\).  Combining the three estimates proves the
bounded-variation estimate.
\end{proof}

\begin{remark}
\label{R:endpoint-safe-pieces}
The endpoint pieces \(\chi_{\xi,0}\) and \(\chi_{\xi,1}\) are not treated by
integration by parts.  They are included in the safe region \(E_\xi\), whose
mass will later be controlled by the local-mass good event.  The oscillatory
integration-by-parts estimates are applied only on the derivative bands
\(\chi_{\xi,d}\), which are separated from the endpoints by the factor
\(\eta_\xi\).  This is the reason for using an endpoint-safe phase
decomposition rather than the stationary-tube decomposition used for closed
curves.
\end{remark}

\subsection{Coefficient estimates}
\label{SS:coefficient-estimates}

For \(d\in\mathfrak D_\xi\), recall from
Subsection~\ref{SS:phase-notation-conventions} that \(m_{\xi,d}\) is the
integer determined by
\(
    2^{m_{\xi,d}-1}<|\xi|d\le2^{m_{\xi,d}}.
\)
For \(J\in\mathcal D_{\ell+1}([0,1])\), recall that
\[
    c_J(\xi,d,\ell)
    =
    2^\ell
    \int_J e^{i\varphi_\xi(t)}
    \chi_{\xi,d}(t)\,dt.
\]

\begin{proposition}
\label{P:arc-phasebin-coefficient-estimates}
There exists \(C_\gamma<\infty\) such that, for every
\(\xi\ne0\), every \(d\in\mathfrak D_\xi\), every \(\ell\ge0\), and every
\(J\in\mathcal D_{\ell+1}([0,1])\), one has
\[
    |c_J(\xi,d,\ell)|
    \le
    C_\gamma2^{\ell-m_{\xi,d}}
    \qquad (0\le\ell<m_{\xi,d}),
\]
and
\[
    |c_J(\xi,d,\ell)|
    \le
    C_\gamma
    \qquad (\ell\ge m_{\xi,d}).
\]
Moreover,
\[
    \left|
    \int_0^1 e^{i\varphi_\xi(t)}
    \chi_{\xi,d}(t)\,dt
    \right|
    \le
    C_\gamma(|\xi|d)^{-1}.
\]
\end{proposition}

\begin{proof}
The post-bin estimate is immediate:
\[
    |c_J(\xi,d,\ell)|
    \le
    2^\ell |J|
    =
    2^\ell2^{-(\ell+1)}
    \le1.
\]

Assume \(0\le\ell<m_{\xi,d}\).  It is enough to show
\[
    \left|
    \int_J e^{i\varphi_\xi(t)}
    \chi_{\xi,d}(t)\,dt
    \right|
    \le
    C_\gamma(|\xi|d)^{-1}.
\]
Set
\(
    G(t)=\chi_{\xi,d}(t).
\)
The set \(J\cap\operatorname{spt}G\) is contained in the union of at most
\(C_\gamma\) intervals of total length at most \(C_\gamma d\).  On these
intervals,
\[
    |\varphi_\xi'(t)|\ge c_\gamma|\xi|d,
    \qquad
    |\varphi_\xi''(t)|\le C_\gamma|\xi|.
\]
The set \(J\cap\{G\ne 0\}\) has at most \(C_\gamma\) connected components, after endpoints at which \(G\) vanishes are discarded. 
On each such component \(K\), integration by parts gives 
\[
    \int_K e^{i\varphi_\xi(t)}G(t)\,dt
    =
    \left[
        \frac{e^{i\varphi_\xi(t)}G(t)}
             {i\varphi_\xi'(t)}
    \right]_{\partial K}
    -
    \int_K e^{i\varphi_\xi(t)}
    \left(
        \frac{G'(t)}{i\varphi_\xi'(t)}
        -
        \frac{G(t)\varphi_\xi''(t)}
             {i(\varphi_\xi'(t))^2}
    \right)\,dt.
\]
The boundary contribution over all components is at most
\(C_\gamma(|\xi|d)^{-1}\).  The \(G'\)-term is controlled by
Proposition~\ref{P:arc-endpoint-safe-phase-bins}:
\[
    \int_0^1
    \frac{|G'(t)|}{|\varphi_\xi'(t)|}\,dt
    \le
    C_\gamma(|\xi|d)^{-1}.
\]
For the \(\varphi_\xi''\)-term,
\[
    \int_{\operatorname{spt}G}
    \frac{|G(t)|\,|\varphi_\xi''(t)|}
         {|\varphi_\xi'(t)|^2}\,dt
    \le
    C_\gamma
    \int_{\operatorname{spt}G}
    \frac{|\xi|}{|\xi|^2d^2}\,dt
    \le
    C_\gamma(|\xi|d)^{-1},
\]
because \(|\operatorname{spt}G|\le C_\gamma d\).  This proves the integral
estimate on \(J\).  Multiplying by \(2^\ell\), and using
\[
    2^{m_{\xi,d}-1}<|\xi|d\le2^{m_{\xi,d}},
\]
gives
\[
    |c_J(\xi,d,\ell)|
    \le
    C_\gamma2^\ell(|\xi|d)^{-1}
    \le
    C_\gamma2^{\ell-m_{\xi,d}}.
\]

The full-interval arclength estimate is the same integration-by-parts argument
with \(J=[0,1]\).
\end{proof}

\section{The finite-\texorpdfstring{\(r\)}{r} annular theorem}
\label{S:finite-r-annular}

This section begins the proof of Theorem~\ref{T:arc-finite-r-annular}.
Throughout Section~\ref{S:finite-r-annular}, fix
\[
    0<s<1,\qquad r>1,\qquad \delta>0,
\]
and let \(U\ge0\) satisfy
\[
    \mathbb E U=1,\qquad
    \mathbb E[U^r]<\infty,
    \qquad
    2^{1-r}\mathbb E[U^r]\le 2^{-r(s+\delta)}.
\]
Let \(\widetilde\nu_\ell\) be the level-\(\ell\) scalar dyadic cascade on
\([0,1]\) generated by \(U\).

We first record the elementary finite-\(r\) consequence needed to place the
cascade in the Kahane--Peyri\`ere nondegenerate regime.  Since \(r>1\) and
\(\mathbb E[U^r]<\infty\), we have
\[
    \mathbb E[U\log_2^+U]<\infty.
\]
Let
\[
    \Phi(q)=\log_2\mathbb E[U^q],
    \qquad 1\le q\le r.
\]
Then \(\Phi\) is convex on \([1,r]\), \(\Phi(1)=0\), and the finite-\(r\)
hypothesis gives
\[
    \Phi(r)
    =
    \log_2\mathbb E[U^r]
    \le
    r-1-r(s+\delta)
    <
    r-1.
\]
Since \(\mathbb E[U^r]<\infty\), the right derivative \(\Phi'_+(1)\) exists and
equals
\[
    \Phi'_+(1)=\mathbb E[U\log_2 U].
\]
By convexity,
\[
    \mathbb E[U\log_2 U]
    =
    \Phi'_+(1)
    \le
    \frac{\Phi(r)-\Phi(1)}{r-1}
    <
    1.
\]
Thus \(U\) satisfies the Kahane--Peyri\`ere assumptions
\[
    U\ge0,\qquad
    \mathbb E U=1,\qquad
    \mathbb E[U\log_2^+U]<\infty,\qquad
    \mathbb E[U\log_2U]<1.
\]
Consequently, the scalar cascade
\((\widetilde\nu_\ell)_{\ell\ge0}\) converges weakly almost surely to a finite
random Borel measure, denoted by \(\widetilde\nu\).  We set
\(
    \nu_\gamma=\gamma_\#\widetilde\nu.
\)
We work on the usual probability-one event on which this weak convergence
holds, all descendant terminal masses exist simultaneously, and dyadic
endpoints have zero \(\widetilde\nu\)-mass.

\subsection{Parameters and local-mass good events}
\label{SS:parameters-good-events}

We first choose the deterministic parameters used throughout the annular
argument.  Let
\[
    0<\delta_1<\min\{\delta,1-s\},
    \qquad
    a=s+\delta_1.
\]
Then
\(
    s<a<1.
\)
The exponent \(a\) will be used in the local-mass thresholds.  The choices of
\(\varepsilon>0\) and \(\kappa>0\) are made in
Lemma~\ref{lemma:annular-parameter-choice} below.

For \(k\ge0\) and \(n\ge1\), define
\[
    T_{k,n}=2^{\varepsilon n}2^{-ak},
    \qquad
    L_n=2^{\kappa n}.
\]
The quantity \(T_{k,n}\) is the local-mass threshold at dyadic scale \(k\) and
annular scale \(n\).  The quantity \(L_n\) is the cap-growth factor used in the
predictable capping step.

For the \(r\)-tail compensator estimates, introduce
\[
    \beta_{\mathrm{comp}}
    =
    r(s+\delta)-(r-1)a
    =
    s+r\delta-(r-1)\delta_1,
\qquad
    \theta_{\mathrm{comp}}
    =
    \min\{\beta_{\mathrm{comp}},1\}.
\]
With the above choice of \(\delta_1\), one has
\[
    \beta_{\mathrm{comp}}>s,
    \qquad
    \theta_{\mathrm{comp}}>s.
\]
These two exponents are introduced for the \(r\)-tail compensator estimates.

\begin{lemma}
\label{lemma:annular-parameter-choice}
Let \(0<\delta_1<\min\{\delta,1-s\}\), and set \(a=s+\delta_1\).  Then
\(\varepsilon,\kappa>0\) can be chosen so that
\[
    20\varepsilon+\kappa<\frac{\delta_1}{2},
\]
and, with
\begin{equation}
\label{eq:vartheta-def}
    \vartheta=1-\frac{s}{2a}-\frac{8\varepsilon}{a},
\end{equation}
one has
\[
    \vartheta>\frac{s}{2},
    \qquad
    \beta_{\mathrm{comp}}\vartheta>\frac{s}{2},
    \qquad
    \theta_{\mathrm{comp}}\vartheta>\frac{s}{2}.
\]
\end{lemma}

\begin{proof}
Since \(a=s+\delta_1\), we have \(s<a<1\).  Since \(\delta_1<\delta\),
\[
    \beta_{\mathrm{comp}}
    =
    s+r\delta-(r-1)\delta_1>s,
    \qquad
    \theta_{\mathrm{comp}}=\min\{\beta_{\mathrm{comp}},1\}>s.
\]
For \(\varepsilon=0\), put
\(
    \vartheta_0=1-\frac{s}{2a}.
\)
Since \(a>s\) and \(s<1\), we have
\(
    \vartheta_0>\frac12>\frac{s}{2}.
\)
Thus
\[
    \beta_{\mathrm{comp}}\vartheta_0>\frac{s}{2},
    \qquad
    \theta_{\mathrm{comp}}\vartheta_0>\frac{s}{2}.
\]
By continuity, choose \(\varepsilon>0\) sufficiently small so that
\[
    \vartheta
    =
    1-\frac{s}{2a}-\frac{8\varepsilon}{a}
    >
    \frac{s}{2},
\]
and the two strict inequalities
\[
    \beta_{\mathrm{comp}}\vartheta>\frac{s}{2},
    \qquad
    \theta_{\mathrm{comp}}\vartheta>\frac{s}{2}
\]
still hold.  We also choose \(\varepsilon>0\) sufficiently small so that
\(
    20\varepsilon<\frac{\delta_1}{2}.
\)
Finally, choose
\(
    0<\kappa<\frac{\delta_1}{2}-20\varepsilon.
\)
Then \(20\varepsilon+\kappa<\delta_1/2\), as required.
\end{proof}

We record the two \(r\)-moment estimates used to construct the local-mass good
events.

\begin{lemma}
\label{L:arc-terminal-r-moment}
Assume the finite-\(r\) hypothesis
\[
    2^{1-r}\mathbb E[U^r]
    \le
    2^{-r(s+\delta)}
    <1 .
\]
Let
\(
    Y_\ell=\widetilde\nu_\ell([0,1]),
\)
\(
    Y=\widetilde\nu([0,1]).
\)
Then there exists \(C_r<\infty\), depending only on \(r\) and the law of \(U\),
such that
\[
    \sup_{\ell\ge0}\mathbb E[Y_\ell^r]\le C_r.
\]
Consequently,
\(
    \mathbb E[Y^r]\le C_r.
\)
The same bound holds for all descendant terminal masses.
\end{lemma}

\begin{proof}
Let \(U_0\) and \(U_1\) be independent copies of \(U\), and define
\(A_i=\frac{U_i}{2}\), for \(i=0,1\).
Then
\(
\mathbb E(A_0+A_1)=1.
\)
Moreover, the finite-\(r\) hypothesis yields
\[
    \rho_r
    :=
    \mathbb E\!\left[A_0^r+A_1^r\right]
    =
    2^{1-r}\mathbb E[U^r]
    \le
    2^{-r(s+\delta)}
    <1.
\]
The terminal masses satisfy the branching identity
\[
    Y_{\ell+1}
    \stackrel{\mathrm d}{=}
    A_0Y_\ell^{(0)}+A_1Y_\ell^{(1)},
\]
where \(Y_\ell^{(0)}\) and \(Y_\ell^{(1)}\) are independent copies of
\(Y_\ell\), independent of \(A_0,A_1\).  Moreover,
\(\mathbb E Y_\ell=1\)
for
\(\ell\ge0\).

For \(x,y\ge0\), there exists \(C_r<\infty\) such that
\[
    (x+y)^r
    \le
    x^r+y^r+C_r\bigl(xy^{r-1}+yx^{r-1}\bigr).
\]
Applying this with \(x=A_0Y_\ell^{(0)},\) \(y=A_1Y_\ell^{(1)},\)
and using independence, we get
\[
    \mathbb E[Y_{\ell+1}^r]
    \le
    \rho_r\mathbb E[Y_\ell^r]
    +
    C_r
    \mathbb E\!\left[A_0A_1^{r-1}+A_1A_0^{r-1}\right]
    \mathbb E[Y_\ell]\,
    \mathbb E[Y_\ell^{r-1}].
\]
By Young's inequality,
\[
    \mathbb E\!\left[A_0A_1^{r-1}+A_1A_0^{r-1}\right]
    <\infty,
\]
and by H\"older's inequality,
\(
   \mathbb E[Y_\ell^{r-1}]
    \le
    \bigl(\mathbb E[Y_\ell^r]\bigr)^{(r-1)/r}.
\)
Thus, writing
\(
    q_\ell=\mathbb E[Y_\ell^r],
\)
we obtain
\[
    q_{\ell+1}
    \le
    \rho_rq_\ell+C_rq_\ell^{(r-1)/r}.
\]
This recursion implies \(\sup_{\ell\ge0}q_\ell<\infty\).  Indeed, for
\(q_\ell\) sufficiently large, the term \(C_rq_\ell^{(r-1)/r}\) is at most
\((1-\rho_r)q_\ell/2\), and hence
\(
q_{\ell+1}\le \frac{1+\rho_r}{2}q_\ell<q_\ell.
\)
Therefore the sequence \((q_\ell)_{\ell\ge0}\) is bounded, and
\(
    \sup_{\ell\ge0}\mathbb E[Y_\ell^r]\le C_r.
\)
Since \(Y_\ell\to Y\) almost surely, Fatou's lemma gives
\[
    \mathbb E[Y^r]
    \le
    \liminf_{\ell\to\infty}\mathbb E[Y_\ell^r]
    \le C_r.
\]
Finally, each descendant cascade has the same law as the original cascade, 
so the same estimate holds for all descendant terminal masses.
\end{proof}

\begin{lemma}
\label{L:arc-r-mass-budgets}
There exists \(C_r<\infty\) such that, for all \(0\le k\le \ell\),
\[
    \mathbb E\left[
        \sum_{I\in\mathcal D_k([0,1])}
        \widetilde\nu_\ell(I)^r
    \right]
    \le
    C_r2^{-r(s+\delta)k}.
\]
Moreover, for all \(k\ge0\),
\[
    \mathbb E\left[
        \sum_{I\in\mathcal D_k([0,1])}
        \widetilde\nu(I)^r
    \right]
    \le
    C_r2^{-r(s+\delta)k}.
\]
\end{lemma}

\begin{proof}
For a dyadic word \(v\in\Sigma_k\), write
\(
Q_v=\prod_{j=1}^kU_{v|j}.
\)
On the simultaneous descendant-convergence event, for every \(\ell\ge k\),
\(
   \widetilde\nu_\ell(I_v)
    =
    2^{-k}Q_vY_{\ell-k}^{(v)},
\)
where \(Y_{\ell-k}^{(v)}\) is the level-\((\ell-k)\) descendant total mass
rooted at \(v\).  
Moreover,
\(
\widetilde\nu(I_v)=2^{-k}Q_vY^{(v)},
\)
where \(Y^{(v)}\) is the limiting descendant total mass rooted at \(v\).  
The descendant masses are independent of \(Q_v\) and have the same laws as the corresponding terminal masses of the original cascade.

Thus, by Lemma~\ref{L:arc-terminal-r-moment},
\[
    \mathbb E[\widetilde\nu_\ell(I_v)^r]
    =
    2^{-kr}\mathbb E[Q_v^r]\,
    \mathbb E[(Y_{\ell-k}^{(v)})^r]
    \le
    C_r2^{-kr}(\mathbb E[U^r])^k.
\]
Summing over the \(2^k\) words \(v\in\Sigma_k\) and using the finite-\(r\)
hypothesis gives
\[
    \mathbb E\left[
        \sum_{I\in\mathcal D_k([0,1])}
        \widetilde\nu_\ell(I)^r
    \right]
    \le
    C_r\bigl(2^{1-r}\mathbb E[U^r]\bigr)^k
    \le
    C_r2^{-r(s+\delta)k}.
\]

The limiting estimate is identical.  Indeed, using
\(
    \widetilde\nu(I_v)=2^{-k}Q_vY^{(v)}
\)
and
\(
    \mathbb E[(Y^{(v)})^r]\le C_r,
\)
we obtain
\[
    \mathbb E[\widetilde\nu(I_v)^r]
    \le
    C_r2^{-kr}(\mathbb E[U^r])^k.
\]
Summing over \(v\in\Sigma_k\) and using the same finite-\(r\) hypothesis gives
\[
    \mathbb E\left[
        \sum_{I\in\mathcal D_k([0,1])}
        \widetilde\nu(I)^r
    \right]
    \le
    C_r2^{-r(s+\delta)k}.
\]
\end{proof}

\begin{definition}
\label{D:arc-local-mass-good-events}
For \(n\ge1\), let \(\mathcal G_n^{\mathrm{pre}}\) be the event that, for every
\(k\ge0\), \(I\in\mathcal D_k([0,1])\),\(\ell\ge k\),
one has
\[
    \widetilde\nu_\ell(I)\le T_{k,n}.
\]
Let \(\mathcal G_n^{\mathrm{lim}}\) be the event that, for every
\(k\ge0\), \(I\in\mathcal D_k([0,1])\),
one has
\[
    \widetilde\nu(I)\le T_{k,n}.
\]
Set
\(
    \mathcal G_n
    =
    \mathcal G_n^{\mathrm{pre}}
    \cap
    \mathcal G_n^{\mathrm{lim}}.
\)
\end{definition}

\begin{proposition}
\label{P:arc-local-mass-good-events}
There are constants \(C<\infty\) and \(c>0\) such that, for all \(n\ge1\),
\[
    \mathbb P(\mathcal G_n^c)\le C2^{-cn}.
\]
In particular,
\(
\sum_{n=1}^\infty\mathbb P(\mathcal G_n^c)<\infty.
\)
\end{proposition}

\begin{proof}
Fix \(k\ge0\) and \(I\in\mathcal D_k([0,1])\).  The process
\(
    \bigl(\widetilde\nu_\ell(I)\bigr)_{\ell\ge k}
\)
is a nonnegative martingale with terminal value \(\widetilde\nu(I)\).  By
Doob's \(L^r\) maximal inequality,
\[
    \mathbb E\left[
        \sup_{\ell\ge k}\widetilde\nu_\ell(I)^r
    \right]
    \le
    C_r\mathbb E[\widetilde\nu(I)^r].
\]
Therefore, by Markov's inequality and Lemma~\ref{L:arc-r-mass-budgets},
\[
\begin{aligned}
    \mathbb P\left(
        \exists I\in\mathcal D_k([0,1]),\ 
        \exists \ell\ge k:
        \widetilde\nu_\ell(I)>T_{k,n}
    \right)                                               \le
    T_{k,n}^{-r}
    \mathbb E\left[
        \sum_{I\in\mathcal D_k([0,1])}
        \sup_{\ell\ge k}\widetilde\nu_\ell(I)^r
    \right]                                               \le
    C T_{k,n}^{-r}2^{-r(s+\delta)k}.
\end{aligned}
\]
Since
\(
T_{k,n}=2^{\varepsilon n}2^{-ak},
\)
and
\(
a=s+\delta_1,
\)
the right-hand side is
\(
    C2^{-r\varepsilon n}2^{-r(\delta-\delta_1)k}.
\)
Summing over \(k\ge0\), and using \(\delta_1<\delta\), gives
\[
    \mathbb P\bigl((\mathcal G_n^{\mathrm{pre}})^c\bigr)
    \le
    C2^{-r\varepsilon n}.
\]

The limiting event is handled in the same way, without Doob's maximal
inequality.  By Markov's inequality and the limiting \(r\)-mass budget,
\[
    \mathbb P\left(
        \exists I\in\mathcal D_k([0,1]):
        \widetilde\nu(I)>T_{k,n}
    \right)
    \le
    C2^{-r\varepsilon n}2^{-r(\delta-\delta_1)k}.
\]
Summing over \(k\ge0\) gives
\(
    \mathbb P\bigl((\mathcal G_n^{\mathrm{lim}})^c\bigr)
    \le
    C2^{-r\varepsilon n}.
\)
Thus
\(
    \mathbb P(\mathcal G_n^c)
    \le
    C2^{-r\varepsilon n},
\)
which is the desired estimate after decreasing \(c>0\).
\end{proof}

\begin{lemma}
\label{L:arc-local-mass-consequences}
Assume \(\mathcal G_n\) holds.  Then, with a constant \(C<\infty\) independent
of \(n\), the following estimates hold.

\begin{enumerate}[label=\textup{(\roman*)}]
\item For every interval \(B\subset[0,1]\),
\[
    \widetilde\nu(B)
    \le
    C2^{\varepsilon n}|B|^a.
\]

\item For every \(\ell\ge0\) and every interval \(B\subset[0,1]\),
\[
    \widetilde\nu_\ell(B)
    \le
    C2^{\varepsilon n}(|B|+2^{-\ell})^a.
\]

\item For every \(\ell\ge0\) and every interval \(B\subset[0,1]\),
\[
    \sum_{\substack{
        I\in\mathcal D_\ell([0,1])\\
        I\cap B\ne\varnothing
    }}
    \widetilde\nu_\ell(I)^2
    \le
    C2^{2\varepsilon n}2^{-a\ell}
    (|B|+2^{-\ell})^a.
\]
\end{enumerate}

In particular,
\(
    \widetilde\nu([0,1])\le 2^{\varepsilon n}.
\)
\end{lemma}

\begin{proof}
We give the standard covering argument.  If \(0<|B|<1\), choose \(k\ge0\) with
\(
 2^{-(k+1)}<|B|\le2^{-k}.
\)
The interval \(B\) is covered by at most a bounded number of dyadic intervals
in \(\mathcal D_k([0,1])\).  On \(\mathcal G_n^{\mathrm{lim}}\), each has mass
at most \(T_{k,n}\).  Hence
\[
    \widetilde\nu(B)
    \le
    C2^{\varepsilon n}2^{-ak}
    \le
    C2^{\varepsilon n}|B|^a.
\]
The cases \(|B|=0\) and \(B=[0,1]\) are immediate: the former follows by shrinking dyadic intervals and the bound
\(\widetilde\nu(I_k)\le2^{\varepsilon n}2^{-ak}\to0\), 
while the latter is the \(k=0\) case.

For the prelimit estimate, put
\(
    \rho=|B|+2^{-\ell}.
\)
If \(\rho\ge1\), the claim follows from \(T_{0,n}=2^{\varepsilon n}\).  If
\(0<\rho<1\), choose \(k\le\ell\) with
\(
    2^{-(k+1)}<\rho\le2^{-k}.
\)
The union of level-\(\ell\) dyadic intervals meeting \(B\) is contained in a bounded number of level-\(k\) dyadic intervals.  

On \(\mathcal G_n^{\mathrm{pre}}\), 
each has \(\widetilde\nu_\ell\)-mass at most \(T_{k,n}\), giving
\(
    \widetilde\nu_\ell(B)
    \le
    C2^{\varepsilon n}\rho^a.
\)

Finally, on \(\mathcal G_n^{\mathrm{pre}}\), every level-\(\ell\) dyadic
interval has mass at most
\(
    T_{\ell,n}=2^{\varepsilon n}2^{-a\ell}.
\)
Thus
\[
\begin{aligned}
    \sum_{\substack{
        I\in\mathcal D_\ell([0,1])\\
        I\cap B\ne\varnothing
    }}
    \widetilde\nu_\ell(I)^2
    &\le
    2^{\varepsilon n}2^{-a\ell}
    \sum_{\substack{
        I\in\mathcal D_\ell([0,1])\\
        I\cap B\ne\varnothing
    }}
    \widetilde\nu_\ell(I).
\end{aligned}
\]
The last sum is controlled by applying the prelimit interval estimate to an
enlarged interval of length \(C(|B|+2^{-\ell})\).  This gives
\[
    \sum_{\substack{
        I\in\mathcal D_\ell([0,1])\\
        I\cap B\ne\varnothing
    }}
    \widetilde\nu_\ell(I)^2
    \le
    C2^{2\varepsilon n}2^{-a\ell}
    (|B|+2^{-\ell})^a.
\]
\end{proof}

We also use a stopped version of the preceding prelimit estimates.  Define
\[
    \tau_n
    =
    \inf
    \left\{
        \ell\ge0:
        \exists\,0\le k\le\ell,\ 
        \exists I\in\mathcal D_k([0,1])
        \text{ such that }
        \widetilde\nu_\ell(I)>T_{k,n}
    \right\},
\]
with the convention \(\inf\varnothing=+\infty\).  Then \(\tau_n\) is a stopping
time and
\[
    \mathcal G_n^{\mathrm{pre}}=\{\tau_n=+\infty\}.
\]

\begin{lemma}
\label{L:arc-stopped-local-mass-consequences}
Let \(n\ge1\) and \(\ell\ge0\).  On the event \(\{\tau_n>\ell\}\), the
following estimates hold for every interval \(B\subset[0,1]\):
\[
    \widetilde\nu_\ell(B)
    \le
    C2^{\varepsilon n}(|B|+2^{-\ell})^a,
\]
and
\[
    \sum_{\substack{
        I\in\mathcal D_\ell([0,1])\\
        I\cap B\ne\varnothing
    }}
    \widetilde\nu_\ell(I)^2
    \le
    C2^{2\varepsilon n}2^{-a\ell}
    (|B|+2^{-\ell})^a.
\]
\end{lemma}

\begin{proof}
On \(\{\tau_n>\ell\}\), all dyadic masses \(\widetilde\nu_\ell(I)\) with
\(I\in\mathcal D_k([0,1])\) and \(0\le k\le\ell\) satisfy
\(
\widetilde\nu_\ell(I)\le T_{k,n}.
\)
Therefore the covering argument used in the proof of
Lemma~\ref{L:arc-local-mass-consequences} applies at the fixed level \(\ell\)
with \(\mathcal G_n^{\mathrm{pre}}\) replaced by the finite-level information
available on \(\{\tau_n>\ell\}\).  This gives both estimates.
\end{proof}

\subsection{Deterministic annular reduction}
\label{SS:deterministic-annular-reduction}

We now reduce the Fourier transform on a fixed annulus to a safe error plus
exact martingale arrays.  Fix \(n\ge1\) and \(\xi\in\mathcal A_n\).  Recall
from Proposition~\ref{P:arc-endpoint-safe-phase-bins} that
\[
    \chi_{\xi,0}
    +
    \chi_{\xi,1}
    +
    \chi_{\xi,\mathrm{sd}}
    +
    \sum_{d\in\mathfrak D_\xi}\chi_{\xi,d}
    =
    1
    \qquad\text{on }[0,1].
\]
Recall also that the safe region is
\[
    E_\xi
    =
    \operatorname{spt}\chi_{\xi,0}
    \cup
    \operatorname{spt}\chi_{\xi,1}
    \cup
    \operatorname{spt}\chi_{\xi,\mathrm{sd}}.
\]

\begin{definition}
\label{D:arc-oscillatory-mass-only-bands}
A derivative band \(d\in\mathfrak D_\xi\) is called mass-only at annular scale
\(n\) if
\[
    d^a\le 2^{-sn/2}2^{-8\varepsilon n}.
\]
It is called oscillatory, or non-mass, if
\[
    d^a>2^{-sn/2}2^{-8\varepsilon n}.
\]
The corresponding families are denoted by
\(\mathfrak D_{\xi,n}^{\mathrm{mass}}\)
and
\(\mathfrak D_{\xi,n}^{\mathrm{osc}}\).
\end{definition}

\begin{lemma}
\label{L:arc-safe-mass-only-estimates}
Assume \(\mathcal G_n\) holds and \(\xi\in\mathcal A_n\).  Then there exists
\(c_\gamma>0\) such that
\[
    \widetilde\nu(E_\xi)
    \le
    C_\gamma2^{-sn/2}2^{-c_\gamma n},
\]
and
\[
    \sum_{d\in\mathfrak D_{\xi,n}^{\mathrm{mass}}}
    \left|
        \int_0^1
        e^{i\varphi_\xi(t)}
        \chi_{\xi,d}(t)\,d\widetilde\nu(t)
    \right|
    \le
    C_\gamma2^{-sn/2}2^{-6\varepsilon n}.
\]
\end{lemma}

\begin{proof}
The endpoint pieces are supported in two endpoint intervals of length
\(O(|\xi|^{-1/2})\).  Moreover,
\[
    \operatorname{spt}\chi_{\xi,\mathrm{sd}}
    \subset
    \left\{
        t\in[0,1]:
        h_\xi(t)^{1/2}\le 2|\xi|^{-1/2}
    \right\}.
\]
By Lemma~\ref{L:arc-g-level-geometry}, this sublevel set is contained in a
union of at most \(C_\gamma\) intervals of total length at most
\(C_\gamma|\xi|^{-1/2}\).  Hence \(E_\xi\) is contained in a union of at most
\(C_\gamma\) intervals of total length at most \(C_\gamma|\xi|^{-1/2}\).

Applying Lemma~\ref{L:arc-local-mass-consequences} to this finite union of
intervals gives
\(
    \widetilde\nu(E_\xi)
    \le
    C_\gamma2^{\varepsilon n}|\xi|^{-a/2}.
\)
Since \(\xi\in\mathcal A_n\), we have \(|\xi|\ge2^n\), and therefore
\[
    \widetilde\nu(E_\xi)
    \le
    C_\gamma2^{\varepsilon n}2^{-an/2}
    =
    C_\gamma2^{-sn/2}2^{-(\delta_1/2-\varepsilon)n}.
\]
The margin \(\delta_1/2-\varepsilon>0\) gives the first estimate, after
renaming the positive exponent as \(c_\gamma\).

Now let \(d\in\mathfrak D_{\xi,n}^{\mathrm{mass}}\), and recall that
\(
    S_{\xi,d}=\operatorname{spt}\chi_{\xi,d}.
\)
By the localization property in Proposition~\ref{P:arc-endpoint-safe-phase-bins}
and the level-set estimate in Lemma~\ref{L:arc-g-level-geometry}, the set
\(S_{\xi,d}\) is contained in a union of at most \(C_\gamma\) intervals of
total length at most \(C_\gamma d\).  Applying
Lemma~\ref{L:arc-local-mass-consequences} to the components of this cover and
summing over their bounded number gives, on \(\mathcal G_n\),
\(
    \widetilde\nu(S_{\xi,d})
    \le
    C_\gamma2^{\varepsilon n}d^a.
\)
Since \(d\) is mass-only,
\(
    d^a\le 2^{-sn/2}2^{-8\varepsilon n},
\)
and therefore
\[
    \left|
        \int_0^1
        e^{i\varphi_\xi(t)}
        \chi_{\xi,d}(t)\,d\widetilde\nu(t)
    \right|
    \le
    \widetilde\nu(S_{\xi,d})
    \le
    C_\gamma2^{-sn/2}2^{-7\varepsilon n}.
\]
There are at most \(C_\gamma(1+n)\) active derivative bands in the annulus.
Since \(1+n\le C_\varepsilon2^{\varepsilon n}\), summing over all mass-only
bands gives the stated \(2^{-6\varepsilon n}\) bound.
\end{proof}

\begin{lemma}
\label{L:arc-oscillatory-range-arclength}
Let \(\vartheta\) be the parameter fixed in \eqref{eq:vartheta-def}.  Assume
\(\xi\in\mathcal A_n\) and \(d\in\mathfrak D_{\xi,n}^{\mathrm{osc}}\).  Then
\[
    \vartheta n-C_\gamma\le m_{\xi,d}\le n+C_\gamma.
\]
Moreover,
\[
    \left|
        \int_0^1
        e^{i\varphi_\xi(t)}
        \chi_{\xi,d}(t)\,dt
    \right|
    \le
    C_\gamma2^{-sn/2}2^{-c_\gamma n}
\]
for some \(c_\gamma>0\).  After possibly decreasing \(c_\gamma\), the same
bound holds after summing over all \(d\in\mathfrak D_{\xi,n}^{\mathrm{osc}}\).
\end{lemma}

\begin{proof}
Since \(d\in\mathfrak D_{\xi,n}^{\mathrm{osc}}\), Definition
\ref{D:arc-oscillatory-mass-only-bands} gives
\(
    d>2^{-sn/(2a)}2^{-8\varepsilon n/a}.
\)
Together with \(|\xi|\ge2^n\), this yields
\[
    |\xi|d
    >
    2^n2^{-sn/(2a)}2^{-8\varepsilon n/a}
    =
    2^{\vartheta n}.
\]
Since
\(
    2^{m_{\xi,d}-1}<|\xi|d\le2^{m_{\xi,d}},
\)
we obtain
\(
    m_{\xi,d}\ge \vartheta n-C_\gamma.
\)

For the upper bound, Proposition~\ref{P:arc-endpoint-safe-phase-bins} gives
\(d\le C_\gamma\).  Together with \(|\xi|\le2^{n+1}\), this gives
\(
    |\xi|d\le C_\gamma2^n,
\)
and hence
\(
    m_{\xi,d}\le n+C_\gamma.
\)

We now prove the Lebesgue forcing estimate.  By
Proposition~\ref{P:arc-phasebin-coefficient-estimates},
\[
    \left|
        \int_0^1
        e^{i\varphi_\xi(t)}
        \chi_{\xi,d}(t)\,dt
    \right|
    \le
    C_\gamma(|\xi|d)^{-1}.
\]
Since \(d\) is oscillatory,
\(
    d^{-1}<2^{sn/(2a)}2^{8\varepsilon n/a}.
\)
Together with \(|\xi|\ge2^n\), this gives
\[
    (|\xi|d)^{-1}
    \le
    2^{-n}2^{sn/(2a)}2^{8\varepsilon n/a}
    =
    2^{-\vartheta n}.
\]
Since \(\vartheta>s/2\), we obtain
\[
    2^{-\vartheta n}
    =
    2^{-sn/2}2^{-(\vartheta-s/2)n}.
\]
Thus the desired bound holds for a suitable \(c_\gamma>0\).

The sum over \(d\in\mathfrak D_{\xi,n}^{\mathrm{osc}}\) loses only a polynomial
factor \(C_\gamma(1+n)\), which is absorbed into the exponential decay after
decreasing \(c_\gamma>0\) if necessary.
\end{proof}

We now record the exact martingale identity.  Recall from
Subsection~\ref{SS:phase-notation-conventions} that if
\(I\in\mathcal D_\ell([0,1])\),
\(J\in\mathcal D_{\ell+1}([0,1])\),
\(J\subset I\),
and
\(M_I=\widetilde\nu_\ell(I),\)
then
\[
    \widetilde\nu_{\ell+1}(J)=\frac12M_IU_J.
\]
With the coefficient convention from
Subsection~\ref{SS:phase-notation-conventions}, the raw martingale increment on
the child interval \(J\) is
\(
c_J(\xi,d,\ell)M_I(U_J-1).
\)
There is no additional factor \(1/2\) in this increment.

\begin{definition}
\label{D:arc-pre-post-critical-pieces}
For \(d\in\mathfrak D_\xi\), recall that
\(
    m_{\xi,d,+}=\max\{m_{\xi,d},0\}.
\)
Set
\[
    F_{\xi,d}^{\mathrm{pre}}
    =
    \sum_{\ell=0}^{m_{\xi,d,+}-1}
    \sum_{I\in\mathcal D_\ell([0,1])}
    \sum_{\substack{
        J\in\mathcal D_{\ell+1}([0,1])\\
        J\subset I
    }}
    c_J(\xi,d,\ell)M_I(U_J-1),
\]
and, for every integer \(L\ge1\),
\[
    F_{\xi,d,L}^{\mathrm{post}}
    =
    \sum_{\ell=m_{\xi,d,+}}^{L-1}
    \sum_{I\in\mathcal D_\ell([0,1])}
    \sum_{\substack{
        J\in\mathcal D_{\ell+1}([0,1])\\
        J\subset I
    }}
    c_J(\xi,d,\ell)M_I(U_J-1),
\]
with the convention that the sum is \(0\) when \(L\le m_{\xi,d,+}\).

Whenever the limit exists, define
\(
    F_{\xi,d}^{\mathrm{post}}
    =
    \lim_{L\to\infty}F_{\xi,d,L}^{\mathrm{post}}.
\)
\end{definition}

\begin{lemma}
\label{L:arc-exact-martingale-array-identity}
Let \(d\in\mathfrak D_\xi\).  On the weak-convergence event,
\(F_{\xi,d}^{\mathrm{post}}\) exists and
\[
    \int_0^1
    e^{i\varphi_\xi(t)}
    \chi_{\xi,d}(t)\,d\widetilde\nu(t)
    =
    \int_0^1
    e^{i\varphi_\xi(t)}
    \chi_{\xi,d}(t)\,dt
    +
    F_{\xi,d}^{\mathrm{pre}}
    +
    F_{\xi,d}^{\mathrm{post}}.
\]
\end{lemma}

\begin{proof}
Let
\(
    G_{\xi,d}(t)=e^{i\varphi_\xi(t)}\chi_{\xi,d}(t).
\)
For finite \(L\ge1\),
\[
    \int G_{\xi,d}\,d\widetilde\nu_L
    =
    \int G_{\xi,d}(t)\,dt
    +
    \sum_{\ell=0}^{L-1}
    \int G_{\xi,d}\,d(\widetilde\nu_{\ell+1}-\widetilde\nu_\ell).
\]
On a child interval \(J\subset I\), the density of \(\widetilde\nu_\ell\) is
\(2^\ell M_I\), while the density of \(\widetilde\nu_{\ell+1}\) is
\(2^\ell M_IU_J\).  Hence
\[
    \int_J
    G_{\xi,d}\,d(\widetilde\nu_{\ell+1}-\widetilde\nu_\ell)
    =
    M_I(U_J-1)
    2^\ell
    \int_J G_{\xi,d}(t)\,dt
    =
    c_J(\xi,d,\ell)M_I(U_J-1).
\]
Splitting the resulting finite sum at \(m_{\xi,d,+}\) gives the finite
identity.  Since \(G_{\xi,d}\) is continuous and
\(\widetilde\nu_L\to\widetilde\nu\) weakly, the left-hand side converges.
Hence the post-bin partial sums converge, and the limiting identity follows.
\end{proof}

\begin{proposition}[Deterministic annular reduction]
\label{P:arc-deterministic-annular-reduction}
Assume that \(\mathcal G_n\) holds, and let \(\xi\in\mathcal A_n\).  On the
weak-convergence event,
\[
    \widehat{\nu_\gamma}(\xi)
    =
    E_{\xi,n}^{\mathrm{safe}}
    +
    \sum_{d\in\mathfrak D_{\xi,n}^{\mathrm{osc}}}
    \left(
        F_{\xi,d}^{\mathrm{pre}}
        +
        F_{\xi,d}^{\mathrm{post}}
    \right),
\]
where
\[
    |E_{\xi,n}^{\mathrm{safe}}|
    \le
    C_\gamma2^{-sn/2}2^{-c_\gamma n}
\]
for some \(c_\gamma>0\).  Moreover, for every
\(d\in\mathfrak D_{\xi,n}^{\mathrm{osc}}\), every \(\ell\ge0\), and every
\(J\in\mathcal D_{\ell+1}([0,1])\), the coefficients satisfy
\[
    |c_J(\xi,d,\ell)|
    \le
    C_\gamma2^{\ell-m_{\xi,d}}
    \qquad (0\le\ell<m_{\xi,d}),
\]
and
\[
    |c_J(\xi,d,\ell)|
    \le C_\gamma
    \qquad (\ell\ge m_{\xi,d}).
\]
\end{proposition}

\begin{proof}
Using the partition of unity from
Proposition~\ref{P:arc-endpoint-safe-phase-bins}, we first decompose
\[
    \widehat{\nu_\gamma}(\xi)
    =
    \int_0^1
    e^{i\varphi_\xi(t)}
    \bigl(
        \chi_{\xi,0}(t)
        +
        \chi_{\xi,1}(t)
        +
        \chi_{\xi,\mathrm{sd}}(t)
    \bigr)\,d\widetilde\nu(t)
    +
    \sum_{d\in\mathfrak D_\xi}
    \int_0^1
    e^{i\varphi_\xi(t)}
    \chi_{\xi,d}(t)\,d\widetilde\nu(t).
\]
Split the active derivative-band scales as
\(
    \mathfrak D_\xi
    =
    \mathfrak D_{\xi,n}^{\mathrm{mass}}
    \cup
    \mathfrak D_{\xi,n}^{\mathrm{osc}}.
\)
For each \(d\in\mathfrak D_{\xi,n}^{\mathrm{osc}}\), the exact martingale-array
identity, Lemma~\ref{L:arc-exact-martingale-array-identity}, gives
\[
    \int_0^1
    e^{i\varphi_\xi(t)}
    \chi_{\xi,d}(t)\,d\widetilde\nu(t)
    =
    \int_0^1
    e^{i\varphi_\xi(t)}
    \chi_{\xi,d}(t)\,dt
    +
    F_{\xi,d}^{\mathrm{pre}}
    +
    F_{\xi,d}^{\mathrm{post}}.
\]

Define
\[
\begin{aligned}
    E_{\xi,n}^{\mathrm{safe}}
    &=
    \int_0^1
    e^{i\varphi_\xi(t)}
    \bigl(
        \chi_{\xi,0}(t)
        +
        \chi_{\xi,1}(t)
        +
        \chi_{\xi,\mathrm{sd}}(t)
    \bigr)\,d\widetilde\nu(t)       \\
    &\quad
    +
    \sum_{d\in\mathfrak D_{\xi,n}^{\mathrm{mass}}}
    \int_0^1
    e^{i\varphi_\xi(t)}
    \chi_{\xi,d}(t)\,d\widetilde\nu(t)       \\
    &\quad
    +
    \sum_{d\in\mathfrak D_{\xi,n}^{\mathrm{osc}}}
    \int_0^1
    e^{i\varphi_\xi(t)}
    \chi_{\xi,d}(t)\,dt .
\end{aligned}
\]
With this definition, substituting the martingale-array identity for all
oscillatory bands gives
\[
    \widehat{\nu_\gamma}(\xi)
    =
    E_{\xi,n}^{\mathrm{safe}}
    +
    \sum_{d\in\mathfrak D_{\xi,n}^{\mathrm{osc}}}
    \left(
        F_{\xi,d}^{\mathrm{pre}}
        +
        F_{\xi,d}^{\mathrm{post}}
    \right).
\]

It remains to estimate \(E_{\xi,n}^{\mathrm{safe}}\).  By
Lemma~\ref{L:arc-safe-mass-only-estimates}, the endpoint-safe,
small-derivative, and mass-only contributions satisfy
\[
\begin{aligned}
    &
    \left|
    \int_0^1
    e^{i\varphi_\xi(t)}
    \bigl(
        \chi_{\xi,0}(t)
        +
        \chi_{\xi,1}(t)
        +
        \chi_{\xi,\mathrm{sd}}(t)
    \bigr)\,d\widetilde\nu(t)
    \right|                                      \\
    &\qquad
    +
    \sum_{d\in\mathfrak D_{\xi,n}^{\mathrm{mass}}}
    \left|
        \int_0^1
        e^{i\varphi_\xi(t)}
        \chi_{\xi,d}(t)\,d\widetilde\nu(t)
    \right|
    \le
    C_\gamma2^{-sn/2}2^{-c_\gamma n}.
\end{aligned}
\]
By Lemma~\ref{L:arc-oscillatory-range-arclength}, the oscillatory Lebesgue forcing terms satisfy
\[
    \sum_{d\in\mathfrak D_{\xi,n}^{\mathrm{osc}}}
    \left|
        \int_0^1
        e^{i\varphi_\xi(t)}
        \chi_{\xi,d}(t)\,dt
    \right|
    \le
    C_\gamma2^{-sn/2}2^{-c_\gamma n}.
\]
After decreasing \(c_\gamma>0\) if necessary, these two estimates imply
\[
    |E_{\xi,n}^{\mathrm{safe}}|
    \le
    C_\gamma2^{-sn/2}2^{-c_\gamma n}.
\]

Finally, the coefficient bounds are exactly those of
Proposition~\ref{P:arc-phasebin-coefficient-estimates}.  This completes the
proof.
\end{proof}

\subsection{Predictable capping and Freedman concentration}
\label{SS:capping-freedman}

We now estimate the martingale arrays in
Proposition~\ref{P:arc-deterministic-annular-reduction}.  Since \(U\) is not
assumed bounded, we cap each fresh child weight by a predictable cap, center the
capped increment, and apply Freedman's inequality.  The predictable drift
created by this cap is postponed to Subsection~\ref{SS:r-tail-compensator}.

Throughout this subsection, fix \(n\ge1\), \(\xi\in\mathcal A_n\), and
\(
    d\in\mathfrak D_{\xi,n}^{\mathrm{osc}}.
\)
Recall that
\[
    \mathcal F_\ell=\sigma\{U_v:1\le |v|\le\ell\}
\]
is the cascade filtration up to level \(\ell\).  Recall also from
Subsection~\ref{SS:phase-notation-conventions} and
Definition~\ref{D:arc-pre-post-critical-pieces} that, for
\(I\in\mathcal D_\ell([0,1])\) and a child \(J\subset I\), we write
\(
    M_I=\widetilde\nu_\ell(I),
\)
and \(U_J\) denotes the fresh child weight on the edge \(I\to J\).  With our
coefficient convention, the raw martingale increment is
\(
c_J(\xi,d,\ell)M_I(U_J-1),
\)
with no additional factor \(1/2\).

\begin{definition}
\label{D:arc-predictable-cap}
For
\(I\in\mathcal D_\ell([0,1])\),
\(M_I=\widetilde\nu_\ell(I)\),
define
\[
    C_{I,n}
    =
    \begin{cases}
    \displaystyle
    L_n\frac{2T_{\ell+1,n}}{M_I},
    & M_I>0,\\[6pt]
    +\infty,
    & M_I=0.
    \end{cases}
\]
For a child \(J\subset I\), set
\(
    U_{J,n}^{\mathrm{cap}}=U_J\wedge C_{I,n},
\)
and define
\(
    \overline U_{I,n}
    =
    \mathbb E\left[
        U_{J,n}^{\mathrm{cap}}
        \mid
        \mathcal F_\ell
    \right].
\)
The quantity \(\overline U_{I,n}\) is independent of the choice of child
\(J\subset I\).  Define
\[
    X_{J,n}^{\mathrm{cap}}(\xi,d,\ell)
    =
    c_J(\xi,d,\ell)M_I
    \left(
        U_{J,n}^{\mathrm{cap}}
        -
        \overline U_{I,n}
    \right).
\]
\end{definition}

The cap \(C_{I,n}\) is predictable, since it is \(\mathcal F_\ell\)-measurable.
The factor \(2\) is calibrated to the child-mass identity
\(
    \widetilde\nu_{\ell+1}(J)=\frac12M_IU_J.
\)
Indeed, whenever \(M_I>0\) and \(U_J\le C_{I,n}\), one has
\[
    \widetilde\nu_{\ell+1}(J)
    =
    \frac12M_IU_J
    \le
    L_nT_{\ell+1,n}.
\]
Since \(L_n\ge1\), the cap is inactive whenever the child mass is bounded by
\(T_{\ell+1,n}\).

\begin{lemma}
\label{L:arc-centered-capped-increment-bounds}
Let \(n\ge1\), \(\ell\ge0\),
\(I\in\mathcal D_\ell([0,1])\), and let \(J\subset I\) be a dyadic child.  Then
\[
    \mathbb E\left[
        X_{J,n}^{\mathrm{cap}}(\xi,d,\ell)
        \mid
        \mathcal F_\ell
    \right]
    =
    0.
\]
Moreover,
\[
    |X_{J,n}^{\mathrm{cap}}(\xi,d,\ell)|
    \le
    C|c_J(\xi,d,\ell)|L_nT_{\ell+1,n}
\]
almost surely, and
\[
    \mathbb E\left[
        |X_{J,n}^{\mathrm{cap}}(\xi,d,\ell)|^2
        \mid
        \mathcal F_\ell
    \right]
    \le
    CL_n|c_J(\xi,d,\ell)|^2M_IT_{\ell+1,n}.
\]
\end{lemma}

\begin{proof}
The conditional mean-zero identity follows directly from the definition of
\(\overline U_{I,n}\).  If \(M_I=0\), then
\(X_{J,n}^{\mathrm{cap}}(\xi,d,\ell)=0\), and all assertions are immediate.  We
may therefore assume that \(M_I>0\).

Since
\(
    0\le U_{J,n}^{\mathrm{cap}}\le C_{I,n}
\)
and \(C_{I,n}\) is \(\mathcal F_\ell\)-measurable, we have
\[
    0\le \overline U_{I,n}\le C_{I,n},
    \qquad
    |U_{J,n}^{\mathrm{cap}}-\overline U_{I,n}|
    \le
    C_{I,n}.
\]
Thus
\[
    |X_{J,n}^{\mathrm{cap}}(\xi,d,\ell)|
    \le
    |c_J(\xi,d,\ell)|M_IC_{I,n}
    =
    2|c_J(\xi,d,\ell)|L_nT_{\ell+1,n},
\]
which proves the jump bound.

For the conditional second moment, first note that
\[
    \operatorname{Var}
    \left(
        U_{J,n}^{\mathrm{cap}}
        \mid
        \mathcal F_\ell
    \right)
    \le
    \mathbb E\left[
        (U_{J,n}^{\mathrm{cap}})^2
        \mid
        \mathcal F_\ell
    \right].
\]
Since
\[
    (U_{J,n}^{\mathrm{cap}})^2
    \le
    C_{I,n}U_{J,n}^{\mathrm{cap}}
    \le
    C_{I,n}U_J,
\]
and \(U_J\) is independent of \(\mathcal F_\ell\) with
\(\mathbb E U_J=1\), we get
\(
    \mathbb E\left[
        (U_{J,n}^{\mathrm{cap}})^2
        \mid
        \mathcal F_\ell
    \right]
    \le
    C_{I,n}.
\)
Therefore
\[
\begin{aligned}
    \mathbb E\left[
        |X_{J,n}^{\mathrm{cap}}(\xi,d,\ell)|^2
        \mid
        \mathcal F_\ell
    \right]
    &=
    |c_J(\xi,d,\ell)|^2M_I^2
    \operatorname{Var}
    \left(
        U_{J,n}^{\mathrm{cap}}
        \mid
        \mathcal F_\ell
    \right)                                      \\
    &\le
    |c_J(\xi,d,\ell)|^2M_I^2C_{I,n}              \\
    &=
    2L_n|c_J(\xi,d,\ell)|^2M_IT_{\ell+1,n}.
\end{aligned}
\]
This proves the conditional second-moment estimate.
\end{proof}

\begin{lemma}
\label{L:arc-capping-identity}
Assume \(\mathcal G_n^{\mathrm{pre}}\) holds.  Then, for every
\(I\in\mathcal D_\ell([0,1])\) and every child \(J\subset I\),
\(
 U_{J,n}^{\mathrm{cap}}=U_J.
\)
Consequently,
\[
    c_J(\xi,d,\ell)M_I(U_J-1)
    =
    X_{J,n}^{\mathrm{cap}}(\xi,d,\ell)
    +
    D_{J,n}(\xi,d,\ell),
\]
where
\[
    D_{J,n}(\xi,d,\ell)
    =
    c_J(\xi,d,\ell)M_I(\overline U_{I,n}-1)
\]
and equivalently
\[
    D_{J,n}(\xi,d,\ell)
    =
    -
    c_J(\xi,d,\ell)M_I
    \mathbb E\left[
        (U_J-C_{I,n})_+
        \mid
        \mathcal F_\ell
    \right].
\]
Moreover, for a child \(J\in\mathcal D_{\ell+1}([0,1])\), the same cap
inactivity \(U_{J,n}^{\mathrm{cap}}=U_J\) holds on the finite-level event
\(
    \{\tau_n>\ell+1\}.
\)
\end{lemma}

\begin{proof}
If \(M_I=0\), the identity is immediate.  Assume \(M_I>0\).  On
\(\mathcal G_n^{\mathrm{pre}}\),
\(
\widetilde\nu_{\ell+1}(J)\le T_{\ell+1,n}.
\)
Since
\(
    \widetilde\nu_{\ell+1}(J)=\frac12M_IU_J,
\)
we obtain
\[
    U_J
    \le
    \frac{2T_{\ell+1,n}}{M_I}
    \le
    L_n\frac{2T_{\ell+1,n}}{M_I}
    =
    C_{I,n}.
\]
Thus \(U_{J,n}^{\mathrm{cap}}=U_J\).  The decomposition then follows by adding
and subtracting \(\overline U_{I,n}\):
\[
    c_JM_I(U_J-1)
    =
    c_JM_I(U_J-\overline U_{I,n})
    +
    c_JM_I(\overline U_{I,n}-1)
    =
    X_{J,n}^{\mathrm{cap}}(\xi,d,\ell)
    +
    D_{J,n}(\xi,d,\ell).
\]
Finally,
\[
    \overline U_{I,n}
    =
    \mathbb E[U_J\wedge C_{I,n}\mid\mathcal F_\ell]
    =
    1-
    \mathbb E\left[
        (U_J-C_{I,n})_+
        \mid
        \mathcal F_\ell
    \right],
\]
which gives the displayed formula for \(D_{J,n}\).

The proof of the finite-level statement is identical.  On
\(\{\tau_n>\ell+1\}\), the bound
\(
    \widetilde\nu_{\ell+1}(J)\le T_{\ell+1,n}
\)
holds for every level-\((\ell+1)\) child \(J\), and hence the same argument
shows \(U_{J,n}^{\mathrm{cap}}=U_J\).
\end{proof}

\begin{lemma}
\label{L:arc-stopped-square-sum-budgets}
There exists \(C_\gamma<\infty\) such that, for every \(n\ge1\), every
\(\xi\in\mathcal A_n\), and every
\(
    d\in\mathfrak D_{\xi,n}^{\mathrm{osc}},
\)
if
\(m=m_{\xi,d}\),
\(m_+=m_{\xi,d,+}\),
then
\[
    \sum_{\ell=0}^{m_+-1}
    \mathbf 1_{\{\tau_n>\ell\}}
    \sum_{I\in\mathcal D_\ell([0,1])}
    \sum_{\substack{
        J\in\mathcal D_{\ell+1}([0,1])\\
        J\subset I
    }}
    |c_J(\xi,d,\ell)|^2M_IT_{\ell+1,n}
    \le
    C_\gamma2^{2\varepsilon n}|\xi|^{-a},
\]
and
\[
    \sup_{L>m_+}
    \sum_{\ell=m_+}^{L-1}
    \mathbf 1_{\{\tau_n>\ell\}}
    \sum_{I\in\mathcal D_\ell([0,1])}
    \sum_{\substack{
        J\in\mathcal D_{\ell+1}([0,1])\\
        J\subset I
    }}
    |c_J(\xi,d,\ell)|^2M_IT_{\ell+1,n}
    \le
    C_\gamma2^{2\varepsilon n}|\xi|^{-a}.
\]
\end{lemma}

\begin{proof}
By Lemma~\ref{L:arc-oscillatory-range-arclength}, finitely many small values
of \(n\) account for all cases where \(m<1\).  These cases are absorbed into
\(C_\gamma\).  We therefore assume \(m\ge1\), so \(m_+=m\).

Recall that
\(
    S_{\xi,d}=\operatorname{spt}\chi_{\xi,d}.
\)
For \(\ell\ge0\), set
\[
    S_{\xi,d}^{(\ell)}
    =
    \left\{
        t\in[0,1]:
        \operatorname{dist}(t,S_{\xi,d})\le2^{-\ell}
    \right\}.
\]
By Proposition~\ref{P:arc-endpoint-safe-phase-bins} and
Lemma~\ref{L:arc-g-level-geometry}, the set \(S_{\xi,d}^{(\ell)}\) is covered
by at most \(C_\gamma\) intervals of total length at most
\(
    C_\gamma(d+2^{-\ell}).
\)

If \(c_J(\xi,d,\ell)\ne0\), then \(J\cap S_{\xi,d}\ne\varnothing\).  Since
\(J\subset I\) and \(|I|=2^{-\ell}\), it follows that
\(
    I\subset S_{\xi,d}^{(\ell)}.
\)
The two children of each parent contribute only an absolute factor.  Therefore,
by Lemma~\ref{L:arc-stopped-local-mass-consequences}, on
\(\{\tau_n>\ell\}\) we have
\[
    \sum_{\substack{
        I\in\mathcal D_\ell([0,1])\\
        I\subset S_{\xi,d}^{(\ell)}
    }}
    M_I
    \le
    C_\gamma2^{\varepsilon n}(d+2^{-\ell})^a.
\]

For \(0\le\ell<m\), we combine the coefficient estimate
\(
|c_J(\xi,d,\ell)|\le C_\gamma2^{\ell-m},
\)
with the stopping-time bound on \(\mathcal G_n\),
\(
T_{\ell+1,n}\le C2^{\varepsilon n}2^{-a\ell}.
\)
Thus the prefix contribution is bounded by
\[
    C_\gamma2^{2\varepsilon n}
    \sum_{\ell=0}^{m-1}
    2^{2(\ell-m)}
    2^{-a\ell}
    (d+2^{-\ell})^a.
\]
We claim that the last sum is \(O_\gamma(|\xi|^{-a})\).  Indeed, when
\(2^{-\ell}\ge d\), we have
\(
    (d+2^{-\ell})^a\le C2^{-a\ell},
\)
and hence this part is at most
\[
    C2^{-2m}
    \sum_{2^{-\ell}\ge d}2^{(2-2a)\ell}
    \le
    C2^{-2m}d^{-(2-2a)}
    \le
    C_\gamma|\xi|^{-a}.
\]
In the last inequality we used
\(
    2^{-m}\le(|\xi|d)^{-1}
\)
and the active-scale lower bound
\(
    d\ge C_\gamma^{-1}|\xi|^{-1/2}.
\)
When \(2^{-\ell}<d\), we have
\(
(d+2^{-\ell})^a\le Cd^a,
\)
and this part is at most
\[
    C2^{-2m}d^a
    \sum_{\ell<m}2^{(2-a)\ell}
    \le
    Cd^a2^{-am}
    \le
    C|\xi|^{-a}.
\]
This proves the prefix estimate.

For the post-bin estimate, use the uniform coefficient bound
\[
    |c_J(\xi,d,\ell)|\le C_\gamma
    \qquad(\ell\ge m).
\]
For such \(\ell\), the definition of \(m=m_{\xi,d}\) gives
\[
    2^{-\ell}\le2^{-m}\le(|\xi|d)^{-1}\le C_\gamma d,
\]
where the last inequality follows from the active-scale lower bound.  Hence
\(
d+2^{-\ell}\le C_\gamma d.
\)
The stopped local-mass estimate therefore gives
\[
    \sum_{\substack{
        I\in\mathcal D_\ell([0,1])\\
        I\subset S_{\xi,d}^{(\ell)}
    }}
    M_I
    \le
    C_\gamma2^{\varepsilon n}d^a.
\]
Thus the level-\(\ell\) post contribution is at most
\(
    C_\gamma2^{2\varepsilon n}2^{-a\ell}d^a.
\)
Summing over \(\ell\ge m\) gives
\[
    C_\gamma2^{2\varepsilon n}d^a2^{-am}
    \le
    C_\gamma2^{2\varepsilon n}|\xi|^{-a},
\]
again by \(2^{-m}\le(|\xi|d)^{-1}\).  This bound is uniform in the upper
cutoff \(L\), and hence proves the post-bin estimate.
\end{proof}

\begin{definition}
\label{D:arc-centered-capped-arrays}
Let \(d\in\mathfrak D_{\xi,n}^{\mathrm{osc}}\), and recall that
\(m_+=m_{\xi,d,+}\).
Define
\[
    \widehat F_{\xi,d,n}^{\mathrm{pre,cap}}
    =
    \sum_{\ell=0}^{m_+-1}
    \mathbf 1_{\{\tau_n>\ell\}}
    \sum_{I\in\mathcal D_\ell([0,1])}
    \sum_{\substack{
        J\in\mathcal D_{\ell+1}([0,1])\\
        J\subset I
    }}
    X_{J,n}^{\mathrm{cap}}(\xi,d,\ell).
\]
For every integer \(L\ge1\), define
\[
    \widehat F_{\xi,d,n,L}^{\mathrm{post,cap}}
    =
    \sum_{\ell=m_+}^{L-1}
    \mathbf 1_{\{\tau_n>\ell\}}
    \sum_{I\in\mathcal D_\ell([0,1])}
    \sum_{\substack{
        J\in\mathcal D_{\ell+1}([0,1])\\
        J\subset I
    }}
    X_{J,n}^{\mathrm{cap}}(\xi,d,\ell),
\]
with the convention that the sum is \(0\) when \(L\le m_+\).

The corresponding unstopped arrays
\(F_{\xi,d,n}^{\mathrm{pre,cap}}\) and
\(F_{\xi,d,n,L}^{\mathrm{post,cap}}\) are obtained by removing the factors
\(\mathbf 1_{\{\tau_n>\ell\}}\). 
The same empty-sum convention is used for
\(F_{\xi,d,n,L}^{\mathrm{post,cap}}\).

Whenever the limit exists, set
\[
    F_{\xi,d,n}^{\mathrm{post,cap}}
    =
    \lim_{L\to\infty}F_{\xi,d,n,L}^{\mathrm{post,cap}}.
\]
\end{definition}

We shall use the following complex-valued form of Freedman's inequality.  The
original real-valued martingale inequality is due to Freedman
\cite{Freedman1975}, and the proof of the complex-valued version used below is
recorded in \cite{2026-M}.

\begin{lemma}
\label{L:complex-freedman}
Let \((S_j,\mathcal H_j)_{j=0}^N\) be a complex-valued martingale with \(S_0=0\), 
and let \(\Delta_j=S_j-S_{j-1}\) be its differences.  
Let \(R,V\in(0,\infty)\).  
Assume that \(|\Delta_j|\le R\)
for
\(1\le j\le N\),
and
\[
    \sum_{j=1}^N
    \mathbb E[|\Delta_j|^2\mid\mathcal H_{j-1}]
    \le
    V
\]
almost surely.  Then, for every \(t>0\),
\[
    \mathbb P
    \left(
        \max_{0\le j\le N}|S_j|>t
    \right)
    \le
    4\exp
    \left(
        -\frac{t^2}{8(V+Rt)}
    \right).
\]
\end{lemma}

\begin{lemma}
\label{L:arc-jump-quadratic-variation-budgets}
There exist constants \(\eta_R>3\varepsilon\) and \(\eta_V>6\varepsilon\) such
that, for every \(n\ge1\), every \(\xi\in\mathcal A_n\), and every
\(d\in\mathfrak D_{\xi,n}^{\mathrm{osc}}\), the stopped prefix martingale and
every finite stopped post-bin martingale associated to \((\xi,d,n)\), ordered
edge by edge generation by generation, admit deterministic Freedman budgets
\(R_{\xi,d,n}\) and \(V_{\xi,d,n}\) satisfying
\[
    R_{\xi,d,n}
    \le
    C_\gamma2^{-sn/2}2^{-\eta_R n},
    \qquad
    V_{\xi,d,n}
    \le
    C_\gamma2^{-sn}2^{-\eta_V n}.
\]
Every martingale difference \(\Delta\) satisfies
\[
    |\Delta|\le R_{\xi,d,n},
\]
and the predictable quadratic variation of each such martingale is bounded by
\(V_{\xi,d,n}\).
\end{lemma}

\begin{proof}
Set
\(m=m_{\xi,d}\),
\(m_+=m_{\xi,d,+}\).
By Lemma~\ref{L:arc-centered-capped-increment-bounds},
\[
    |X_{J,n}^{\mathrm{cap}}(\xi,d,\ell)|
    \le
    C|c_J(\xi,d,\ell)|L_nT_{\ell+1,n}.
\]
The coefficient estimates give
\[
    |c_J(\xi,d,\ell)|T_{\ell+1,n}
    \le
    C_\gamma2^{\varepsilon n}2^{-am}
    \qquad(\ell\ge0).
\]
Indeed, if \(\ell<m\), then
\[
    |c_J|T_{\ell+1,n}
    \le
    C_\gamma2^{\ell-m}2^{\varepsilon n}2^{-a\ell}
    =
    C_\gamma2^{\varepsilon n}2^{-m}2^{(1-a)\ell}
    \le
    C_\gamma2^{\varepsilon n}2^{-am},
\]
while if \(\ell\ge m\), then
\[
    |c_J|T_{\ell+1,n}
    \le
    C_\gamma2^{\varepsilon n}2^{-a\ell}
    \le
    C_\gamma2^{\varepsilon n}2^{-am}.
\]
Hence every stopped prefix or finite stopped post-bin increment is bounded by
\(
    C_\gamma L_n2^{\varepsilon n}2^{-am}.
\)
Using \(L_n=2^{\kappa n}\), \(2^{-m}\le(|\xi|d)^{-1}\), the oscillatory
condition
\(
    d^{-a}<2^{sn/2}2^{8\varepsilon n},
\)
and \(|\xi|\ge2^n\), we obtain
\[
    R_{\xi,d,n}
    \le
    C_\gamma2^{(\kappa+\varepsilon)n}|\xi|^{-a}d^{-a}
    \le
    C_\gamma2^{(\kappa+\varepsilon)n}2^{-an}
    2^{sn/2}2^{8\varepsilon n}
    =
    C_\gamma2^{-sn/2}
    2^{-(\delta_1-\kappa-9\varepsilon)n}.
\]
The parameter choice gives 
\(
    \delta_1-\kappa-9\varepsilon>3\varepsilon,
\)
so we can choose
\(
    3\varepsilon<\eta_R<\delta_1-\kappa-9\varepsilon
\)
and obtain
\[
    R_{\xi,d,n}
    \le
    C_\gamma2^{-sn/2}2^{-\eta_Rn}.
\]

It remains to justify the predictable quadratic-variation bound in the
edge-revealing filtration.  Fix either the prefix range
\(
    0\le\ell<m_+
\)
or a finite post-bin range
\(
    m_+\le\ell<L.
\)
List all child edges in the chosen range as
\[
    e_j=(I_j,J_j),
    \qquad
    I_j\in\mathcal D_{\ell_j}([0,1]),
    \quad
    J_j\subset I_j,
\]
first by parent generation \(\ell_j\), and then in a fixed deterministic order.
Let \(\ell_-\) denote the initial generation of the chosen range, and set
\[
    \mathcal H_0=\mathcal F_{\ell_-},
    \qquad
    \mathcal H_j
    =
    \mathcal F_{\ell_-}
    \vee
    \sigma(U_{J_1},\ldots,U_{J_j}).
\]
Because the ordering is generation by generation, all weights in generations
strictly below \(\ell_j+1\) have already been revealed before \(U_{J_j}\) is
revealed.  Hence \(\mathcal H_{j-1}\) contains \(\mathcal F_{\ell_j}\), the
mass \(M_{I_j}\), the cap \(C_{I_j,n}\), and the indicator
\(\mathbf 1_{\{\tau_n>\ell_j\}}\), but it does not contain the fresh child
weight \(U_{J_j}\).  The latter is independent of \(\mathcal H_{j-1}\).  Therefore, for
\[
    \Delta_j
    =
    \mathbf 1_{\{\tau_n>\ell_j\}}
    X_{J_j,n}^{\mathrm{cap}}(\xi,d,\ell_j),
\]
we have
\(
    \mathbb E[\Delta_j\mid\mathcal H_{j-1}]=0.
\)
Moreover, the proof of
Lemma~\ref{L:arc-centered-capped-increment-bounds} applies with
\(\mathcal H_{j-1}\) in place of \(\mathcal F_{\ell_j}\), because the only
fresh random variable is \(U_{J_j}\).  Hence
\[
    \mathbb E
    \left[
        |\Delta_j|^2
        \mid
        \mathcal H_{j-1}
    \right]
    \le
    CL_n\mathbf 1_{\{\tau_n>\ell_j\}}
    |c_{J_j}(\xi,d,\ell_j)|^2M_{I_j}T_{\ell_j+1,n}.
\]

Summing this estimate and applying
Lemma~\ref{L:arc-stopped-square-sum-budgets}, first in the prefix range and
then in any finite post-bin range, gives the common deterministic bound
\[
    \sum_j
    \mathbb E
    \left[
        |\Delta_j|^2
        \mid
        \mathcal H_{j-1}
    \right]
    \le
    C_\gamma L_n2^{2\varepsilon n}|\xi|^{-a}.
\]
Thus we may take
\[
    V_{\xi,d,n}
    =
    C_\gamma L_n2^{2\varepsilon n}|\xi|^{-a}.
\]
Using \(L_n=2^{\kappa n}\), \(|\xi|\ge2^n\), and \(a=s+\delta_1\), we get
\[
    V_{\xi,d,n}
    \le
    C_\gamma2^{(\kappa+2\varepsilon)n}2^{-an}
    =
    C_\gamma2^{-sn}
    2^{-(\delta_1-\kappa-2\varepsilon)n}.
\]
The parameter choice gives 
\(
    \delta_1-\kappa-2\varepsilon>6\varepsilon,
\)
so we can choose
\(
    6\varepsilon<\eta_V<\delta_1-\kappa-2\varepsilon
\)
and obtain
\[
    V_{\xi,d,n}
    \le
    C_\gamma2^{-sn}2^{-\eta_Vn}.
\]
The edge-ordered partial sums are therefore martingales with jump bound
\(R_{\xi,d,n}\) and predictable quadratic variation bounded by
\(V_{\xi,d,n}\), as claimed.
\end{proof}

\begin{proposition}[Capped Freedman estimate]
\label{P:arc-capped-freedman}
There exist constants \(C_\gamma<\infty\), \(c_\gamma>0\), and \(\eta>0\) such
that, for every \(n\ge1\), every \(\xi\in\mathcal A_n\), and every
\(d\in\mathfrak D_{\xi,n}^{\mathrm{osc}}\),
\[
    \mathbb P
    \left(
        \left|
            \widehat F_{\xi,d,n}^{\mathrm{pre,cap}}
        \right|
        >
        2^{-sn/2}2^{-3\varepsilon n}
    \right)
    \le
    C_\gamma\exp(-c_\gamma2^{\eta n}),
\]
and
\[
    \mathbb P
    \left(
        \sup_{L>m_{\xi,d,+}}
        \left|
            \widehat F_{\xi,d,n,L}^{\mathrm{post,cap}}
        \right|
        >
        2^{-sn/2}2^{-3\varepsilon n}
    \right)
    \le
    C_\gamma\exp(-c_\gamma2^{\eta n}).
\]
Consequently,
\[
    \mathbb P
    \left(
        \mathcal G_n^{\mathrm{pre}}
        \cap
        \left\{
            |F_{\xi,d,n}^{\mathrm{pre,cap}}|
            +
            \sup_{L>m_{\xi,d,+}}
            |F_{\xi,d,n,L}^{\mathrm{post,cap}}|
            >
            2^{-sn/2}2^{-2\varepsilon n}
        \right\}
    \right)
    \le
    C_\gamma\exp(-c_\gamma2^{\eta n}).
\]
\end{proposition}

\begin{proof}
Set
\(
    t_n=2^{-sn/2}2^{-3\varepsilon n}.
\)
By Lemma~\ref{L:arc-jump-quadratic-variation-budgets}, the stopped prefix
martingale has jump and predictable quadratic-variation bounds
\[
    R_{\xi,d,n}
    \le
    C_\gamma2^{-sn/2}2^{-\eta_R n},
    \qquad
    V_{\xi,d,n}
    \le
    C_\gamma2^{-sn}2^{-\eta_V n}.
\]
Since
\(
    t_n^2=2^{-sn}2^{-6\varepsilon n},
\)
and since \(\eta_R>3\varepsilon\) and \(\eta_V>6\varepsilon\), there exists
\(\eta>0\) such that
\[
    V_{\xi,d,n}+R_{\xi,d,n}t_n
    \le
    C_\gamma t_n^2 2^{-\eta n}.
\]
Applying Lemma~\ref{L:complex-freedman} gives
\[
    \mathbb P
    \left(
        \left|
            \widehat F_{\xi,d,n}^{\mathrm{pre,cap}}
        \right|
        >t_n
    \right)
    \le
    C_\gamma\exp(-c_\gamma2^{\eta n}).
\]

The same argument applies to the finite stopped post-bin martingales uniformly in the terminal generation.  
Taking the supremum over \(L\) follows by monotone convergence of the events.  
Hence
\[
    \mathbb P
    \left(
        \sup_{L>m_{\xi,d,+}}
        \left|
            \widehat F_{\xi,d,n,L}^{\mathrm{post,cap}}
        \right|
        >t_n
    \right)
    \le
    C_\gamma\exp(-c_\gamma2^{\eta n}).
\]

On \(\mathcal G_n^{\mathrm{pre}}\), we have \(\tau_n=+\infty\).  Hence the
stopped and unstopped centered capped arrays agree on
\(\mathcal G_n^{\mathrm{pre}}\).  Moreover, for all sufficiently large \(n\),
\(
    2t_n\le2^{-sn/2}2^{-2\varepsilon n}.
\)
Therefore,
\[
\begin{aligned}
    &
    \mathcal G_n^{\mathrm{pre}}
    \cap
    \left\{
        |F_{\xi,d,n}^{\mathrm{pre,cap}}|
        +
        \sup_{L>m_{\xi,d,+}}
        |F_{\xi,d,n,L}^{\mathrm{post,cap}}|
        >
        2^{-sn/2}2^{-2\varepsilon n}
    \right\}                                      \\
    &\qquad\subset
    \left\{
        |\widehat F_{\xi,d,n}^{\mathrm{pre,cap}}|>t_n
    \right\}
    \cup
    \left\{
        \sup_{L>m_{\xi,d,+}}
        |\widehat F_{\xi,d,n,L}^{\mathrm{post,cap}}|>t_n
    \right\}.
\end{aligned}
\]
The desired probability estimate follows from the two preceding Freedman
bounds and the union bound.  The finitely many small values of \(n\) are
absorbed by increasing \(C_\gamma\).
\end{proof}

\subsection{The \texorpdfstring{\(r\)}{r}-tail compensator}
\label{SS:r-tail-compensator}

We estimate the predictable drift created by capping.  Recall from
Lemma~\ref{L:arc-capping-identity} that
\[
    D_{J,n}(\xi,d,\ell)
    =
    c_J(\xi,d,\ell)M_I(\overline U_{I,n}-1),
\]
and equivalently
\[
    D_{J,n}(\xi,d,\ell)
    =
    -
    c_J(\xi,d,\ell)M_I
    \mathbb E\left[
        (U_J-C_{I,n})_+
        \mid
        \mathcal F_\ell
    \right].
\]

\begin{lemma}
\label{L:arc-one-step-r-tail-bound}
There exists \(C<\infty\) such that, for every \(n\ge1\), every
\(\ell\ge0\), every \(I\in\mathcal D_\ell([0,1])\), and every child
\(J\subset I\),
\[
    |D_{J,n}(\xi,d,\ell)|
    \le
    C|c_J(\xi,d,\ell)|
    M_I^r
    L_n^{1-r}
    T_{\ell+1,n}^{1-r}.
\]
\end{lemma}

\begin{proof}
If \(M_I=0\), then the estimate is immediate.  Assume \(M_I>0\).  By the
definition of the cap,
\[
    C_{I,n}
    =
    L_n\frac{2T_{\ell+1,n}}{M_I}.
\]
For \(K>0\),
\[
    \mathbb E[(U_J-K)_+]
    \le
    \mathbb E[U_J\mathbf 1_{\{U_J>K\}}]
    \le
    K^{1-r}\mathbb E[U_J^r].
\]
Applying this with \(K=C_{I,n}\), and using \(\mathbb E[U^r]<\infty\), gives
\[
\begin{aligned}
    |D_{J,n}(\xi,d,\ell)|
    &\le
    |c_J(\xi,d,\ell)|M_I
    \mathbb E[(U_J-C_{I,n})_+\mid\mathcal F_\ell]      \\
    &\le
    C|c_J(\xi,d,\ell)|M_I C_{I,n}^{1-r}               \\
    &\le
    C|c_J(\xi,d,\ell)|
    M_I^r
    L_n^{1-r}
    T_{\ell+1,n}^{1-r}.
\end{aligned}
\]
This proves the claim.
\end{proof}

\begin{definition}
\label{D:arc-level-compensator-envelope}
For \(n\ge1\) and \(\ell\ge0\), define
\[
    R_{\ell,n}
    =
    2^{-(\varepsilon+\kappa)(r-1)n}
    2^{-\beta_{\mathrm{comp}}\ell},
\]
where \(\beta_{\mathrm{comp}}\) is the exponent fixed in
Subsection~\ref{SS:parameters-good-events}.
\end{definition}

\begin{lemma}[Expected level envelope]
\label{L:arc-level-compensator-envelope}
There exists \(C<\infty\) such that, for every \(n\ge1\) and every
\(\ell\ge0\),
\[
    \mathbb E\left[
        \sum_{I\in\mathcal D_\ell([0,1])}
        M_I^r
        L_n^{1-r}
        T_{\ell+1,n}^{1-r}
    \right]
    \le
    CR_{\ell,n}.
\]
\end{lemma}

\begin{proof}
By the definition of \(L_n\) and \(T_{\ell+1,n}\),
\[
    L_n^{1-r}
    =
    2^{-\kappa(r-1)n},
\qquad
\text{and}
\qquad
    T_{\ell+1,n}^{1-r}
    =
    2^{-\varepsilon(r-1)n}
    2^{a(r-1)(\ell+1)}.
\]
Therefore,
\[
    L_n^{1-r}T_{\ell+1,n}^{1-r}
    \le
    C
    2^{-(\varepsilon+\kappa)(r-1)n}
    2^{a(r-1)\ell}.
\]
Using the dyadic \(r\)-mass budget,
Lemma~\ref{L:arc-r-mass-budgets}, we get
\[
\begin{aligned}
    \mathbb E\left[
        \sum_{I\in\mathcal D_\ell([0,1])}
        M_I^r
        L_n^{1-r}
        T_{\ell+1,n}^{1-r}
    \right]
    &\le
    C
    2^{-(\varepsilon+\kappa)(r-1)n}
    2^{a(r-1)\ell}
    2^{-r(s+\delta)\ell}        \\
    &=
    C
    2^{-(\varepsilon+\kappa)(r-1)n}
    2^{-\beta_{\mathrm{comp}}\ell}       \\
    &=
    CR_{\ell,n}.
\end{aligned}
\]
\end{proof}

Choose \(C_{\gamma,0}<\infty\) large enough so that
Lemma~\ref{L:arc-oscillatory-range-arclength} implies
\[
    \vartheta n-C_{\gamma,0}
    \le
    m_{\xi,d}
    \le
    n+C_{\gamma,0}
\]
for every \(n\ge1\), every \(\xi\in\mathcal A_n\), and every
\(d\in\mathfrak D_{\xi,n}^{\mathrm{osc}}\), after enlarging
\(C_{\gamma,0}\) to absorb finitely many small values of \(n\).  Set
\[
    m_n^-
    =
    \max\{0,\lfloor \vartheta n-C_{\gamma,0}\rfloor\},
    \qquad
    m_n^+
    =
    \lceil n+C_{\gamma,0}\rceil.
\]
Since the prefix--post splitting is made at the truncated generation
\(
    m_{\xi,d,+}=\max\{m_{\xi,d},0\},
\)
we shall use the corresponding deterministic window
\[
    m_n^-\le m_{\xi,d,+}\le m_n^+.
\]

\begin{definition}
\label{D:arc-absolute-compensator-arrays}
For \(n\ge1\), \(\xi\in\mathcal A_n\), and
\(d\in\mathfrak D_{\xi,n}^{\mathrm{osc}}\), recall that
\(
    m_+=m_{\xi,d,+}.
\)
Define the prefix absolute compensator by
\[
    \mathfrak C_{\xi,d,n}^{\mathrm{pre}}
    =
    \sum_{\ell=0}^{m_+-1}
    \sum_{I\in\mathcal D_\ell([0,1])}
    \sum_{\substack{
        J\in\mathcal D_{\ell+1}([0,1])\\
        J\subset I
    }}
    |D_{J,n}(\xi,d,\ell)|.
\]
For every integer \(L\ge1\), define the finite post-bin absolute compensator by
\[
    \mathfrak C_{\xi,d,n,L}^{\mathrm{post}}
    =
    \sum_{\ell=m_+}^{L-1}
    \sum_{I\in\mathcal D_\ell([0,1])}
    \sum_{\substack{
        J\in\mathcal D_{\ell+1}([0,1])\\
        J\subset I
    }}
    |D_{J,n}(\xi,d,\ell)|,
\]
with the convention that the sum is \(0\) when \(L\le m_+\).
\end{definition}

\begin{lemma}
\label{L:arc-bounded-band-count-phase-bin}
There exists \(C_\gamma<\infty\) such that, for every \(\xi\ne0\) and every
integer \(m\),
\[
    \#\{d\in\mathfrak D_\xi:m_{\xi,d}=m\}
    \le
    C_\gamma.
\]
\end{lemma}

\begin{proof}
If \(m_{\xi,d}=m\), then
\(
    2^{m-1}|\xi|^{-1}<d\le2^m|\xi|^{-1}.
\)
Since \(d\) ranges over dyadic values, this interval contains only \(O(1)\)
possible values of \(d\).  This proves the claim.
\end{proof}

We shall also use the following immediate consequence.  With
\(m_{\xi,d,+}:=\max\{m_{\xi,d},0\}\), after increasing \(C_\gamma\) if necessary,
one has
\begin{equation}
\label{E:truncated-band-count}
    \#\{d\in\mathfrak D_{\xi,n}^{\mathrm{osc}}:m_{\xi,d,+}=m\}
    \le
    C_\gamma
    \qquad
    (n\ge1,\ \xi\in\mathcal A_n,\ m\ge0).
\end{equation}
Indeed, for \(m\ge1\) this follows immediately from
Lemma~\ref{L:arc-bounded-band-count-phase-bin}.  For \(m=0\), the deterministic
range for \(m_{\xi,d}\) gives
\[
    m_{\xi,d}\ge \vartheta n-C_{\gamma,0}\ge -C_{\gamma,0},
\]
so the level \(m_{\xi,d,+}=0\) is contained in the union of \(O_\gamma(1)\)
untruncated levels \(m_{\xi,d}=q\le0\).  Applying
Lemma~\ref{L:arc-bounded-band-count-phase-bin} to these levels gives
\eqref{E:truncated-band-count}.

\begin{lemma}
\label{L:arc-prefix-compensator-envelope}
There exists a nonnegative random variable \(E_n^{\mathrm{pre}}\) such that
\[
    \sup_{\xi\in\mathcal A_n}
    \sum_{d\in\mathfrak D_{\xi,n}^{\mathrm{osc}}}
    \mathfrak C_{\xi,d,n}^{\mathrm{pre}}
    \le
    E_n^{\mathrm{pre}},
\qquad
\text{and}
\qquad
    \mathbb E[E_n^{\mathrm{pre}}]
    \le
    C_\gamma n^2
    2^{-\theta_{\mathrm{comp}}\vartheta n}.
\]
\end{lemma}

\begin{proof}
For \(\ell\ge0\), put
\[
    \mathcal R_{\ell,n}
    =
    \sum_{I\in\mathcal D_\ell([0,1])}
    M_I^rL_n^{1-r}T_{\ell+1,n}^{1-r}.
\]
By Lemma~\ref{L:arc-level-compensator-envelope},
\(
    \mathbb E[\mathcal R_{\ell,n}]
    \le
    CR_{\ell,n}.
\)

Fix \(\xi\in\mathcal A_n\) and
\(d\in\mathfrak D_{\xi,n}^{\mathrm{osc}}\), and write
\(
    m_+=m_{\xi,d,+}.
\)
If \(m_{\xi,d}\le0\), then \(m_+=0\), and the prefix compensator is empty.  If
\(m_{\xi,d}>0\), then \(m_+=m_{\xi,d}\), and Lemma~\ref{L:arc-one-step-r-tail-bound},
together with the prefix coefficient bound
\[
    |c_J(\xi,d,\ell)|
    \le
    C_\gamma2^{\ell-m_{\xi,d}}
    \qquad (0\le\ell<m_{\xi,d}),
\]
gives
\[
    \mathfrak C_{\xi,d,n}^{\mathrm{pre}}
    \le
    C_\gamma
    \sum_{\ell=0}^{m_+-1}
    2^{\ell-m_+}\mathcal R_{\ell,n}.
\]
Thus the same bound holds for every \(d\), with the convention that the sum is
empty when \(m_+=0\).

Now sum over \(d\in\mathfrak D_{\xi,n}^{\mathrm{osc}}\).  By
\eqref{E:truncated-band-count} and the deterministic window
\(m_n^-\le m_{\xi,d,+}\le m_n^+\), the right-hand side is bounded by
\[
    C_\gamma
    \sum_{m=m_n^-}^{m_n^+}
    \sum_{\ell=0}^{m-1}
    2^{\ell-m}\mathcal R_{\ell,n}.
\]
This bound no longer depends on \(\xi\).  We therefore define
\[
    E_n^{\mathrm{pre}}
    =
    C_\gamma
    \sum_{m=m_n^-}^{m_n^+}
    \sum_{\ell=0}^{m-1}
    2^{\ell-m}\mathcal R_{\ell,n}.
\]
Then
\[
    \sup_{\xi\in\mathcal A_n}
    \sum_{d\in\mathfrak D_{\xi,n}^{\mathrm{osc}}}
    \mathfrak C_{\xi,d,n}^{\mathrm{pre}}
    \le
    E_n^{\mathrm{pre}}.
\]

It remains to estimate the expectation.  Since
\[
    R_{\ell,n}
    =
    2^{-(\varepsilon+\kappa)(r-1)n}
    2^{-\beta_{\mathrm{comp}}\ell}
    \le
    2^{-\beta_{\mathrm{comp}}\ell},
\]
we get
\[
    \mathbb E[E_n^{\mathrm{pre}}]
    \le
    C_\gamma
    \sum_{m=m_n^-}^{m_n^+}
    \sum_{\ell=0}^{m-1}
    2^{\ell-m}R_{\ell,n}
    \le
    C_\gamma
    \sum_{m=m_n^-}^{m_n^+}
    \sum_{\ell=0}^{m-1}
    2^{\ell-m}2^{-\beta_{\mathrm{comp}}\ell}.
\]
Recall that
\(
    \theta_{\mathrm{comp}}
    =
    \min\{\beta_{\mathrm{comp}},1\}.
\)
For each \(m\ge1\),
\[
    \sum_{\ell=0}^{m-1}
    2^{\ell-m}2^{-\beta_{\mathrm{comp}}\ell}
    \le
    Cm2^{-\theta_{\mathrm{comp}}m}.
\]
Therefore,
\[
    \mathbb E[E_n^{\mathrm{pre}}]
    \le
    C_\gamma n^2
    2^{-\theta_{\mathrm{comp}}m_n^-}
    \le
    C_\gamma n^2
    2^{-\theta_{\mathrm{comp}}\vartheta n},
\]
after increasing \(C_\gamma\).
\end{proof}

\begin{lemma}
\label{L:arc-post-compensator-envelope}
There exists a nonnegative random variable \(E_n^{\mathrm{post}}\) such that
\[
    \sup_{\xi\in\mathcal A_n}
    \sup_{L\ge1}
    \sum_{d\in\mathfrak D_{\xi,n}^{\mathrm{osc}}}
    \mathfrak C_{\xi,d,n,L}^{\mathrm{post}}
    \le
    E_n^{\mathrm{post}},
\qquad
\text{and}
\qquad
    \mathbb E[E_n^{\mathrm{post}}]
    \le
    C_\gamma n
    2^{-\beta_{\mathrm{comp}}\vartheta n}.
\]
\end{lemma}

\begin{proof}
With \(\mathcal R_{\ell,n}\) as in the proof of
Lemma~\ref{L:arc-prefix-compensator-envelope}, Lemma~\ref{L:arc-one-step-r-tail-bound}
and the post-bin coefficient bound
\[
    |c_J(\xi,d,\ell)|\le C_\gamma
    \qquad(\ell\ge m_{\xi,d})
\]
give
\[
    \mathfrak C_{\xi,d,n,L}^{\mathrm{post}}
    \le
    C_\gamma
    \sum_{\ell=m_{\xi,d,+}}^{L-1}
    \mathcal R_{\ell,n}.
\]
Summing over \(d\in\mathfrak D_{\xi,n}^{\mathrm{osc}}\), and using
\eqref{E:truncated-band-count} together with the deterministic window
\(m_n^-\le m_{\xi,d,+}\le m_n^+\), gives
\[
    \sum_{d\in\mathfrak D_{\xi,n}^{\mathrm{osc}}}
    \mathfrak C_{\xi,d,n,L}^{\mathrm{post}}
    \le
    C_\gamma
    \sum_{m=m_n^-}^{m_n^+}
    \sum_{\ell=m}^{\infty}
    \mathcal R_{\ell,n}.
\]
The right-hand side is independent of \(\xi\) and \(L\).  Define
\[
    E_n^{\mathrm{post}}
    =
    C_\gamma
    \sum_{m=m_n^-}^{m_n^+}
    \sum_{\ell=m}^{\infty}
    \mathcal R_{\ell,n}.
\]
Then
\[
    \sup_{\xi\in\mathcal A_n}
    \sup_{L\ge1}
    \sum_{d\in\mathfrak D_{\xi,n}^{\mathrm{osc}}}
    \mathfrak C_{\xi,d,n,L}^{\mathrm{post}}
    \le
    E_n^{\mathrm{post}}.
\]

Taking expectations and using
\(
    \mathbb E[\mathcal R_{\ell,n}]
    \le
    CR_{\ell,n}
    \le
    C2^{-\beta_{\mathrm{comp}}\ell}
\)
gives
\[
\begin{aligned}
    \mathbb E[E_n^{\mathrm{post}}]
    \le
    C_\gamma
    \sum_{m=m_n^-}^{m_n^+}
    \sum_{\ell=m}^{\infty}
    2^{-\beta_{\mathrm{comp}}\ell}\le
    C_\gamma
    \sum_{m=m_n^-}^{m_n^+}
    2^{-\beta_{\mathrm{comp}}m} \le
    C_\gamma n
    2^{-\beta_{\mathrm{comp}}\vartheta n},
\end{aligned}
\]
after increasing \(C_\gamma\).
\end{proof}

\begin{proposition}[\(r\)-tail compensator estimate]
\label{P:arc-r-tail-compensator}
There exist constants \(C_\gamma<\infty\) and \(c_\gamma>0\) such that, for
every \(n\ge1\),
\[
    \mathbb P\left(
        \sup_{\xi\in\mathcal A_n}
        \left(
            \sum_{d\in\mathfrak D_{\xi,n}^{\mathrm{osc}}}
            \mathfrak C_{\xi,d,n}^{\mathrm{pre}}
            +
            \sup_{L\ge1}
            \sum_{d\in\mathfrak D_{\xi,n}^{\mathrm{osc}}}
            \mathfrak C_{\xi,d,n,L}^{\mathrm{post}}
        \right)
        >
        2n^{-2}2^{-sn/2}
    \right)
    \le
    C_\gamma2^{-c_\gamma n}.
\]
Consequently, these probabilities are summable in \(n\).
\end{proposition}

\begin{proof}
By Lemma~\ref{L:arc-prefix-compensator-envelope}, there is a nonnegative random
variable \(E_n^{\mathrm{pre}}\) such that
\[
    \sup_{\xi\in\mathcal A_n}
    \sum_{d\in\mathfrak D_{\xi,n}^{\mathrm{osc}}}
    \mathfrak C_{\xi,d,n}^{\mathrm{pre}}
    \le
    E_n^{\mathrm{pre}},
\qquad
\text{and}
\qquad
    \mathbb E[E_n^{\mathrm{pre}}]
    \le
    C_\gamma n^2
    2^{-\theta_{\mathrm{comp}}\vartheta n}.
\]
By Lemma~\ref{L:arc-post-compensator-envelope}, there is a nonnegative random
variable \(E_n^{\mathrm{post}}\) such that
\[
    \sup_{\xi\in\mathcal A_n}
    \sup_{L\ge1}
    \sum_{d\in\mathfrak D_{\xi,n}^{\mathrm{osc}}}
    \mathfrak C_{\xi,d,n,L}^{\mathrm{post}}
    \le
    E_n^{\mathrm{post}},
\qquad
\text{and}
\qquad
    \mathbb E[E_n^{\mathrm{post}}]
    \le
    C_\gamma n
    2^{-\beta_{\mathrm{comp}}\vartheta n}.
\]

By Lemma~\ref{lemma:annular-parameter-choice},
\[
    \theta_{\mathrm{comp}}\vartheta>\frac{s}{2},
    \qquad
    \beta_{\mathrm{comp}}\vartheta>\frac{s}{2}.
\]
Hence, after decreasing \(c_\gamma>0\) and increasing \(C_\gamma\), the
polynomial factors \(n^2\) and \(n\) can be absorbed into the exponential
margins, giving
\[
    \mathbb E[E_n^{\mathrm{pre}}+E_n^{\mathrm{post}}]
    \le
    C_\gamma2^{-sn/2}2^{-c_\gamma n}.
\]

Therefore,
\[
\begin{aligned}
    &\mathbb P\left(
        \sup_{\xi\in\mathcal A_n}
        \left(
            \sum_{d\in\mathfrak D_{\xi,n}^{\mathrm{osc}}}
            \mathfrak C_{\xi,d,n}^{\mathrm{pre}}
            +
            \sup_{L\ge1}
            \sum_{d\in\mathfrak D_{\xi,n}^{\mathrm{osc}}}
            \mathfrak C_{\xi,d,n,L}^{\mathrm{post}}
        \right)
        >
        2n^{-2}2^{-sn/2}
    \right)                                      \\
    &\qquad\le
    \mathbb P\left(
        E_n^{\mathrm{pre}}+E_n^{\mathrm{post}}
        >
        2n^{-2}2^{-sn/2}
    \right)                                      \\
    &\qquad\le
    C_\gamma n^2 2^{-c_\gamma n}
    \le
    C_\gamma2^{-c_\gamma n},
\end{aligned}
\]
after decreasing \(c_\gamma>0\) and increasing \(C_\gamma\), if necessary.
The final bound is summable in \(n\), and the proof is complete.
\end{proof}

\begin{definition}
\label{D:arc-signed-compensator-sums}
For \(n\ge1\), \(\xi\in\mathcal A_n\), and
\(d\in\mathfrak D_{\xi,n}^{\mathrm{osc}}\), recall that
\(
    m_+=m_{\xi,d,+}.
\)
Define
\[
    D_{\xi,d,n}^{\mathrm{pre}}
    =
    \sum_{\ell=0}^{m_+-1}
    \sum_{I\in\mathcal D_\ell([0,1])}
    \sum_{\substack{
        J\in\mathcal D_{\ell+1}([0,1])\\
        J\subset I
    }}
    D_{J,n}(\xi,d,\ell).
\]
For every integer \(L\ge1\), define
\[
    D_{\xi,d,n,L}^{\mathrm{post}}
    =
    \sum_{\ell=m_+}^{L-1}
    \sum_{I\in\mathcal D_\ell([0,1])}
    \sum_{\substack{
        J\in\mathcal D_{\ell+1}([0,1])\\
        J\subset I
    }}
    D_{J,n}(\xi,d,\ell),
\]
with the convention that the sum is \(0\) when \(L\le m_+\).

Whenever the limit exists, set
\[
    D_{\xi,d,n}^{\mathrm{post}}
    =
    \lim_{L\to\infty}
    D_{\xi,d,n,L}^{\mathrm{post}}.
\]
\end{definition}

\begin{lemma}
\label{L:arc-compensator-array-identities}
Assume \(\mathcal G_n^{\mathrm{pre}}\) holds.  Then, for every
\(\xi\in\mathcal A_n\) and every
\(d\in\mathfrak D_{\xi,n}^{\mathrm{osc}}\),
\[
    F_{\xi,d}^{\mathrm{pre}}
    =
    F_{\xi,d,n}^{\mathrm{pre,cap}}
    +
    D_{\xi,d,n}^{\mathrm{pre}}.
\]
Moreover, for every integer \(L\ge1\),
\[
    F_{\xi,d,L}^{\mathrm{post}}
    =
    F_{\xi,d,n,L}^{\mathrm{post,cap}}
    +
    D_{\xi,d,n,L}^{\mathrm{post}}.
\]
If, in addition,
\begin{equation}\label{eq:posterior_uniform_bound}
    \sup_{L\geq 1}
    \mathfrak C_{\xi,d,n,L}^{\mathrm{post}}<\infty,
\end{equation}
then \(D_{\xi,d,n}^{\mathrm{post}}\) exists.  In this case, whenever
one of the two limits
\[
    F_{\xi,d}^{\mathrm{post}}
    =
    \lim_{L\to\infty}F_{\xi,d,L}^{\mathrm{post}},
    \qquad
    F_{\xi,d,n}^{\mathrm{post,cap}}
    =
    \lim_{L\to\infty}F_{\xi,d,n,L}^{\mathrm{post,cap}}
\]
exists, the other exists as well, and
\[
    F_{\xi,d}^{\mathrm{post}}
    =
    F_{\xi,d,n}^{\mathrm{post,cap}}
    +
    D_{\xi,d,n}^{\mathrm{post}}.
\]
\end{lemma}

\begin{proof}
On \(\mathcal G_n^{\mathrm{pre}}\), Lemma~\ref{L:arc-capping-identity} gives,
for every child \(J\subset I\),
\[
    c_J(\xi,d,\ell)M_I(U_J-1)
    =
    X_{J,n}^{\mathrm{cap}}(\xi,d,\ell)
    +
    D_{J,n}(\xi,d,\ell).
\]
Summing this identity over the prefix range
\(
    0\le\ell<m_{\xi,d,+}
\)
gives
\[
    F_{\xi,d}^{\mathrm{pre}}
    =
    F_{\xi,d,n}^{\mathrm{pre,cap}}
    +
    D_{\xi,d,n}^{\mathrm{pre}}.
\]
Summing over the finite post-bin range
\(
    m_{\xi,d,+}\le\ell<L
\)
gives
\[
    F_{\xi,d,L}^{\mathrm{post}}
    =
    F_{\xi,d,n,L}^{\mathrm{post,cap}}
    +
    D_{\xi,d,n,L}^{\mathrm{post}}.
\]

If \eqref{eq:posterior_uniform_bound} holds,
then the signed post-bin compensator sums are absolutely Cauchy.  Hence
\[
    D_{\xi,d,n}^{\mathrm{post}}
    =
    \lim_{L\to\infty}
    D_{\xi,d,n,L}^{\mathrm{post}}
\]
exists.  Passing to the limit in the finite post-bin identity proves the final
assertion.
\end{proof}

\subsection{Annular grid assembly}
\label{SS:annular-grid-assembly}

We now complete the proof of the finite-\(r\) annular theorem.  Throughout the
deterministic arguments in this subsection, we work on the probability-one
weak-convergence event
\(
    \widetilde\nu_L\to\widetilde\nu .
\)
Set
\[
    R_\gamma
    =
    \max\left\{
        1,\sup_{0\le t\le1}|\gamma(t)|
    \right\}.
\]

\begin{lemma}
\label{L:arc-annular-net-grid-passage}
For every integer \(n\ge1\), set
\[
    \rho_n
    =
    (8\pi R_\gamma)^{-1}2^{-(s/2+\varepsilon)n}.
\]
There exists a finite set \(\mathcal N_n\subset\mathcal A_n\) such that
\(\mathcal A_n\) is covered by the balls
\(
    B(\eta,\rho_n)
\)
for
\(
\eta\in\mathcal N_n,
\)
and
\(
    \#\mathcal N_n
    \le
    C_\gamma2^{(2+s+2\varepsilon)n}.
\)
If \(\mathcal G_n^{\mathrm{lim}}\) holds and, for some \(A\ge0\),
\[
    \max_{\eta\in\mathcal N_n}
    |\widehat{\nu_\gamma}(\eta)|
    \le
    A2^{-sn/2},
\]
then
\[
    \sup_{\xi\in\mathcal A_n}
    |\widehat{\nu_\gamma}(\xi)|
    \le
    (A+1)2^{-sn/2}.
\]
\end{lemma}

\begin{proof}
Let \(\mathcal N_n\) be a maximal \(\rho_n\)-separated subset of
\(\mathcal A_n\).  Since \(\mathcal A_n\) is compact, such a finite maximal
subset exists.  By maximality, \(\mathcal N_n\) is a \(\rho_n\)-net for
\(\mathcal A_n\).

The balls
\(
    B(\eta,\rho_n/2)
\)
for
\(
\eta\in\mathcal N_n,
\)
are pairwise disjoint and are contained in
\(
    B(0,2^{n+1}+\rho_n/2).
\)
Hence a standard area comparison gives
\[
    \#\mathcal N_n
    \le
    C\left(1+\frac{2^n}{\rho_n}\right)^2
    \le
    C_\gamma2^{(2+s+2\varepsilon)n}.
\]

On \(\mathcal G_n^{\mathrm{lim}}\), 
we have
\(
    \widetilde\nu([0,1])\le T_{0,n}=2^{\varepsilon n}.
\)
Since \(\nu_\gamma=\gamma_\#\widetilde\nu\), it follows that, for all
\(\xi,\eta\in\mathbb R^2\),
\[
\begin{aligned}
    |\widehat{\nu_\gamma}(\xi)-\widehat{\nu_\gamma}(\eta)|
    \le
    \int_0^1
    \left|
        e^{-2\pi i\xi\cdot\gamma(t)}
        -
        e^{-2\pi i\eta\cdot\gamma(t)}
    \right|\,d\widetilde\nu(t)             
    \le
    2\pi R_\gamma2^{\varepsilon n}|\xi-\eta|.
\end{aligned}
\]
Now fix \(\xi\in\mathcal A_n\), and choose \(\eta\in\mathcal N_n\) with
\(
    |\xi-\eta|\le\rho_n.
\)
Then
\[
    |\widehat{\nu_\gamma}(\xi)-\widehat{\nu_\gamma}(\eta)|
    \le
    2\pi R_\gamma2^{\varepsilon n}\rho_n
    =
    \frac14 2^{-sn/2}.
\]
Using the assumed estimate on \(\mathcal N_n\), we get
\[
    |\widehat{\nu_\gamma}(\xi)|
    \le
    |\widehat{\nu_\gamma}(\eta)|
    +
    |\widehat{\nu_\gamma}(\xi)-\widehat{\nu_\gamma}(\eta)|
    \le
    \left(A+\frac14\right)2^{-sn/2}
    \le
    (A+1)2^{-sn/2}.
\]
Taking the supremum over \(\xi\in\mathcal A_n\) completes the proof.
\end{proof}

For each \(n\ge1\), fix once and for all a set \(\mathcal N_n\) satisfying
the covering and cardinality conclusions of
Lemma~\ref{L:arc-annular-net-grid-passage}.

\begin{definition}
\label{D:arc-annular-exceptional-events}
Define the capped-martingale exceptional event
\(\mathcal E_n^{\mathrm{cap}}\) to be the event that there exist
\(\xi\in\mathcal N_n\),
\(d\in\mathfrak D_{\xi,n}^{\mathrm{osc}}\),
such that \(\mathcal G_n^{\mathrm{pre}}\) holds and
\[
    |F_{\xi,d,n}^{\mathrm{pre,cap}}|
    +
    \sup_{L>m_{\xi,d,+}}
    |F_{\xi,d,n,L}^{\mathrm{post,cap}}|
    >
    2^{-sn/2}2^{-2\varepsilon n}.
\]
Define the compensator exceptional event \(\mathcal E_n^{\mathrm{comp}}\) by
\[
\begin{aligned}
    \mathcal E_n^{\mathrm{comp}}
    =
    \Bigg\{
    \sup_{\xi\in\mathcal A_n}
    \Bigg(
        \sum_{d\in\mathfrak D_{\xi,n}^{\mathrm{osc}}}
        \mathfrak C_{\xi,d,n}^{\mathrm{pre}}
        +
        \sup_{L\ge1}
        \sum_{d\in\mathfrak D_{\xi,n}^{\mathrm{osc}}}
        \mathfrak C_{\xi,d,n,L}^{\mathrm{post}}
    \Bigg)
    >
    2n^{-2}2^{-sn/2}
    \Bigg\}.
\end{aligned}
\]
\end{definition}

\begin{lemma}
\label{L:arc-annular-exceptional-probabilities}
There exist constants \(C_\gamma<\infty\), \(c_\gamma>0\), and \(\eta>0\)
such that, for every integer \(n\ge1\),
\[
    \mathbb P(\mathcal E_n^{\mathrm{cap}})
    \le
    C_\gamma\exp(-c_\gamma2^{\eta n}),
\qquad
\text{and}
\qquad
    \mathbb P(\mathcal E_n^{\mathrm{comp}})
    \le
    C_\gamma2^{-c_\gamma n}.
\]
\end{lemma}

\begin{proof}
By Proposition~\ref{P:arc-capped-freedman}, after relabeling its
stretched-exponential exponent as \(\eta_0>0\), for every
\(\xi\in\mathcal N_n\) and every
\(d\in\mathfrak D_{\xi,n}^{\mathrm{osc}}\),
\[
    \mathbb P
    \left(
        \mathcal G_n^{\mathrm{pre}}
        \cap
        \left\{
            |F_{\xi,d,n}^{\mathrm{pre,cap}}|
            +
            \sup_{L>m_{\xi,d,+}}
            |F_{\xi,d,n,L}^{\mathrm{post,cap}}|
            >
            2^{-sn/2}2^{-2\varepsilon n}
        \right\}
    \right)
    \le
    C_\gamma\exp(-c_\gamma2^{\eta_0 n}).
\]
Moreover, Lemma~\ref{L:arc-annular-net-grid-passage} gives
\(
    \#\mathcal N_n
    \le
    C_\gamma2^{(2+s+2\varepsilon)n},
\)
and Proposition~\ref{P:arc-endpoint-safe-phase-bins} gives
\[
    \#\mathfrak D_{\xi,n}^{\mathrm{osc}}
    \le
    \#\mathfrak D_\xi
    \le
    C_\gamma(1+n)
    \qquad(\xi\in\mathcal N_n).
\]
Therefore, by the definition of \(\mathcal E_n^{\mathrm{cap}}\) and the union
bound,
\[
    \mathbb P(\mathcal E_n^{\mathrm{cap}})
    \le
    C_\gamma(1+n)2^{(2+s+2\varepsilon)n}
    \exp(-c_\gamma2^{\eta_0 n})
    \le
    C_\gamma\exp(-c_\gamma2^{\eta n})
\]
for some \(0<\eta\le\eta_0\), after decreasing \(c_\gamma>0\), increasing
\(C_\gamma\), and absorbing the polynomial and exponential-in-\(n\) prefactor
into the stretched-exponential decay.

For the compensator event, observe that the event
\(\mathcal E_n^{\mathrm{comp}}\) is exactly the event controlled by
Proposition~\ref{P:arc-r-tail-compensator}.  Hence
\[
    \mathbb P(\mathcal E_n^{\mathrm{comp}})
    \le
    C_\gamma2^{-c_\gamma n}.
\]
This proves both estimates.
\end{proof}

\begin{proposition}[Annular grid assembly]
\label{P:arc-annular-grid-assembly}
There exist constants \(C_\gamma<\infty\), \(c_\gamma>0\), and \(\eta>0\)
such that, for every integer \(n\ge1\),
\[
    \mathbb P
    \left(
        \sup_{\xi\in\mathcal A_n}
        |\widehat{\nu_\gamma}(\xi)|
        >
        C_\gamma2^{-sn/2}
    \right)
    \le
    C_\gamma\exp(-c_\gamma2^{\eta n})
    +
    C_\gamma2^{-c_\gamma n}.
\]
\end{proposition}

\begin{proof}
By Proposition~\ref{P:arc-local-mass-good-events} and
Lemma~\ref{L:arc-annular-exceptional-probabilities}, after adjusting constants,
\[
    \mathbb P
    \left(
        \mathcal G_n^c
        \cup
        \mathcal E_n^{\mathrm{cap}}
        \cup
        \mathcal E_n^{\mathrm{comp}}
    \right)
    \le
    C_\gamma\exp(-c_\gamma2^{\eta n})
    +
    C_\gamma2^{-c_\gamma n}.
\]
It remains to prove that, on the probability-one weak-convergence event and
outside
\(
    \mathcal G_n^c
    \cup
    \mathcal E_n^{\mathrm{cap}}
    \cup
    \mathcal E_n^{\mathrm{comp}},
\)
one has
\[
    \sup_{\xi\in\mathcal A_n}
    |\widehat{\nu_\gamma}(\xi)|
    \le
    C_\gamma2^{-sn/2}.
\]

Outside this exceptional event, \(\mathcal G_n\) holds.  In particular,
\(\mathcal G_n^{\mathrm{pre}}\)
and
\(\mathcal G_n^{\mathrm{lim}}\)
both hold.

We first prove the corresponding grid estimate.  Fix \(\xi\in\mathcal N_n\).
Since \(\mathcal N_n\subset\mathcal A_n\),
Proposition~\ref{P:arc-deterministic-annular-reduction} applies and gives
\[
    \widehat{\nu_\gamma}(\xi)
    =
    E_{\xi,n}^{\mathrm{safe}}
    +
    \sum_{d\in\mathfrak D_{\xi,n}^{\mathrm{osc}}}
    \left(
        F_{\xi,d}^{\mathrm{pre}}
        +
        F_{\xi,d}^{\mathrm{post}}
    \right),
\]
with
\(
    |E_{\xi,n}^{\mathrm{safe}}|
    \le
    C_\gamma2^{-sn/2}2^{-c_\gamma n}.
\)

For each \(d\in\mathfrak D_{\xi,n}^{\mathrm{osc}}\),
Lemma~\ref{L:arc-compensator-array-identities} applies on
\(\mathcal G_n^{\mathrm{pre}}\) and yields
\[
    F_{\xi,d}^{\mathrm{pre}}
    =
    F_{\xi,d,n}^{\mathrm{pre,cap}}
    +
    D_{\xi,d,n}^{\mathrm{pre}}.
\]
Moreover, since \((\mathcal E_n^{\mathrm{comp}})^c\) holds,
\[
    \sum_{d\in\mathfrak D_{\xi,n}^{\mathrm{osc}}}
    \mathfrak C_{\xi,d,n}^{\mathrm{pre}}
    +
    \sup_{L\ge1}
    \sum_{d\in\mathfrak D_{\xi,n}^{\mathrm{osc}}}
    \mathfrak C_{\xi,d,n,L}^{\mathrm{post}}
    \le
    2n^{-2}2^{-sn/2}.
\]
In particular, for every \(d\in\mathfrak D_{\xi,n}^{\mathrm{osc}}\), the
post-bin compensator is absolutely convergent.  By
Lemma~\ref{L:arc-exact-martingale-array-identity},
\(
    F_{\xi,d,L}^{\mathrm{post}}\to F_{\xi,d}^{\mathrm{post}}
\)
on the weak-convergence event.  Hence
Lemma~\ref{L:arc-compensator-array-identities} also gives
\[
    F_{\xi,d}^{\mathrm{post}}
    =
    F_{\xi,d,n}^{\mathrm{post,cap}}
    +
    D_{\xi,d,n}^{\mathrm{post}},
\]
with
\[
    |F_{\xi,d,n}^{\mathrm{post,cap}}|
    \le
    \sup_{L>m_{\xi,d,+}}
    |F_{\xi,d,n,L}^{\mathrm{post,cap}}|.
\]

Since \(\mathcal G_n^{\mathrm{pre}}\) holds and
\(\mathcal E_n^{\mathrm{cap}}\) does not occur, the definition of
\(\mathcal E_n^{\mathrm{cap}}\) gives, for every
\(d\in\mathfrak D_{\xi,n}^{\mathrm{osc}}\),
\[
    |F_{\xi,d,n}^{\mathrm{pre,cap}}|
    +
    \sup_{L>m_{\xi,d,+}}
    |F_{\xi,d,n,L}^{\mathrm{post,cap}}|
    \le
    2^{-sn/2}2^{-2\varepsilon n}.
\]
Also, by Proposition~\ref{P:arc-endpoint-safe-phase-bins},
\(
    \#\mathfrak D_{\xi,n}^{\mathrm{osc}}
    \le
    \#\mathfrak D_\xi
    \le
    C_\gamma(1+n).
\)
Therefore,
\[
    \sum_{d\in\mathfrak D_{\xi,n}^{\mathrm{osc}}}
    \left(
        |F_{\xi,d,n}^{\mathrm{pre,cap}}|
        +
        |F_{\xi,d,n}^{\mathrm{post,cap}}|
    \right)
    \le
    C_\gamma(1+n)2^{-sn/2}2^{-2\varepsilon n}
    \le
    C_\gamma2^{-sn/2}.
\]

The compensator contribution is bounded by
\[
    \sum_{d\in\mathfrak D_{\xi,n}^{\mathrm{osc}}}
    \left(
        |D_{\xi,d,n}^{\mathrm{pre}}|
        +
        |D_{\xi,d,n}^{\mathrm{post}}|
    \right)
    \le
    \sum_{d\in\mathfrak D_{\xi,n}^{\mathrm{osc}}}
    \mathfrak C_{\xi,d,n}^{\mathrm{pre}}
    +
    \sup_{L\ge1}
    \sum_{d\in\mathfrak D_{\xi,n}^{\mathrm{osc}}}
    \mathfrak C_{\xi,d,n,L}^{\mathrm{post}}
    \le
    2n^{-2}2^{-sn/2}.
\]
Combining the safe term, the centered capped contribution, and the compensator
contribution, we obtain
\(
    |\widehat{\nu_\gamma}(\xi)|
    \le
    C_\gamma2^{-sn/2}.
\)
Taking the maximum over \(\xi\in\mathcal N_n\) gives
\[
    \max_{\xi\in\mathcal N_n}
    |\widehat{\nu_\gamma}(\xi)|
    \le
    C_\gamma2^{-sn/2}.
\]
Finally, since \(\mathcal G_n^{\mathrm{lim}}\) holds,
Lemma~\ref{L:arc-annular-net-grid-passage} upgrades this grid estimate to the
whole annulus:
\[
    \sup_{\xi\in\mathcal A_n}
    |\widehat{\nu_\gamma}(\xi)|
    \le
    C_\gamma2^{-sn/2},
\]
after absorbing the additional constant into \(C_\gamma\).  This proves the
deterministic implication outside the exceptional event.  Together with the
probability bound for the exceptional event, the proposition follows.
\end{proof}

\begin{proof}[Proof of Theorem~\ref{T:arc-finite-r-annular}]
The annular probability estimate asserted in the theorem follows from
Proposition~\ref{P:arc-annular-grid-assembly}.  Its right-hand side is
summable in \(n\).  Hence the Borel--Cantelli lemma gives, almost surely, a
finite random integer \(N\ge1\) such that, for every \(n\ge N\),
\[
    \sup_{\xi\in\mathcal A_n}
    |\widehat{\nu_\gamma}(\xi)|
    \le
    C_\gamma2^{-sn/2}.
\]
Now let \(|\xi|\ge2^N\), and choose \(n\ge N\) such that
\(
    2^n\le|\xi|\le2^{n+1}.
\)
Then
\[
    |\widehat{\nu_\gamma}(\xi)|
    \le
    C_\gamma2^{-sn/2}
    \le
    C_\gamma2^{s/2}|\xi|^{-s/2}.
\]
Thus
\(
    |\widehat{\nu_\gamma}(\xi)|
    =
    O(|\xi|^{-s/2})
\)
for 
\(|\xi|\to\infty\)
almost surely.  This completes the proof.
\end{proof}

\section{Local dimension and the deterministic upper obstruction}
\label{S:local-dimension-upper-obstruction}

This section contains the two inputs needed for the endpoint upper bound.  First,
we transfer the imported scalar circle local-dimension theorem to the interval
cascade and then to a fixed arc.  Second, we prove a deterministic obstruction:
a nonzero finite measure supported on a fixed \(C^2\) embedded arc cannot have
Fourier dimension larger than its minimum lower local dimension.

\subsection{Endpoint local dimension by gluing}
\label{S:endpoint-gluing}

The only imported input in this subsection is
Theorem~\ref{T:imported-circle-local-dimension}; the remaining arguments are
deterministic.

\begin{lemma}
\label{L:arc-endpoint-gluing-local-dimension}
Let \(\eta\) be a nonzero finite Borel measure on \([0,1]\). Then
\[
    \alpha_{\min}^{\mathbb T}(q_{\#}\eta)
    =
    \alpha_{\min}^{[0,1]}(\eta).
\]
\end{lemma}

\begin{proof}
Set
\(
    \zeta=q_{\#}\eta.
\)
For \(s,t\in[0,1]\),
\[
    d_{\mathbb T}(q(s),q(t))
    =
    \inf_{k\in\mathbb Z}|s-t+k|
    =
    \min\{|s-t|,1-|s-t|\}.
\]
Since \([0,1]\) is compact and \(q\) is continuous,
\(  \operatorname{spt}\zeta=q(\operatorname{spt}\eta).
\)

Fix
\(
    t\in\operatorname{spt}\eta\cap(0,1)
\)
and
\(
    0<\rho<\min\{t,1-t\}.
\)
Then
\(
    B_{[0,1]}(t,\rho)=(t-\rho,t+\rho).
\)
For \(s\in[0,1]\),
\[
    s\in q^{-1}\bigl(B_{\mathbb T}(q(t),\rho)\bigr)
\quad
\Longleftrightarrow
\quad
    \inf_{k\in\mathbb Z}|s-t+k|<\rho.
\]
If \(k\ge1\), then
\(
    s-t+k\ge1-t>\rho;
\)
if \(k\le-1\), then
\(
    s-t+k\le -t<-\rho.
\)
Thus only \(k=0\) can occur.  Hence
\[
    q^{-1}\bigl(B_{\mathbb T}(q(t),\rho)\bigr)
    =
    (t-\rho,t+\rho)
    =
    B_{[0,1]}(t,\rho),
\]
and therefore
\[
    \zeta\bigl(B_{\mathbb T}(q(t),\rho)\bigr)
    =
    \eta\bigl(B_{[0,1]}(t,\rho)\bigr)
\]
for all sufficiently small \(\rho\).  It follows that the lower local
dimension of \(\zeta\) at \(q(t)\), computed in \(\mathbb T\), equals the lower
local dimension of \(\eta\) at \(t\), computed in \([0,1]\).

It remains to treat the glued endpoint.  For \(0<\rho<1/2\),
\[
    q^{-1}\bigl(B_{\mathbb T}(q(0),\rho)\bigr)
    =
    [0,\rho)\cup(1-\rho,1].
\]
Set
\(A_0(\rho)=\eta([0,\rho))\),
\(A_1(\rho)=\eta((1-\rho,1])\).
Then
\[
    \zeta\bigl(B_{\mathbb T}(q(0),\rho)\bigr)
    =
    A_0(\rho)+A_1(\rho).
\]
Define
\[
    \lambda_i
    =
    \liminf_{\rho\downarrow0}
    \frac{\log_2 A_i(\rho)}{\log_2\rho},
    \qquad i=0,1,
\]
with \(\lambda_i=+\infty\) if \(A_i(\rho)=0\) for all sufficiently small
\(\rho\).  Thus \(\lambda_0\) and \(\lambda_1\) are the one-sided lower local
dimensions at \(0\) and \(1\), with the same \(+\infty\) convention.

We claim that
\[
    \liminf_{\rho\downarrow0}
    \frac{\log_2(A_0(\rho)+A_1(\rho))}{\log_2\rho}
    =
    \min\{\lambda_0,\lambda_1\}.
\]
If some \(A_i\) is eventually zero, the claim reduces to the other side; if
both are eventually zero, both sides are interpreted as \(+\infty\).  Otherwise,
by monotonicity of \(A_i\), the logarithms are defined for all sufficiently
small \(\rho\).  Since
\(
    A_0(\rho)+A_1(\rho)\ge A_i(\rho)
\)
and \(\log_2\rho<0\), we have
\[
    \frac{\log_2(A_0(\rho)+A_1(\rho))}{\log_2\rho}
    \le
    \frac{\log_2 A_i(\rho)}{\log_2\rho},
    \qquad i=0,1.
\]
Therefore the left-hand side is at most \(\min\{\lambda_0,\lambda_1\}\).

Conversely, if \(\beta<\min\{\lambda_0,\lambda_1\}\), then \(A_i(\rho)<\rho^\beta\) for \(i=0,1\) and all sufficiently small \(\rho\).
Thus
\(
    A_0(\rho)+A_1(\rho)\le2\rho^\beta,
\)
and hence
\[
    \frac{\log_2(A_0(\rho)+A_1(\rho))}{\log_2\rho}
    \ge
    \beta+\frac{1}{\log_2\rho}.
\]
Letting \(\rho\downarrow0\) gives the lower bound by \(\beta\), and then
letting \(\beta\uparrow\min\{\lambda_0,\lambda_1\}\) gives the reverse
inequality.  Therefore the lower local dimension of \(\zeta\) at the glued
point \(q(0)=q(1)\) is \(\min\{\lambda_0,\lambda_1\}\).

All interior lower local dimensions are preserved, and the glued endpoint
contributes precisely the minimum of the two one-sided endpoint exponents.
Taking the infimum over support points gives
\[
    \alpha_{\min}^{\mathbb T}(q_{\#}\eta)
    =
    \alpha_{\min}^{[0,1]}(\eta).
\]
\end{proof}

\begin{lemma}
\label{L:arc-standard-circle-transfer}
Let \(\zeta\) be a nonzero finite Borel measure on \(\mathbb T\), and set
\(
    \eta_{\circ}=(f_{\mathbb T})_{\#}\zeta.
\)
Then
\[
    \alpha_{\min}(\eta_{\circ})
    =
    \alpha_{\min}^{\mathbb T}(\zeta).
\]
\end{lemma}

\begin{proof}
For \(\theta,\theta'\in\mathbb T\), set
\(
    \delta=d_{\mathbb T}(\theta,\theta')\in[0,1/2].
\)
Choose representatives \(a,b\in\mathbb R\) and \(k\in\mathbb Z\) with
\(
    |a-b+k|=\delta.
\)
Then
\[
    |f_{\mathbb T}(\theta)-f_{\mathbb T}(\theta')|
    =
    2|\sin(\pi(a-b))|
    =
    2|\sin(\pi(a-b+k))|
    =
    2\sin(\pi\delta).
\]
Since \(\frac{2}{\pi}x\le \sin x\le x\) for \(0\le x\le\pi/2\), we get
\[
    4d_{\mathbb T}(\theta,\theta')
    \le
    |f_{\mathbb T}(\theta)-f_{\mathbb T}(\theta')|
    \le
    2\pi d_{\mathbb T}(\theta,\theta').
\]
Thus \(f_{\mathbb T}\) is a bi-Lipschitz homeomorphism from \(\mathbb T\) onto
\(\mathbb S^1\).  In particular,
\(
    \operatorname{spt}\eta_{\circ}
    =
    f_{\mathbb T}(\operatorname{spt}\zeta).
\)

Fix \(\theta\in\operatorname{spt}\zeta\), and set
\(
    y=f_{\mathbb T}(\theta).
\)
For all sufficiently small \(\rho>0\),
\[
    B_{\mathbb T}\left(\theta,\frac{\rho}{2\pi}\right)
    \subset
    f_{\mathbb T}^{-1}\bigl(B(y,\rho)\bigr)
    \subset
    B_{\mathbb T}\left(\theta,\frac{\rho}{4}\right),
\]
where \(B(y,\rho)\) denotes the Euclidean ball in \(\mathbb R^2\).  Hence
\[
    \zeta\left(
        B_{\mathbb T}\left(\theta,\frac{\rho}{2\pi}\right)
    \right)
    \le
    \eta_{\circ}(B(y,\rho))
    \le
    \zeta\left(
        B_{\mathbb T}\left(\theta,\frac{\rho}{4}\right)
    \right).
\]
Taking logarithms, dividing by \(\log_2\rho<0\), and using the invariance of
lower local exponents under multiplication of the radius by a fixed positive
constant, the lower local dimension of \(\eta_{\circ}\) at \(y\) equals the
lower local dimension of \(\zeta\) at \(\theta\).  Taking the infimum over
support points gives
\(
    \alpha_{\min}(\eta_{\circ})
    =
    \alpha_{\min}^{\mathbb T}(\zeta).
\)
\end{proof}

\begin{lemma}
\label{L:arc-interval-to-circle-cascade}
Let \(W\ge0\) satisfy \(\mathbb E W=1\).  Let
\(
    (\widetilde\mu_n)_{n\ge0}
\)
be the dyadic scalar cascade on \([0,1]\) generated by the dyadic weights \((W_v)\), 
and suppose that
\(
    \widetilde\mu_n\weak\widetilde\mu.
\)
Set
\[
    g=f_{\mathbb T}\circ q,
    \qquad
    \mu_{\circ}=g_{\#}\widetilde\mu.
\]
Then \(\mu_{\circ}\) is the weak limit of the scalar dyadic circle cascade
constructed from the same weights, with level-\(n\) dyadic arcs
\[
    A_v=g(I_v)\subset\mathbb S^1,
    \qquad v\in\Sigma_n,
\]
using the endpoint convention inherited from the interval construction.  In
particular, \(\mu_{\circ}\) has the law of the scalar dyadic circle cascade
generated by \(W\).  Moreover,
\(
\mu_{\circ}(\mathbb S^1)=\widetilde\mu([0,1]).
\)
\end{lemma}

\begin{proof}
Set
\(
    \mu_{\circ,n}=g_{\#}\widetilde\mu_n.
\)
For \(v\in\Sigma_n\), let \(Q_v\) be the product of the weights along \(v\).
Then
\[
    \widetilde\mu_n(I_v)=2^{-n}Q_v.
\]
Since \(\widetilde\mu_n\) is absolutely continuous and
\(g^{-1}(A_v)\triangle I_v\) is contained in finitely many dyadic endpoints,
\[
    \mu_{\circ,n}(A_v)
    =
    \widetilde\mu_n(g^{-1}(A_v))
    =
    \widetilde\mu_n(I_v)
    =
    2^{-n}Q_v.
\]
Thus \((\mu_{\circ,n})_{n\ge0}\) satisfies the finite-level cylinder-mass rule
for the scalar dyadic circle cascade constructed from the same weights.

Since \(g\) is continuous and
\(
    \widetilde\mu_n\weak\widetilde\mu,
\)
we have
\[
    \mu_{\circ,n}=g_{\#}\widetilde\mu_n
    \weak
    g_{\#}\widetilde\mu
    =
    \mu_{\circ}.
\]
Therefore \(\mu_{\circ}\) is the weak limit of that circle cascade, and hence
has the law of the scalar dyadic circle cascade generated by \(W\).  Finally,
\[
    \mu_{\circ}(\mathbb S^1)
    =
    \widetilde\mu(g^{-1}(\mathbb S^1))
    =
    \widetilde\mu([0,1]).
\]
\end{proof}

\begin{proposition}
\label{P:arc-interval-local-dimension-identity}
Let \(W\) be in the minimal Kahane--Peyri\`ere regime, and let
\(\widetilde\mu\) be the dyadic scalar cascade on \([0,1]\) generated by
\(W\).  Then, almost surely on
\(
    \{\widetilde\mu([0,1])>0\},
\)
one has
\[
    \alpha_{\min}^{[0,1]}(\widetilde\mu)
    =
    A_{\mathrm{loc}}(W).
\]
\end{proposition}

\begin{proof}
Set
\[
    S_0=\{\widetilde\mu([0,1])>0\},
    \qquad
    \mu_{\circ}=(f_{\mathbb T})_{\#}q_{\#}\widetilde\mu .
\]
By Lemma~\ref{L:arc-interval-to-circle-cascade},
\(\mu_{\circ}\) has the law of the scalar dyadic circle cascade generated by
\(W\), and
\(
    \{\mu_{\circ}(\mathbb S^1)>0\}=S_0.
\)
Hence Theorem~\ref{T:imported-circle-local-dimension} gives
\(
    \alpha_{\min}(\mu_{\circ})
    =
    A_{\mathrm{loc}}(W)
\)
almost surely on \(S_0\).  On \(S_0\), the measures
\(\widetilde\mu\),
\(q_{\#}\widetilde\mu\),
\(\mu_{\circ}\)
are nonzero.  
Thus Lemma~\ref{L:arc-standard-circle-transfer} and
Lemma~\ref{L:arc-endpoint-gluing-local-dimension} apply.
Therefore, almost
surely on \(S_0\),
\[
    \alpha_{\min}^{[0,1]}(\widetilde\mu)
    =
    \alpha_{\min}^{\mathbb T}(q_{\#}\widetilde\mu)
    =
    \alpha_{\min}(\mu_{\circ})
    =
    A_{\mathrm{loc}}(W).
\]
\end{proof}

\begin{proposition}
\label{P:arc-bilipschitz-local-dimension-transfer}
Let \(\gamma:[0,1]\to\mathbb R^2\) be a \(C^1\) embedded arc with
\(
    \inf_{0\le t\le1}|\gamma'(t)|>0.
\)
If \(\eta\) is a nonzero finite Borel measure on \([0,1]\), then
\[
    \alpha_{\min}(\gamma_{\#}\eta)
    =
    \alpha_{\min}^{[0,1]}(\eta).
\]
\end{proposition}

\begin{proof}
The map
\(
    \gamma:[0,1]\to\gamma([0,1])
\)
is bi-Lipschitz.  Hence there are constants
\(
    0<c_\gamma<C_\gamma<\infty
\)
such that
\[
    c_\gamma|s-t|
    \le
    |\gamma(s)-\gamma(t)|
    \le
    C_\gamma|s-t|
    \qquad (0\le s,t\le1).
\]
For \(t\in\operatorname{spt}\eta\) and all sufficiently small \(\rho>0\),
\[
    [0,1]\cap
    \left(t-\frac{\rho}{C_\gamma},t+\frac{\rho}{C_\gamma}\right)
    \subset
    \gamma^{-1}(B(\gamma(t),\rho))
    \subset
    [0,1]\cap
    \left(t-\frac{\rho}{c_\gamma},t+\frac{\rho}{c_\gamma}\right).
\]
Applying \(\eta\) to these inclusions and using that fixed multiplicative
changes of radius do not affect lower local exponents, the lower local
dimension of \(\gamma_{\#}\eta\) at \(\gamma(t)\) equals the lower local
dimension of \(\eta\) at \(t\).  Since \(\gamma\) is a homeomorphism from
\([0,1]\) onto its image,
\(
    \operatorname{spt}(\gamma_{\#}\eta)
    =
    \gamma(\operatorname{spt}\eta).
\)
Taking the infimum over support points gives the result.
\end{proof}

\begin{theorem}[Endpoint local dimension on a fixed arc]
\label{T:arc-local-dimension-identity}
Let \(W\) be in the minimal Kahane--Peyri\`ere regime, and let
\(\widetilde\mu\) be the dyadic scalar cascade on \([0,1]\) generated by \(W\).  

Let
\(
    \gamma:[0,1]\to\mathbb R^2
\)
be a fixed nondegenerate \(C^2\) embedded arc, and set
\(
    \mu_\gamma=\gamma_{\#}\widetilde\mu.
\)
Then, almost surely on
\(
    S_\gamma=\{\widetilde\mu([0,1])>0\},
\)
one has
\[
    \alpha_{\min}(\mu_\gamma)
    =
    A_{\mathrm{loc}}(W).
\]
\end{theorem}

\begin{proof}
On \(S_\gamma\), the measure \(\widetilde\mu\) is nonzero.  By
Proposition~\ref{P:arc-interval-local-dimension-identity},
\[
    \alpha_{\min}^{[0,1]}(\widetilde\mu)
    =
    A_{\mathrm{loc}}(W)
\]
almost surely on \(S_\gamma\).  Since a fixed nondegenerate \(C^2\) embedded
arc satisfies the hypotheses of
Proposition~\ref{P:arc-bilipschitz-local-dimension-transfer}, that proposition
applies on \(S_\gamma\) with \(\eta=\widetilde\mu\).  Hence, almost surely on
\(S_\gamma\),
\[
    \alpha_{\min}(\mu_\gamma)
    =
    \alpha_{\min}(\gamma_{\#}\widetilde\mu)
    =
    \alpha_{\min}^{[0,1]}(\widetilde\mu)
    =
    A_{\mathrm{loc}}(W).
\]
\end{proof}

\subsection{The deterministic arc-support upper obstruction}
\label{S:upper-obstruction}

We next prove the deterministic upper obstruction used in the endpoint theorem.
The argument is independent of the cascade.  It applies to any nonzero finite
Borel measure supported on a fixed \(C^2\) embedded arc.

The idea is that a \(C^2\) arc has quadratic contact with its tangent line.
Equivalently, in a normal direction at a point, the support has thickness
\(O(\rho^2)\) inside a ball of radius \(\rho\).  A one-dimensional Gaussian
Fourier average along that normal frequency line therefore detects local mass.

Throughout this subsection,
\(
    \gamma:[0,1]\to\mathbb R^2
\)
is a fixed \(C^2\) embedded arc satisfying
\(
\inf_{0\le t\le1}|\gamma'(t)|>0.
\)
The nonvanishing curvature assumption
\[
    \inf_{0\le t\le1}|\det(\gamma'(t),\gamma''(t))|>0
\]
is not needed for this upper obstruction; it is used only in the finite-\(r\)
lower-bound theorem.

We use the standard comparison
\(
    \dim_{\mathrm F}(\eta)
    \le
    \dim_{\mathrm H}(\operatorname{spt}\eta)
\)
for nonzero finite Borel measures on \(\mathbb R^2\), which follows from the
Fourier-energy identity; see \cite{Mattila1995}.  In particular, if \(\eta\) is
supported on a Lipschitz arc, then
\(
    \dim_{\mathrm F}(\eta)\le1.
\)

\begin{lemma}
\label{L:arc-uniform-quadratic-normal-projection}
There exists \(C_\gamma<\infty\) such that the following holds.  For every
\(t_0\in[0,1]\), put
\(
    x_0=\gamma(t_0),
\)
and let \(N_0\) be any unit vector with
\(
    N_0\cdot\gamma'(t_0)=0.
\)
If \(t,u\in[0,1]\),
\(x=\gamma(t)\),
\(y=\gamma(u)\),
then
\[
    |N_0\cdot(x-y)|
    \le
    C_\gamma
    \left(
        |x-x_0|^2+|y-x_0|^2
    \right).
\]
Consequently, if \(\rho>0\) and
\(
    x,y\in\gamma([0,1])\cap B(x_0,\rho),
\)
then
\[
    |N_0\cdot(x-y)|
    \le
    C_\gamma\rho^2.
\]
\end{lemma}

\begin{proof}
By the bi-Lipschitz fact recorded in Subsection~\ref{SS:fixed-curve-consequences}, and since
\(\gamma\in C^2\) on the compact interval \([0,1]\), there is
\(C_\gamma<\infty\) such that
\[
    |t-t_0|
    \le
    C_\gamma|\gamma(t)-\gamma(t_0)|,
\qquad
\text{and}
\qquad
    |\gamma'(v)-\gamma'(t_0)|
    \le
    C_\gamma|v-t_0|
\]
for all \(t,t_0,v\in[0,1]\).

For every \(t,t_0\in[0,1]\), Taylor's formula in integral form gives
\[
    \gamma(t)-\gamma(t_0)
    =
    \gamma'(t_0)(t-t_0)
    +
    \int_{t_0}^{t}
    \bigl(\gamma'(v)-\gamma'(t_0)\bigr)\,dv,
\]
where the integral is understood in the oriented sense if \(t<t_0\).  Using
the preceding bound on \(\gamma'\), we get
\[
    \left|
        \int_{t_0}^{t}
        \bigl(\gamma'(v)-\gamma'(t_0)\bigr)\,dv
    \right|
    \le
    C_\gamma
    \int_{\min\{t,t_0\}}^{\max\{t,t_0\}}
    |v-t_0|\,dv
    \le
    C_\gamma|t-t_0|^2.
\]
Taking the dot product with \(N_0\) and using
\(
    N_0\cdot\gamma'(t_0)=0,
\)
we obtain
\[
    |N_0\cdot(\gamma(t)-\gamma(t_0))|
    \le
    C_\gamma|t-t_0|^2
    \le
    C_\gamma|\gamma(t)-\gamma(t_0)|^2.
\]
The same estimate with \(u\) in place of \(t\) gives
\[
\begin{aligned}
    |N_0\cdot(\gamma(t)-\gamma(u))|
    &\le
    |N_0\cdot(\gamma(t)-\gamma(t_0))|
    +
    |N_0\cdot(\gamma(u)-\gamma(t_0))|        \\
    &\le
    C_\gamma
    \left(
        |\gamma(t)-\gamma(t_0)|^2
        +
        |\gamma(u)-\gamma(t_0)|^2
    \right).
\end{aligned}
\]
Since
\(x=\gamma(t)\),
\(y=\gamma(u)\),
\(x_0=\gamma(t_0)\),
this proves the first assertion.

If
\(
    x,y\in\gamma([0,1])\cap B(x_0,\rho),
\)
then
\[
    |x-x_0|^2+|y-x_0|^2\le2\rho^2,
\]
and the final estimate follows after enlarging \(C_\gamma\).  The endpoint
cases \(t_0=0\) and \(t_0=1\) are included, since the same oriented-integral
identity and estimate remain valid on the compact interval.
\end{proof}

\begin{lemma}
\label{L:arc-gaussian-normal-frequency-lower-bound}
Let \(\eta\) be a nonzero finite Borel measure supported on \(\gamma([0,1])\).
There exists a constant \(c_\gamma>0\) such that, for every
\(
    x_0\in\operatorname{spt}\eta,
\)
every \(t_0\in[0,1]\) with
\(
    x_0=\gamma(t_0),
\)
every unit vector \(N_0\) with
\(
    N_0\cdot\gamma'(t_0)=0,
\)
and every \(0<\rho\le1\), with
\(
    \Lambda=\rho^{-2},
\)
one has
\[
    \int_{\mathbb R}
    |\widehat\eta(RN_0)|^2
    e^{-\pi(R/\Lambda)^2}\,dR
    \ge
    c_\gamma\Lambda\,\eta(B(x_0,c_\gamma\rho))^2.
\]
\end{lemma}

\begin{proof}
Let \(C_0<\infty\) be the constant in
Lemma~\ref{L:arc-uniform-quadratic-normal-projection}.  Choose
\(0<c_\gamma\le1/2\) such that
\(
    \pi C_0^2c_\gamma^4\le(\log_2 e)^{-1}.
\)
Since \(\eta\) is finite and the Gaussian is integrable, 
Fubini's theorem and the Fourier transform of the Gaussian give
\[
\begin{aligned}
    &\int_{\mathbb R}
    |\widehat\eta(RN_0)|^2
    e^{-\pi(R/\Lambda)^2}\,dR       \\
    &\quad=
    \iint
    \left(
        \int_{\mathbb R}
        e^{-2\pi i R N_0\cdot(x-y)}
        e^{-\pi(R/\Lambda)^2}\,dR
    \right)\,d\eta(x)\,d\eta(y)       \\
    &\quad=
    \Lambda
    \iint
    e^{-\pi\Lambda^2(N_0\cdot(x-y))^2}
    \,d\eta(x)\,d\eta(y).
\end{aligned}
\]
Set
\(
    E_\rho=\gamma([0,1])\cap B(x_0,c_\gamma\rho).
\)
If \(x,y\in E_\rho\), then
Lemma~\ref{L:arc-uniform-quadratic-normal-projection} gives
\(
    |N_0\cdot(x-y)|
    \le
    C_0c_\gamma^2\rho^2.
\)
Since
\(
    \Lambda=\rho^{-2},
\)
we have
\[
    \pi\Lambda^2(N_0\cdot(x-y))^2
    \le
    \pi C_0^2c_\gamma^4
    \le
    (\log_2 e)^{-1}.
\]
Hence
\[
    e^{-\pi\Lambda^2(N_0\cdot(x-y))^2}
    \ge
    \frac12
    \ge
    c_\gamma
    \qquad (x,y\in E_\rho).
\]
Restricting the nonnegative double integral to \(E_\rho\times E_\rho\) yields
\[
    \int_{\mathbb R}
    |\widehat\eta(RN_0)|^2
    e^{-\pi(R/\Lambda)^2}\,dR
    \ge
    c_\gamma\Lambda\,\eta(E_\rho)^2.
\]
Since \(\eta\) is supported on \(\gamma([0,1])\),
\(
    \eta(E_\rho)
    =
    \eta(B(x_0,c_\gamma\rho)).
\)
This proves the desired estimate.
\end{proof}

\begin{lemma}
\label{L:arc-local-mass-from-fourier-decay}
Let \(\eta\) be a nonzero finite Borel measure supported on \(\gamma([0,1])\).
Suppose that \(\sigma>0\) is an admissible Fourier decay exponent for
\(\eta\).  Then, for every
\(
    0<\tau<\min\{\sigma,1\},
\)
there exists \(C_{\tau,\gamma}<\infty\) such that, for every
\(x_0\in\operatorname{spt}\eta\) and all sufficiently small \(\rho>0\),
\[
    \eta(B(x_0,\rho))
    \le
    C_{\tau,\gamma}\rho^\tau.
\]
Consequently,
\(
    \alpha_{\min}(\eta)\ge\tau.
\)
\end{lemma}

\begin{proof}
Fix
\(
    0<\tau<\min\{\sigma,1\}.
\)
Since \(\sigma\) is an admissible Fourier decay exponent for \(\eta\), 
there exists \(A_\tau<\infty\) such that
\[
    |\widehat\eta(\xi)|^2
    \le
    A_\tau(1+|\xi|)^{-\tau}
    \qquad(\xi\in\mathbb R^2).
\]
Thus, for every unit vector \(N\in\mathbb R^2\) and every \(\Lambda\ge1\),
\[
    \int_{\mathbb R}
    |\widehat\eta(RN)|^2
    e^{-\pi(R/\Lambda)^2}\,dR
    \le
    A_\tau
    \int_{\mathbb R}
    (1+|R|)^{-\tau}
    e^{-\pi(R/\Lambda)^2}\,dR
    \le
    C_\tau\Lambda^{1-\tau},
\]
where the last inequality uses \(0<\tau<1\).

Fix \(x_0\in\operatorname{spt}\eta\), choose \(t_0\in[0,1]\) with
\(
    x_0=\gamma(t_0),
\)
and let \(N_0\) be a unit normal at \(t_0\).  Let \(c_\gamma\) and
\(\rho_\gamma\) be as in
Lemma~\ref{L:arc-gaussian-normal-frequency-lower-bound}.  For
\(0<\rho\le\rho_\gamma\), set
\(
    \Lambda=\rho^{-2}.
\)
Since \(\rho_\gamma\le1\), we have \(\Lambda\ge1\), and hence
\[
    c_\gamma\Lambda\,\eta(B(x_0,c_\gamma\rho))^2
    \le
    C_\tau\Lambda^{1-\tau}.
\]
Because \(\Lambda=\rho^{-2}\), this gives
\(
    \eta(B(x_0,c_\gamma\rho))
    \le
    C_{\tau,\gamma}\rho^\tau.
\)
Equivalently, after replacing \(c_\gamma\rho\) by \(\rho\) and enlarging
\(C_{\tau,\gamma}\),
\(
 \eta(B(x_0,\rho))
    \le
    C_{\tau,\gamma}\rho^\tau
\)
for all sufficiently small \(\rho>0\), uniformly in
\(x_0\in\operatorname{spt}\eta\).

It remains to pass from the local mass estimate to the lower local exponent.
Since \(x_0\in\operatorname{spt}\eta\), one has
\(
    \eta(B(x_0,\rho))>0
\)
for every \(\rho>0\).  For all sufficiently small \(\rho\in(0,1)\),
\[
    \log_2\eta(B(x_0,\rho))
    \le
    \log_2 C_{\tau,\gamma}
    +
    \tau\log_2\rho.
\]
Since \(\log_2\rho<0\), division by \(\log_2\rho\) reverses the inequality:
\[
    \frac{\log_2\eta(B(x_0,\rho))}{\log_2\rho}
    \ge
    \tau+
    \frac{\log_2 C_{\tau,\gamma}}{\log_2\rho}.
\]
Letting \(\rho\downarrow0\) gives
\[
    \liminf_{\rho\downarrow0}
    \frac{\log_2\eta(B(x_0,\rho))}{\log_2\rho}
    \ge
    \tau.
\]
Taking the infimum over \(x_0\in\operatorname{spt}\eta\) proves
\(
    \alpha_{\min}(\eta)\ge\tau.
\)
\end{proof}

\begin{proposition}[Deterministic arc-support upper obstruction]
\label{P:arc-support-upper-obstruction}
Let \(\gamma:[0,1]\to\mathbb R^2\) be a fixed \(C^2\) embedded arc with
\(
\inf_{0\le t\le1}|\gamma'(t)|>0.
\)
If \(\eta\) is a nonzero finite Borel measure supported on \(\gamma([0,1])\),
then
\[
    \dim_{\mathrm F}(\eta)
    \le
    \alpha_{\min}(\eta).
\]
In particular, this conclusion applies to every fixed nondegenerate \(C^2\)
embedded arc.
\end{proposition}

\begin{proof}
Since \(\gamma([0,1])\) is a Lipschitz image of an interval and
\(\operatorname{spt}\eta\subset\gamma([0,1])\), the standard Fourier-energy
support comparison gives
\[
    \dim_{\mathrm F}(\eta)
    \le
    \dim_{\mathrm H}(\operatorname{spt}\eta)
    \le
    \dim_{\mathrm H}(\gamma([0,1]))
    \le
    1.
\]
Since \(\eta\) is finite, one has
\(
    \alpha_{\min}(\eta)\ge0.
\)
Hence the case
\(
    \dim_{\mathrm F}(\eta)=0
\)
is immediate.

Assume
\(
    \dim_{\mathrm F}(\eta)>0,
\)
and fix
\(
    0<\sigma<\dim_{\mathrm F}(\eta).
\)
By the definition of Fourier dimension, there exists an admissible Fourier
decay exponent \(\sigma'\) for \(\eta\) with
\(
    \sigma<\sigma'.
\)
Since
\(
    \dim_{\mathrm F}(\eta)\le1,
\)
we also have \(\sigma<1\).  Hence one can choose \(\tau\) such that
\(
\sigma<\tau<\min\{\sigma',1\}.
\)
Applying Lemma~\ref{L:arc-local-mass-from-fourier-decay} with the admissible
exponent \(\sigma'\) gives
\(
 \alpha_{\min}(\eta)\ge\tau>\sigma.
\)
Taking the supremum over all
\(
 0<\sigma<\dim_{\mathrm F}(\eta)
\)
gives
\(
\alpha_{\min}(\eta)\ge\dim_{\mathrm F}(\eta).
\)
This proves the proposition.
\end{proof}

\begin{remark}
\label{R:arc-upper-obstruction-endpoints}
The proof treats endpoint and interior support points uniformly.  If
\(x_0=\gamma(0)\) or \(x_0=\gamma(1)\), the Taylor estimate in
Lemma~\ref{L:arc-uniform-quadratic-normal-projection} is one-sided in the
parameter, which is sufficient because the measure is supported on the
one-sided arc near the endpoint.
\end{remark}

\section{Endpoint assembly}
\label{S:endpoint-assembly}

We now prove the endpoint formula.  The first subsection treats fixed arcs.
The second subsection derives the closed Jordan curve case by cutting the
circle cascade into two interval cascades and applying the fixed-arc theorem.

\subsection{The fixed-arc endpoint formula}
\label{SS:arc-endpoint-assembly}

Throughout this subsection, \(W\) is in the minimal Kahane--Peyri\`ere regime,
\(\widetilde\mu\) is the dyadic scalar cascade on \([0,1]\) generated by \(W\),
and, for a fixed nondegenerate \(C^2\) embedded arc
\(
    \gamma:[0,1]\to\mathbb R^2,
\)
we set
\(
    \mu_\gamma=\gamma_{\#}\widetilde\mu.
\)

Let
\(
    S_\gamma=\{\widetilde\mu([0,1])>0\}
\)
be the non-extinction event.  On \(S_\gamma\), the measure \(\mu_\gamma\) is a
nonzero finite Borel measure supported on \(\gamma([0,1])\).

\begin{lemma}
\label{L:arc-strict-subendpoint-annular-decay}
Assume \(A_{\mathrm{loc}}(W)>0\), and let
\(
    0<\sigma<A_{\mathrm{loc}}(W).
\)
Then there exist constants
\(C_{\gamma,\sigma}<\infty\),
\(c_{\gamma,\sigma}>0\),
\(\eta_{\gamma,\sigma}>0\)
such that, for every \(n\ge1\),
\[
    \mathbb P
    \left(
        \sup_{\xi\in\mathcal A_n}
        |\widehat{\mu_\gamma}(\xi)|
        >
        C_{\gamma,\sigma}2^{-\sigma n/2}
    \right)
    \le
    C_{\gamma,\sigma}\exp(-c_{\gamma,\sigma}2^{\eta_{\gamma,\sigma}n})
    +
    C_{\gamma,\sigma}2^{-c_{\gamma,\sigma}n}.
\]
Consequently,
\[
    |\widehat{\mu_\gamma}(\xi)|
    =
    O(|\xi|^{-\sigma/2})
    \qquad(|\xi|\to\infty)
\]
almost surely.
\end{lemma}

\begin{proof}
By the definition of \(A_{\mathrm{loc}}(W)\), there exists \(r>1\) such that
\[
    \mathbb E[W^r]<\infty,
    \qquad
    \frac{r-1-\log_2\mathbb E[W^r]}{r}
    >
    \sigma.
\]
Choose \(\delta>0\) such that
\[
    \sigma+\delta
    <
    \frac{r-1-\log_2\mathbb E[W^r]}{r}.
\]
Equivalently,
\(
    2^{1-r}\mathbb E[W^r]
    \le
    2^{-r(\sigma+\delta)}.
\)
Since \(0<\sigma<A_{\mathrm{loc}}(W)\le1\), we have \(0<\sigma<1\).
Therefore Theorem~\ref{T:arc-finite-r-annular} applies with
\(U=W\) and \(s=\sigma\).
The cascade generated by \(U=W\) is \(\widetilde\mu\), and its arc pushforward
is \(\mu_\gamma\).  Renaming the constants in
Theorem~\ref{T:arc-finite-r-annular} as
\(C_{\gamma,\sigma}\),
\(c_{\gamma,\sigma}\),
\(\eta_{\gamma,\sigma}\),
we obtain the stated annular probability estimate.

The right-hand side is summable in \(n\).  Hence, by the Borel--Cantelli lemma,
almost surely there exists a finite random integer \(N_\sigma\ge1\) such that,
for every \(n\ge N_\sigma\),
\[
    \sup_{\xi\in\mathcal A_n}
    |\widehat{\mu_\gamma}(\xi)|
    \le
    C_{\gamma,\sigma}2^{-\sigma n/2}.
\]
If \(2^n\le|\xi|\le2^{n+1}\), then
\(2^{-\sigma n/2}\le 2^{\sigma/2}|\xi|^{-\sigma/2}\).
Therefore, \(|\widehat{\mu_\gamma}(\xi)|=O(|\xi|^{-\sigma/2})\) as
\(|\xi|\to\infty\), almost surely.
\end{proof}

\begin{proposition}
\label{P:arc-endpoint-lower-bound}
Almost surely on \(S_\gamma\),
\[
    \dim_{\mathrm F}(\mu_\gamma)
    \ge
    A_{\mathrm{loc}}(W).
\]
Moreover, if \(A_{\mathrm{loc}}(W)>0\), then almost surely, for every
\(
    0<\sigma<A_{\mathrm{loc}}(W),
\)
there exist finite constants
\(C_\sigma<\infty\),
\(R_\sigma<\infty\)
such that
\[
    |\widehat{\mu_\gamma}(\xi)|
    \le
    C_\sigma|\xi|^{-\sigma/2}
    \qquad(|\xi|\ge R_\sigma).
\]
\end{proposition}

\begin{proof}
If \(A_{\mathrm{loc}}(W)=0\), then the lower bound follows from
\(
    \dim_{\mathrm F}(\mu_\gamma)\ge0,
\)
and the moreover statement is void.

Assume \(A_{\mathrm{loc}}(W)>0\), and set
\(
    \mathcal Q
    =
    \mathbb Q\cap(0,A_{\mathrm{loc}}(W)).
\)
We use rational exponents only to take a countable intersection of
probability-one events.

For each \(q\in\mathcal Q\), Lemma~\ref{L:arc-strict-subendpoint-annular-decay}
gives a probability-one event \(E_q\) on which
\(|\widehat{\mu_\gamma}(\xi)|=O(|\xi|^{-q/2})\) as \(|\xi|\to\infty\).
Let
\(
    E=\bigcap_{q\in\mathcal Q}E_q.
\)
Then \(\mathbb P(E)=1\).  On \(E\cap S_\gamma\), each
\(q\in\mathcal Q\) is an admissible Fourier decay exponent for \(\mu_\gamma\).
Hence
\[
    \dim_{\mathrm F}(\mu_\gamma)\ge q
    \qquad(q\in\mathcal Q).
\]
Taking the supremum over \(q\in\mathcal Q\) gives
\(
    \dim_{\mathrm F}(\mu_\gamma)
    \ge
    A_{\mathrm{loc}}(W)
\)
on \(E\cap S_\gamma\).

For the moreover statement, fix \(\omega\in E\) and
\(
    0<\sigma<A_{\mathrm{loc}}(W).
\)
Choose \(q\in\mathcal Q\) such that
\(
    \sigma<q<A_{\mathrm{loc}}(W).
\)
Since \(\omega\in E_q\), the \(q\)-decay estimate holds with some finite
constants \(C_q(\omega)\) and \(R_q(\omega)\).  Set
\[
    C_\sigma(\omega)=C_q(\omega),
    \qquad
    R_\sigma(\omega)=\max\{R_q(\omega),1\}.
\]
Since \(q>\sigma\), for \(|\xi|\ge R_\sigma(\omega)\) we obtain
\(
    |\widehat{\mu_\gamma}(\xi)|
    \le
    C_\sigma(\omega)|\xi|^{-\sigma/2}.
\)
This proves the moreover statement.
\end{proof}

\begin{proposition}
\label{P:arc-endpoint-upper-bound}
Almost surely on \(S_\gamma\),
\[
    \dim_{\mathrm F}(\mu_\gamma)
    \le
    A_{\mathrm{loc}}(W).
\]
\end{proposition}

\begin{proof}
On \(S_\gamma\), the measure \(\mu_\gamma\) is a nonzero finite Borel measure
supported on \(\gamma([0,1])\).  Hence
Proposition~\ref{P:arc-support-upper-obstruction} gives
\(
    \dim_{\mathrm F}(\mu_\gamma)
    \le
    \alpha_{\min}(\mu_\gamma).
\)
By Theorem~\ref{T:arc-local-dimension-identity},
\(
    \alpha_{\min}(\mu_\gamma)
    =
    A_{\mathrm{loc}}(W)
\)
almost surely on \(S_\gamma\).  Combining these two facts yields
\[
    \dim_{\mathrm F}(\mu_\gamma)
    \le
    \alpha_{\min}(\mu_\gamma)
    =
    A_{\mathrm{loc}}(W)
\]
almost surely on \(S_\gamma\).
\end{proof}

\begin{proof}[Proof of Theorem~\ref{T:arc-main-endpoint}]
The lower bound
\(
    \dim_{\mathrm F}(\mu_\gamma)
    \ge
    A_{\mathrm{loc}}(W),
\)
together with the strict subendpoint Fourier decay statement, is given by
Proposition~\ref{P:arc-endpoint-lower-bound}.  The upper bound
\(
    \dim_{\mathrm F}(\mu_\gamma)
    \le
    A_{\mathrm{loc}}(W)
\)
is Proposition~\ref{P:arc-endpoint-upper-bound}.  Combining the two bounds
gives
\[
    \dim_{\mathrm F}(\mu_\gamma)
    =
    A_{\mathrm{loc}}(W)
\]
almost surely on \(S_\gamma\).  This proves the theorem.
\end{proof}

\subsection{Closed Jordan curves}
\label{SS:jordan-endpoint-assembly}

We now derive the closed-curve endpoint formula from the fixed-arc theorem and
the dyadic cutting of the circle cascade.  Throughout this subsection,
\(
    \gamma:\mathbb T\to\mathbb R^2
\)
is a fixed nondegenerate \(C^2\) Jordan curve, and
\(\widetilde\mu^{\mathbb T}\) denotes the scalar dyadic cascade on
\(\mathbb T\) generated by \(W\).  We set
\[
    \mu_\gamma^{\mathbb T}
    =
    \gamma_{\#}\widetilde\mu^{\mathbb T},
\qquad
\text{and}
\qquad
    S_\gamma^{\mathbb T}
    =
    \{\widetilde\mu^{\mathbb T}(\mathbb T)>0\}.
\]

\begin{proposition}[Closed-curve local-dimension identity]
\label{P:jordan-local-dimension-identity}
Let \(W\) be in the minimal Kahane--Peyri\`ere regime.  Then, almost surely on
\(S_\gamma^{\mathbb T}\),
\[
    \alpha_{\min}(\mu_\gamma^{\mathbb T})
    =
    A_{\mathrm{loc}}(W).
\]
\end{proposition}

\begin{proof}
By Theorem~\ref{T:imported-circle-local-dimension}, applied to the circle
pushforward
\(
    (f_{\mathbb T})_\#\widetilde\mu^{\mathbb T},
\)
we have
\[
    \alpha_{\min}\bigl((f_{\mathbb T})_\#\widetilde\mu^{\mathbb T}\bigr)
    =
    A_{\mathrm{loc}}(W)
\]
almost surely on \(S_\gamma^{\mathbb T}\).  Since
\(f_{\mathbb T}:\mathbb T\to\mathbb S^1\) is bi-Lipschitz,
\[
    \alpha_{\min}^{\mathbb T}(\widetilde\mu^{\mathbb T})
    =
    \alpha_{\min}\bigl((f_{\mathbb T})_\#\widetilde\mu^{\mathbb T}\bigr).
\]
By the bi-Lipschitz observation for fixed Jordan curves,
\[
    \alpha_{\min}(\mu_\gamma^{\mathbb T})
    =
    \alpha_{\min}^{\mathbb T}(\widetilde\mu^{\mathbb T}).
\]
Combining the three identities gives
\[
    \alpha_{\min}(\mu_\gamma^{\mathbb T})
    =
    A_{\mathrm{loc}}(W).
\]
\end{proof}

\begin{proposition}[Closed-curve deterministic upper obstruction]
\label{P:jordan-support-upper-obstruction}
Let
\(
    \gamma:\mathbb T\to\mathbb R^2
\)
be a fixed \(C^2\) embedded Jordan curve with
\(
    \inf_{t\in\mathbb T}|\gamma'(t)|>0.
\)
If \(\eta\) is a nonzero finite Borel measure supported on
\(\gamma(\mathbb T)\), then
\[
    \dim_{\mathrm F}(\eta)
    \le
    \alpha_{\min}(\eta).
\]
In particular, this conclusion applies to every fixed nondegenerate \(C^2\)
Jordan curve.
\end{proposition}

\begin{proof}
Set
\(
    \Gamma=\gamma(\mathbb T).
\)
Choose finitely many open arcs \(V_j\subset\mathbb T\) and compact arcs
\(I_j\subset\mathbb T\) such that
\[
    \mathbb T=\bigcup_{j=1}^M V_j,
    \qquad
    \overline{V_j}\subset \operatorname{int}_{\mathbb T} I_j,
\]
and such that each restriction \(\gamma|_{I_j}\) is a \(C^2\) embedded arc.
Put
\(
    K_j=\mathbb T\setminus \operatorname{int}_{\mathbb T} I_j.
\)
Since \(\gamma\) is injective, the compact sets
\(\gamma(\overline{V_j})\) and \(\gamma(K_j)\) are disjoint.  Hence
\[
    \delta
    :=
    \frac12
    \min_{1\le j\le M}
    \operatorname{dist}\bigl(\gamma(\overline{V_j}),\gamma(K_j)\bigr)
    >0.
\]
It follows that, if \(x_0=\gamma(t_0)\) with \(t_0\in V_j\), then
\(
    \Gamma\cap B(x_0,\delta)\subset \gamma(I_j).
\)

For each \(j\), choose a \(C^2\) diffeomorphism
\(\theta_j:[0,1]\to I_j\) with nonvanishing derivative, and set
\(
    \gamma_j=\gamma\circ\theta_j.
\)
Applying Lemma~\ref{L:arc-uniform-quadratic-normal-projection} to the finitely
many arcs \(\gamma_j\), and taking the maximum of the corresponding constants, 
gives a constant \(C_\gamma<\infty\).  
Set
\(
    \rho_\gamma=\min\{1,\delta\}.
\)
Now let \(x_0=\gamma(t_0)\in\Gamma\), choose \(j\) with \(t_0\in V_j\), and let
\(N_0\) be a unit vector satisfying
\(
    N_0\cdot\gamma'(t_0)=0.
\)
If
\[
    x,y\in\Gamma\cap B(x_0,\rho),
    \qquad
    0<\rho\le\rho_\gamma,
\]
then \(x,y\in\gamma(I_j)\).  Since \(N_0\) is also normal to the arc
\(\gamma_j\) at \(x_0\), Lemma~\ref{L:arc-uniform-quadratic-normal-projection}
gives
\(
    |N_0\cdot(x-y)|
    \le
    C_\gamma\rho^2.
\)

With this uniform local quadratic normal projection estimate in hand, the proof
of Lemma~\ref{L:arc-gaussian-normal-frequency-lower-bound} applies verbatim.
Thus, after possibly decreasing \(\rho_\gamma\) and \(c_\gamma>0\), for every
\(x_0\in\operatorname{spt}\eta\), every unit normal \(N_0\) at \(x_0\), and
every \(0<\rho\le\rho_\gamma\), with
\(
    \Lambda=\rho^{-2},
\)
one has
\[
    \int_{\mathbb R}
    |\widehat\eta(RN_0)|^2
    e^{-\pi(R/\Lambda)^2}\,dR
    \ge
    c_\gamma\Lambda\,\eta(B(x_0,c_\gamma\rho))^2.
\]

Now let \(\sigma>0\) be an admissible Fourier decay exponent for \(\eta\), and
fix
\(
    0<\tau<\min\{\sigma,1\}.
\)
Repeating the proof of Lemma~\ref{L:arc-local-mass-from-fourier-decay}, using
the preceding Gaussian lower bound in place of the arc one, gives a constant
\(C_{\tau,\gamma}<\infty\) such that
\(
    \eta(B(x_0,\rho))
    \le
    C_{\tau,\gamma}\rho^\tau
\)
for every \(x_0\in\operatorname{spt}\eta\) and all sufficiently small
\(\rho>0\).  Hence
\(
    \alpha_{\min}(\eta)\ge\tau.
\)

Finally, since \(\eta\) is supported on the Lipschitz curve \(\Gamma\), the
standard Fourier-energy support comparison gives
\[
    \dim_{\mathrm F}(\eta)
    \le
    \dim_{\mathrm H}(\operatorname{spt}\eta)
    \le
    \dim_{\mathrm H}(\Gamma)
    \le
    1.
\]
If \(\dim_{\mathrm F}(\eta)=0\), then the desired inequality is immediate.

Otherwise, fix
\(
    0<\sigma<\dim_{\mathrm F}(\eta).
\)
By the definition of Fourier dimension, there exists an admissible Fourier
decay exponent \(\sigma'\) for \(\eta\) such that
\(
    \sigma<\sigma'.
\)
Since \(\dim_{\mathrm F}(\eta)\le1\), we also have \(\sigma<1\).  Choose
\(
    \sigma<\tau<\min\{\sigma',1\}.
\)
The preceding paragraph gives
\(
\alpha_{\min}(\eta)\ge\tau>\sigma.
\)
Taking the supremum over all
\(
    0<\sigma<\dim_{\mathrm F}(\eta)
\)
yields
\(
 \alpha_{\min}(\eta)\ge\dim_{\mathrm F}(\eta).
\)
\end{proof}

\begin{proposition}
\label{P:jordan-endpoint-lower-bound}
Almost surely on \(S_\gamma^{\mathbb T}\),
\[
    \dim_{\mathrm F}(\mu_\gamma^{\mathbb T})
    \ge
    A_{\mathrm{loc}}(W).
\]
Moreover, if \(A_{\mathrm{loc}}(W)>0\), then almost surely, for every
\(
    0<\sigma<A_{\mathrm{loc}}(W),
\)
there exist finite constants
\(C_\sigma<\infty\),
\(R_\sigma<\infty\)
such that
\[
    |\widehat{\mu_\gamma^{\mathbb T}}(\xi)|
    \le
    C_\sigma|\xi|^{-\sigma/2}
    \qquad(|\xi|\ge R_\sigma).
\]
\end{proposition}

\begin{proof}
If \(A_{\mathrm{loc}}(W)=0\), the lower bound is immediate.
Assume \(A_{\mathrm{loc}}(W)>0\).  Let
\[
    \rho_0(t)=\frac{t}{2},
    \qquad
    \rho_1(t)=\frac{1+t}{2}\pmod 1,
    \qquad 0\le t\le1.
\]
Let \(E_{\mathrm{cut}}\) be the probability-one event in
Lemma~\ref{L:circle-cascade-dyadic-cutting}.  On \(E_{\mathrm{cut}}\), we may
write
\[
    \widetilde\mu^{\mathbb T}
    =
    \frac{W_0}{2}(\rho_0)_{\#}\widetilde\mu^{(0)}
    +
    \frac{W_1}{2}(\rho_1)_{\#}\widetilde\mu^{(1)},
\]
where \(\widetilde\mu^{(0)}\) and \(\widetilde\mu^{(1)}\) are interval
cascades generated by independent copies of the same dyadic weights.  Therefore
\[
    \mu_\gamma^{\mathbb T}
    =
    \frac{W_0}{2}(\gamma\circ\rho_0)_{\#}\widetilde\mu^{(0)}
    +
    \frac{W_1}{2}(\gamma\circ\rho_1)_{\#}\widetilde\mu^{(1)}.
\]
By Lemma~\ref{L:jordan-dyadic-subarcs}, the two maps
\(\gamma_0=\gamma\circ\rho_0\),
\(\gamma_1=\gamma\circ\rho_1\)
are fixed nondegenerate \(C^2\) embedded arcs.

For \(i=0,1\), let \(E_i\) be the probability-one event given by
Proposition~\ref{P:arc-endpoint-lower-bound}, applied to the fixed arc
\(\gamma_i\) and to the descendant cascade \(\widetilde\mu^{(i)}\).  On the
common probability-one event
\(
    E_{\mathrm{cut}}\cap E_0\cap E_1,
\)
for each \(i=0,1\), every strict subendpoint exponent
\(
    0<\sigma<A_{\mathrm{loc}}(W)
\)
is an admissible Fourier decay exponent for
\(
    (\gamma_i)_{\#}\widetilde\mu^{(i)}
\)
whenever this measure is nonzero; if it is the zero measure, the corresponding
Fourier transform estimate is trivial.  Multiplication by the finite random
constants \(W_i/2\) does not affect the decay exponent.
Hence the finite sum
\[
    \mu_\gamma^{\mathbb T}
    =
    \frac{W_0}{2}(\gamma_0)_{\#}\widetilde\mu^{(0)}
    +
    \frac{W_1}{2}(\gamma_1)_{\#}\widetilde\mu^{(1)}
\]
also satisfies
\[
    |\widehat{\mu_\gamma^{\mathbb T}}(\xi)|
    =
    O(|\xi|^{-\sigma/2})
    \qquad(|\xi|\to\infty)
\]
for every
\(
    0<\sigma<A_{\mathrm{loc}}(W).
\)
On \(S_\gamma^{\mathbb T}\), the measure \(\mu_\gamma^{\mathbb T}\) is nonzero.
Therefore every strict subendpoint exponent is admissible, and
\[
    \dim_{\mathrm F}(\mu_\gamma^{\mathbb T})
    \ge
    A_{\mathrm{loc}}(W)
\]
almost surely on \(S_\gamma^{\mathbb T}\).  The same argument gives the stated
strict subendpoint Fourier decay estimate with finite random constants
\(C_\sigma\) and \(R_\sigma\).
\end{proof}

\begin{proposition}
\label{P:jordan-endpoint-upper-bound}
Almost surely on \(S_\gamma^{\mathbb T}\),
\[
    \dim_{\mathrm F}(\mu_\gamma^{\mathbb T})
    \le
    A_{\mathrm{loc}}(W).
\]
\end{proposition}

\begin{proof}
On \(S_\gamma^{\mathbb T}\), 
the measure \(\mu_\gamma^{\mathbb T}\) is a nonzero finite Borel measure supported 
on \(\gamma(\mathbb T)\).  
By Proposition~\ref{P:jordan-support-upper-obstruction},
\(
    \dim_{\mathrm F}(\mu_\gamma^{\mathbb T})
    \le
    \alpha_{\min}(\mu_\gamma^{\mathbb T}).
\)
By Proposition~\ref{P:jordan-local-dimension-identity},
\(
    \alpha_{\min}(\mu_\gamma^{\mathbb T})
    =
    A_{\mathrm{loc}}(W)
\)
almost surely on \(S_\gamma^{\mathbb T}\).  Combining the two estimates gives
\(
    \dim_{\mathrm F}(\mu_\gamma^{\mathbb T})
    \le
    A_{\mathrm{loc}}(W)
\)
almost surely on \(S_\gamma^{\mathbb T}\).
\end{proof}

\begin{proof}[Proof of Theorem~\ref{T:jordan-main-endpoint}]
The lower bound is Proposition~\ref{P:jordan-endpoint-lower-bound}, and the
upper bound is Proposition~\ref{P:jordan-endpoint-upper-bound}.  Combining them
gives
\[
    \dim_{\mathrm F}(\mu_\gamma^{\mathbb T})
    =
    A_{\mathrm{loc}}(W)
\]
almost surely on \(S_\gamma^{\mathbb T}\).  The strict subendpoint Fourier
decay statement is included in Proposition~\ref{P:jordan-endpoint-lower-bound}.
This proves the theorem.
\end{proof}

\begin{corollary}
\label{C:positive-zero-endpoint-regimes}
The conclusions of Theorems~\ref{T:arc-main-endpoint} and
\ref{T:jordan-main-endpoint} imply the following in the corresponding fixed-arc
or fixed-Jordan setting.
\begin{enumerate}[label=\textup{(\roman*)}]
\item If there exists \(q>1\) such that
\(
    \mathbb E[W^q]<2^{q-1},
\)
then the corresponding pushforward cascade has positive Fourier dimension
almost surely on non-extinction.

\item If
\(\mathbb E[W^q]=\infty\)
for every \(q>1\),
then the corresponding pushforward cascade has Fourier dimension \(0\) almost
surely on non-extinction.
\end{enumerate}
\end{corollary}

\begin{proof}
If \(\mathbb E[W^q]<2^{q-1}\) for some \(q>1\), then the corresponding term in
the definition of \(A_{\mathrm{loc}}(W)\) is strictly positive:
\[
    \frac{q-1-\log_2\mathbb E[W^q]}{q}>0.
\]
Hence
\(
    A_{\mathrm{loc}}(W)>0.
\)
The endpoint formula in either Theorem~\ref{T:arc-main-endpoint} or
Theorem~\ref{T:jordan-main-endpoint} gives positive Fourier dimension almost surely on the relevant non-extinction event.

If \(\mathbb E[W^q]=\infty\) for every \(q>1\), then every \(q\)-term in Definition~\ref{D:endpoint-local-exponent} is interpreted as zero.  Therefore
\(
    A_{\mathrm{loc}}(W)=0.
\)
The endpoint formula in either setting gives Fourier dimension \(0\) almost
surely on the relevant non-extinction event.
\end{proof}

Given the exact dimension formulas we have, it is a straightforward matter to derive the following Salem consequences. 

\begin{corollary}
\label{C:Salem-classification}
Let \(W\) be in the minimal Kahane--Peyri\`ere regime, and set
\(\chi=1-\mathbb E[W\log_2 W]\). In each of the following two settings---the fixed-arc setting of
Theorem~\ref{T:arc-main-endpoint} and the fixed-Jordan setting of
Theorem~\ref{T:jordan-main-endpoint}---the corresponding pushforward cascade
measure is Salem almost surely on the relevant non-extinction event if and
only if there exists \(a\in[1,2)\) such that
\[
    \mathbb P(W=a)=a^{-1},
    \qquad
    \mathbb P(W=0)=1-a^{-1}.
\]
Here Salem means that the Fourier dimension of the measure equals its
Hausdorff dimension.
\end{corollary}

\begin{proof}
By the standard exact-dimensionality theorem for scalar Mandelbrot cascades,
the interval and circle parameter cascades have Hausdorff dimension
\(\chi=1-\mathbb E[W\log_2 W]\) almost surely on their respective
non-extinction events; see, for example,
\cite{Kahane1985,KahanePeyriere1976}. Since the fixed curve
parametrizations are bi-Lipschitz, the corresponding pushforward measures
also have Hausdorff dimension \(\chi\). In view of
Theorems~\ref{T:arc-main-endpoint} and
\ref{T:jordan-main-endpoint}, it therefore suffices to determine when
\(A_{\mathrm{loc}}(W)=\chi\).

Let \(d\mathbb P^\star=W\,d\mathbb P\). Since \(\mathbb EW=1\),
\(\mathbb P^\star\) is a probability measure. Moreover,
\(\mathbb P^\star(W=0)=0\), and the minimal Kahane--Peyri\`ere assumptions
imply \(\log_2 W\in L^1(\mathbb P^\star)\) and
\[
    \mathbb E^\star[\log_2 W]
    =
    \mathbb E[W\log_2 W]
    =
    1-\chi.
\]
For \(q>1\) with \(\mathbb E[W^q]<\infty\), set
\(
    \Phi(q)=\log_2\mathbb E[W^q].
\)
Since \(\mathbb E[W^q]=\mathbb E^\star[W^{q-1}]\), Jensen's inequality gives
\[
    \Phi(q)
    \geq
    (q-1)\mathbb E^\star[\log_2 W]
    =
    (q-1)(1-\chi).
\]
Thus
\[
    \max\left\{
        0,\frac{q-1-\Phi(q)}{q}
    \right\}
    \leq
    \chi\frac{q-1}{q}.
\]
The terms corresponding to infinite \(q\)-moments are \(0\) by definition,
so
\(
    A_{\mathrm{loc}}(W)\leq\chi.
\)

Suppose that \(A_{\mathrm{loc}}(W)=\chi\). Since \(\chi>0\), the definition of
\(A_{\mathrm{loc}}(W)\) allows us to choose \(q_j>1\) such that
\(\mathbb E[W^{q_j}]<\infty\) and
\[
    \frac{q_j-1-\Phi(q_j)}{q_j}\longrightarrow\chi.
\]
The preceding bound implies \(q_j\to\infty\), and
\[
\begin{aligned}
    0
    \leq
    \frac{\Phi(q_j)-(q_j-1)(1-\chi)}{q_j}
    =
    \chi\frac{q_j-1}{q_j}
    -
    \frac{q_j-1-\Phi(q_j)}{q_j}
    \longrightarrow0.
\end{aligned}
\]
Hence
\(
    \frac{\Phi(q_j)}{q_j}
    \longrightarrow
    1-\chi
    =
    \mathbb E^\star[\log_2 W].
\)
On the other hand, the standard convergence of \(L^p\)-norms as
\(p\to\infty\), understood in the extended real sense, gives
\[
    \lim_{j\to\infty}\frac{\Phi(q_j)}{q_j}
    =
    \lim_{j\to\infty}
    \log_2\bigl(\mathbb E[W^{q_j}]\bigr)^{1/q_j}
    =
    \log_2\|W\|_\infty.
\]
Since the limit on the left is finite, \(a:=\|W\|_\infty<\infty\), and
\(
    \mathbb E^\star[\log_2 W]=\log_2 a.
\)
Since \(\log_2 W\leq\log_2 a\)
\(\mathbb P^\star\)-almost surely, equality of the expectations implies
\(W=a\) \(\mathbb P^\star\)-almost surely. Consequently,
\[
    0
    =
    \mathbb P^\star(W\neq a)
    =
    \mathbb E\!\left[W\mathbf 1_{\{W\neq a\}}\right],
\]
so \(W\in\{0,a\}\) almost surely. Since \(\mathbb EW=1\),
\[
    \mathbb P(W=a)=a^{-1},
    \qquad
    \mathbb P(W=0)=1-a^{-1}.
\]
Necessarily \(a\geq1\), while
\(
    \chi
    =
    1-\mathbb E[W\log_2 W]
    =
    1-\log_2 a>0
\)
gives \(a<2\).

Conversely, suppose that, for some \(a\in[1,2)\),
\[
    \mathbb P(W=a)=a^{-1},
    \qquad
    \mathbb P(W=0)=1-a^{-1}.
\]
Then, for every \(q>1\),
\[
    \mathbb E[W^q]=a^{q-1},
    \qquad
    \chi=1-\log_2 a,
\]
and hence
\[
\begin{aligned}
    A_{\mathrm{loc}}(W)
    =
    \sup_{q>1}
    \frac{(q-1)(1-\log_2 a)}{q}
    =
    1-\log_2 a
    =
    \chi.
\end{aligned}
\]
The endpoint formulas therefore show that the relevant pushforward measure
has equal Fourier and Hausdorff dimensions, and is thus Salem almost surely
on the corresponding non-extinction event.
\end{proof}

\section*{Acknowledgments}

X. F. was partially supported by the National Science
and Technology Council, Taiwan, grant no. 114-2115-M-A49-003-MY3.

The authors used an artificial-intelligence tool for language editing and
\LaTeX\ formatting; all mathematical content was written, checked, and approved
by the authors.

\end{document}